\documentclass[reqno]{amsart}

\usepackage[english]{babel}

\usepackage{amsmath,amssymb,amsthm,amsfonts,enumerate,mathtools,verbatim,comment,afterpage}

\usepackage[shortlabels]{enumitem}
    \setenumerate[0]{label={\rm (\roman*)}, leftmargin=*}
\usepackage[T1]{fontenc}
\usepackage[utf8]{inputenc}
\setlist{nolistsep}

\usepackage[textsize=footnotesize,textwidth=20ex,colorinlistoftodos]{todonotes}

\tikzset{%
	myvertex/.style = {circle, draw, fill, very thick, outer sep=3.5pt, inner sep=0pt, minimum size=#1},
	myvertex/.default = 3.5pt,
	myblob/.style = {myvertex, fill=white, minimum size=#1},
	myedge/.style = {draw=black, very thick, line cap=round, line join=round},
}

\definecolor{ugentblue}{RGB}{30,100,200}
\definecolor{ugentyellow}{RGB}{120,190,0}
\definecolor{ugentred}{RGB}{220,78,40}

\newcommand{\diagnode}[1]{\fill #1 circle (.1);}
\newcommand{\distorbit}[1]{\draw #1 circle (.2);}
\newcommand{\distlongorbit}[1]{\draw (#1 - .2, .4) -- (#1 - .2, -.4) arc (-180:0:.2) -- (#1 + .2, -.4) -- (#1 + .2, .401) arc (0:180:.2) -- cycle;}

\usepackage{hyperref}
\hypersetup{
    colorlinks=true,
    linkcolor=ugentblue,
    filecolor=magenta,
    citecolor=ugentblue,
    urlcolor=cyan,
}

\usepackage[capitalize]{cleveref}
    \crefname{enumi}{}{}
    \Crefname{enumi}{Item}{Items}
    \crefname{equation}{}{}
    \Crefname{equation}{Equation}{Equations}
    \crefname{notation}{Notation}{Notations}
    \crefname{maintheorem}{Theorem}{Theorems}

\newtheorem{proposition}{Proposition}[section]
\newtheorem{lemma}[proposition]{Lemma}
\newtheorem{corollary}[proposition]{Corollary}
\newtheorem{theorem}[proposition]{Theorem}
\newtheorem{maintheorem}{Theorem}

\theoremstyle{definition}
\newtheorem{definition}[proposition]{Definition}
\newtheorem{notation}[proposition]{Notation}

\newtheorem{remark}[proposition]{Remark}

\newtheorem{assumption}[proposition]{Assumption}

\DeclareMathOperator{\Aut}{Aut}
\DeclareMathOperator{\Gal}{Gal}

\DeclareMathOperator{\id}{id}
\DeclareMathOperator{\ad}{ad}
\DeclareMathOperator{\Char}{char}

\DeclareMathOperator{\Fix}{Fix}

\DeclareMathOperator{\E}{\mathcal{E}}
\DeclareMathOperator{\F}{\mathcal{F}}

\DeclareMathOperator{\dist}{dist}
\DeclareMathOperator{\Exp}{Exp}

\newcommand{\LL}{L}

\newcommand{\g}{L}
\newcommand{\Z}{\mathbb{Z}}

\newcommand{\dash}{\nobreakdash-\hspace{0pt}}

\begin{document}

\title[$5 \times 5$-graded Lie algebras]{$\mathbf{5 \times 5}$-graded Lie algebras, cubic norm structures and quadrangular algebras}

\author{Tom De Medts}
\author{Jeroen Meulewaeter}

\date{\today}

\begin{abstract}
    We study simple Lie algebras generated by extremal elements, over arbitrary fields of arbitrary characteristic.
    We show:
    (1) If the extremal geometry contains lines, then the Lie algebra admits a $5 \times 5$-grading that can be parametrized by a cubic norm structure;
    (2) If there exists a field extension of degree at most $2$ such that the extremal geometry over that field extension contains lines, and in addition, there exist symplectic pairs of extremal elements, then the Lie algebra admits a $5 \times 5$-grading that can be parametrized by a quadrangular algebra.
    
    One of our key tools is a new definition of exponential maps that makes sense even over fields of characteristic $2$ and $3$, which ought to be interesting in its own right.
\end{abstract}

\maketitle

\begin{flushright}
\vspace*{-4ex}
\begin{minipage}[t]{0.60\linewidth}\itshape\small
Not only was Jacques Tits a constant source of inspiration through his work, he also had a direct personal influence, notably through his threat to speak evil of our work if it did not include the characteristic 2 case.

\vspace{2mm}

\hfill\upshape \textemdash ``The Book of Involutions'' \cite[p.\@~xv]{Knus1998}
\end{minipage}
\end{flushright}

\section*{Introduction}
\addcontentsline{toc}{chapter}{Introduction}

The classification of simple Lie algebras over algebraically closed fields is a classical subject, and the situation is, in many aspects, well understood.
That is not the case at all over arbitrary fields. Typically, this situation is studied by first extending to a splitting field (e.g. the algebraic closure) and then performing Galois descent, using tools such as Galois cohomology.

This is not the only possible approach, however.
Inspired by the successful applications (e.g., by Efim Zel'manov and Alekseĭ­ Kostrikin \cite{Zelcprimemanov1990}) of \emph{sandwich elements}, i.e., non-zero elements $x$ in a Lie algebra $L$ such that $\ad_x^2 = 0$, various people have studied simple Lie algebras through the concept of \emph{extremal elements}. A non-zero element $x \in L$ is called extremal if $[x, [x, L]] \leq kx$ (together with an additional axiom that is relevant only in characteristic $2$); notice that this implies $\ad_x^3 = 0$. A typical example of an extremal element is a long root element in a classical Lie algebra. In fact, these extremal elements span a one-dimensional \emph{inner ideal} of the Lie algebra, a notion going back to John Faulkner \cite{Faulkner1973} and Georgia Benkart \cite{Benkart1974}. The first paper explicitly underlying the importance of one-dimensional inner ideals that we are aware of, is Alexander Premet's paper \cite{Premet1986} from 1986; see, in particular, its main result stating that any finite-dimensional Lie algebra over an algebraically closed field of characteristic $p>5$ contains a one-dimensional inner ideal (Theorem~1).
The idea of extremal elements has also famously been used by Vladimir Chernousov in his proof of the Hasse principle for algebraic groups of type $E_8$ \cite{Chernousov1989}. (He called the extremal elements \emph{root elements} because in his situation, they belong to root groups. A detailed exposition can be found in \cite[Chapter 6, Section 8]{Platonov1994}; see, in particular, the discussion about root elements starting on p.\@~387.)

Around the turn of the century, Arjeh Cohen and his collaborators (Anja Steinbach, Rosane Ushirobira, David Wales; Gábor Ivanyos; Dan Roozemond) obtained the clever insight that the extremal elements have a \emph{geometric} interpretation, and the corresponding one-dimensional inner ideals can serve as the point set of a so-called shadow space of a building \cite{Cohen2001,Cohen2006,Cohen2007,Cohen2008}. That idea led them to the notion of a \emph{root filtration space} and makes it possible to visualize the different relations between extremal elements. In particular, distinct extremal points can be in four possible relations: collinear, symplectic, special or hyperbolic. These notions will play an important role in \cref{section grading,se:l-exp} of our paper.
In particular, a crucial ingredient will be the result by Cohen and Ivanyos that each \emph{hyperbolic} pair of extremal points gives rise to a $5$-grading of the Lie algebra.

Since then, extremal elements in Lie algebras, and the corresponding extremal geometry, have been studied by various people; see, for instance, \cite{Draisma2008,intpanhuis2009,Roozemond2011,Cohen2012,Cuypers2015,Lopez2016,Cuypers2018,Cuypers2021,Cuypers2023} and \cite[Chapter 6]{FernandezLopez2019}.
The origin of the terminology ``extremal elements'' remains slightly mysterious.
We have learned from Arjeh Cohen that he believes that he picked it up from Jean-Pierre Serre and consequently used it in their paper \cite{Cohen2001}, but neither Cohen nor we have been able to find a written source confirming this origin, and in fact, Serre informed us that he does not think that he is responsible for this terminology, but ``that it is possible that [he] used it once in a discussion.''

\medskip

In our paper, we will adopt a different point of view on extremal elements, and relate the Lie algebras to two types of exceptional algebraic structures: \emph{cubic norm structures}, and \emph{quadrangular algebras}.

Cubic norm structures are well known: The Book of Involutions \cite{Knus1998} spends a whole Chapter IX on them. In particular, \S 38 and the Notes at the end of Chapter~IX contain a lot of additional information and (historical) background. The key players in this theory are Hans Freudenthal \cite{Freudenthal1954}, Tonny Springer \cite{Springer1962}, Kevin McCrimmon \cite{McCrimmon1969a}, Holger Petersson and Michel Racine \cite{Petersson1984,Petersson1986a,Petersson1986}.
We will also encounter ``twin cubic norm structures'' (or ``cubic norm pairs''), a notion that seems to have appeared only once before in the literature in a paper by John Faulkner \cite{Faulkner2001}.

Quadrangular algebras are certainly not so well known. They first appeared implicitly in the work of Jacques Tits and Richard Weiss on the classification of Moufang polygons, which are geometric structures (spherical buildings of rank 2) associated to simple linear algebraic groups of relative rank 2 \cite{Tits2002}. The first explicit definition of these algebras, in the anisotropic case, appeared in Weiss' monograph \cite{Weiss2006}. The general definition of arbitrary quadrangular algebras is even more recent and is due to Bernhard Mühlherr and Richard Weiss \cite{Muehlherr2019}. The definition looks rather daunting (see \cref{prelim:def quadr alg} below) mostly because of technicalities in characteristic $2$, although the definition itself is characteristic-free.

It is worth pointing out that the anisotropic quadrangular algebras are precisely the parametrizing structures for \emph{Moufang quadrangles}, and that the anisotropic cubic norm structures are precisely the parametrizing structures for \emph{Moufang hexagons}. More generally, arbitrary quadrangular algebras and arbitrary cubic norm structures (which might or might not be anisotropic) have been used in the theory of \emph{Tits quadrangles} and \emph{Tits hexagons}, respectively \cite{Muehlherr2022}. See, in particular, the appendix of~\cite{Muehlherr2019} for the connection between quadrangular algebras and Tits quadrangles, and \cite[Chapter 2]{Muehlherr2022} for the connection between cubic norm structures and Tits hexagons.
Notice that Moufang quadrangles are, in fact, buildings of type $B_2$ or $BC_2$, and that Moufang hexagons are buildings of type $G_2$. (In the language of linear algebraic groups, this is the type of the \emph{relative} root system.)

In our approach, we will start from a simple Lie algebra admitting extremal elements satisfying two possible sets of assumptions (see \cref{main:B,main:C} below).
In each of these two cases, we will use \emph{two different} hyperbolic pairs of extremal points, giving rise to two different $5$-gradings on the Lie algebra, so we obtain a \emph{$5 \times 5$-graded Lie algebra}.
Depending on which of the two cases we are considering, this $5 \times 5$-grading takes a different shape.
For the reader's convenience, we have reproduced smaller versions of \cref{fig:cns,fig:qa} of these gradings below, that can be found in full size on pages \pageref{fig:cns} and \pageref{fig:qa}, respectively.

\vspace*{-3ex}
\[
\scalebox{.39}{%
	\begin{tikzpicture}[x=28mm, y=28mm, 	label distance=-3pt]
		\node[myblob=8mm] at (0,3) (N03) {$\langle x\rangle$};
		\node[myblob=8mm] at (1,1) (N11) {$\langle d \rangle$};
		\node[myblob=16mm] at (1,2) (N12) {$J'$};
		\node[myblob=16mm] at (1,3) (N13) {$J$};
		\node[myblob=8mm] at (1,4) (N14) {$\langle c \rangle$};
		\node[myblob=16mm] at (2,1) (N21) {$[q, J]$};
		\node[myblob=24mm] at (2,2) (N22) {$L_0\cap L'_0$};
		\node[myblob=16mm] at (2,3) (N23) {$[p, J']$};
		\node[myblob=8mm] at (3,0) (N30) {$\langle q \rangle$};
		\node[myblob=16mm] at (3,1) (N31) {$[y,J']$};
		\node[myblob=16mm] at (3,2) (N32) {$[y,J]$};
		\node[myblob=8mm] at (3,3) (N33) {$\langle p \rangle$};
		\node[myblob=8mm] at (4,1) (N41) {$\langle y\rangle$};
		\path[ugentred]
			(0,4.5) node (C0) {\Large $L_{-2}$}
			(1,4.5) node (C1) {\Large $L_{-1}$}
			(2,4.5) node (C2) {\Large $L_{0}$}
			(3,4.5) node (C3) {\Large $L_{1}$}
			(4,4.5) node (C4) {\Large $L_{2}$};
		\path[ugentblue]
			(4.5,0) node (R0) {\Large $L'_{2}$}
			(4.5,1) node (R1) {\Large $L'_{1}$}
			(4.5,2) node (R2) {\Large $L'_{0}$}
			(4.5,3) node (R3) {\Large $L'_{-1}$}
			(4.5,4) node (R4) {\Large $L'_{-2}$};
		\path[ugentyellow]
			(2.8,3.2) node (H0) {$L''_{2}$}
			(1.5,3.65) node (H1) {$L''_{1}$}
			(1.5,2.65) node (H2) {$L''_{0}$}
			(.53,2.65) node (H3) {$L''_{-1}$}
			(.8,1.2) node (H4) {$L''_{-2}$};
		\draw[myedge,ugentblue]
			(N11) -- (N21) -- (N31) -- (N41)
			(N12) -- (N22) -- (N32)
			(N03) -- (N13) -- (N23) -- (N33);
		\draw[myedge,ugentblue,opacity=.15]
			(-.25,0) -- (N30) -- (R0)
			(-.25,1) -- (N11) (N41) -- (R1)
			(-.25,2) -- (N12) (N32) -- (R2)
			(-.25,3) -- (N03) (N33) -- (R3)
			(-.25,4) -- (N14) -- (R4);
		\draw[myedge,ugentred]
			(N11) -- (N12) -- (N13) -- (N14)
			(N21) -- (N22) -- (N23)
			(N30) -- (N31) -- (N32) -- (N33);
		\draw[myedge,ugentred,opacity=.15]
			(0,-.25) -- (N03) -- (C0)
			(1,-.25) -- (N11) (N14) -- (C1)
			(2,-.25) -- (N21) (N23) -- (C2)
			(3,-.25) -- (N30) (N33) -- (C3)
			(4,-.25) -- (N41) -- (C4);
		\draw[myedge,ugentyellow]
			(N03) --  (N12) -- (N21) -- (N30)	
			(N13) -- (N22) -- (N31) 
			(N14) --  (N23) -- (N32) -- (N41);
	\end{tikzpicture}
} \hspace*{5ex}
	\scalebox{.39}{%
	\begin{tikzpicture}[x=28mm, y=28mm, 	label distance=-3pt]
		\node[myblob=7mm] at (0,2) (N02) {$\langle x\rangle$};
		\node[myblob=14mm] at (1,1) (N11) {$V'$};
		\node[myblob=20mm] at (1,2) (N12) {$X$};
		\node[myblob=14mm] at (1,3) (N13) {$V$};
		\node[myblob=7mm] at (2,0) (N20) {$\langle d\rangle$};
		\node[myblob=20mm] at (2,1) (N21) {$[d, X']$};
		\node[myblob=24mm] at (2,2) (N22) {$L_0\cap L'_0$};
		\node[myblob=20mm] at (2,3) (N23) {$X'$};
		\node[myblob=7mm] at (2,4) (N24) {$\langle c\rangle$};
		\node[myblob=14mm] at (3,1) (N31) {$[y,V']$};
		\node[myblob=20mm] at (3,2) (N32) {$[y,X]$};
		\node[myblob=14mm] at (3,3) (N33) {$[y,V]$};
		\node[myblob=7mm] at (4,2) (N42) {$\langle y\rangle$};
		\path[ugentred]
			(0,4.5) node (C0) {\Large $L_{-2}$}
			(1,4.5) node (C1) {\Large $L_{-1}$}
			(2,4.5) node (C2) {\Large $L_{0}$}
			(3,4.5) node (C3) {\Large $L_{1}$}
			(4,4.5) node (C4) {\Large $L_{2}$};
		\path[ugentblue]
			(4.5,0) node (R0) {\Large $L'_{2}$}
			(4.5,1) node (R1) {\Large $L'_{1}$}
			(4.5,2) node (R2) {\Large $L'_{0}$}
			(4.5,3) node (R3) {\Large $L'_{-1}$}
			(4.5,4) node (R4) {\Large $L'_{-2}$};
		\draw[myedge,ugentblue]
			(N11) -- (N21) -- (N31)
			(N02) -- (N12) -- (N22) -- (N32) -- (N42)
			(N13) -- (N23) -- (N33);
		\draw[myedge,ugentblue,opacity=.15]
			(-.25,0) -- (N20) -- (R0)
			(-.25,1) -- (N11) (N31) -- (R1)
			(-.25,2) -- (N02) (N42) -- (R2)
			(-.25,3) -- (N13) (N33) -- (R3)
			(-.25,4) -- (N24) -- (R4);
		\draw[myedge,ugentred]
			(N11) -- (N12) -- (N13)
			(N20) -- (N21) -- (N22) -- (N23) -- (N24)
			(N31) -- (N32) -- (N33);
		\draw[myedge,ugentred,opacity=.15]
			(0,-.25) -- (N02) -- (C0)
			(1,-.25) -- (N11) (N13) -- (C1)
			(2,-.25) -- (N20) (N24) -- (C2)
			(3,-.25) -- (N31) (N33) -- (C3)
			(4,-.25) -- (N42) -- (C4);
		\draw[myedge,ugentyellow]
			(N02) --  (N13) -- (N24)	
			(N12) --  (N23) 
			(N11) --  (N22) -- (N33)	
			(N21) --  (N32) 
			(N20) --  (N31) -- (N42)	;
	\end{tikzpicture}
	}
\]
\vspace*{-4.5ex}

A closer look at these gradings reveals that, in fact, the $5 \times 5$-grading in \cref{fig:cns} is a $G_2$-grading, and the $5 \times 5$-grading in \cref{fig:qa} is a $BC_2$-grading.
We will show that in the first case, the Lie algebra is parametrized by a cubic norm pair, in the sense that each of the pieces arising in the decomposition as well as the explicit formulas describing their Lie brackets, are determined by the structure of this cubic norm pair (\cref{main:B}).
In the second case, it turns out that the Lie algebra is parametrized by a quadrangular algebra (\cref{main:C}).

The study of Lie algebras \emph{graded by root systems} is not new, and we refer, for instance, to the work of Bruce Allison, Georgia Benkart, Yun Gao and Efim Zel'manov \cite{ABG02,BZ96}. However, their goals have been rather different from ours. Whereas earlier research has focused on classifying all Lie algebras graded by certain root systems in terms of central extensions of TKK-constructions, we focus on \emph{simple} Lie algebras, but we have as goal to understand the \emph{internal structure} of the Lie algebra, via the different pieces of the grading, directly in terms of an algebraic structure.%
\footnote{Our description of the resulting internal structure of the Lie algebra is perhaps more in line with Seligman's approach from \cite{Seligman1976}. For the case of a $G_2$-grading, Seligman's Theorem III.6 somewhat resembles our description in terms of a cubic norm pair, but his Theorem III.10 for the $BC_2$-case clearly illustrates our point that the connection with quadrangular algebras reveals what is really going on internally.}
In addition, we do not \emph{assume} the existence of a $BC_2$- or $G_2$-grading, but it is a consequence of some natural conditions on the extremal elements of the Lie algebra. These conditions are satisfied very often, and in fact, there are many Lie algebras for which \emph{both} sets of conditions are satisfied, and hence admit both a $BC_2$- and a $G_2$-grading, and hence they can be parametrized both by a cubic norm structure and by a quadrangular algebra!
It is our hope that this observation will lead to a deeper understanding of connections between different forms of exceptional Lie algebras (and exceptional algebraic groups, as well as their geometric counterparts, namely the corresponding Moufang polygons and Tits polygons mentioned above). We have indicated a first impression in this direction in \cref{rem:both} in the appendix.

\medskip

We formulate our main results. We emphasize that $k$ is an arbitrary field of arbitrary characteristic.
\begin{maintheorem}\label{main:A}
	Let $L$ be a simple Lie algebra.
	Assume that for some Galois extension $k'/k$ with $k'\neq\mathbb F_2$, $L_{k'} := L \otimes_k k'$ is a simple Lie algebra generated by its pure extremal elements such that the extremal geometry contains lines.

	Consider a hyperbolic pair $(x,y)$ of extremal elements and let $L=L_{-2}\oplus L_{-1}\oplus L_0\oplus L_1\oplus L_2$ be the associated $5$\dash grading.

	Then for each $l \in L_1$, we can define an ``$l$-exponential automorphism''
	\[ \alpha \colon L \to L \colon m \mapsto m+[l,m]+q_{\alpha}(m)+n_{\alpha}(m)+v_{\alpha}(m) , \]
    where $q_{\alpha},n_{\alpha},v_{\alpha} \colon L \to L$ are maps with
	\begin{equation*}
    	q_{\alpha}(L_i)\subseteq L_{i+2}, \quad n_{\alpha}(L_i)\subseteq L_{i+3}, \quad v_{\alpha}(L_i)\subseteq L_{i+4}
	\end{equation*}
    for all $i\in [-2,2]$.
\end{maintheorem}
This is shown in \cref{recover:Galois descent} below, with the essential part being obtained by \cref{th:alg}, along with uniqueness results and other auxiliary results. (Notice that $\alpha$ is not unique: there is a one-parameter family of such automorphisms.)

In particular, when $\Char(k) \neq 2$, we can always choose $q_\alpha(m) = \tfrac{1}{2}[l, [l, m]]$, and if, in addition, $\Char(k) \neq 3$, then it is a consequence that $n_\alpha(m) = \tfrac{1}{6} [l, [l, [l, m]]]$ and that $v_\alpha(m) = \tfrac{1}{24} [l, [l, [l, [l, m]]]]$, so our $l$-exponential automorphism really are a generalization of the usual exponential automorphisms.

It is worth mentioning that the condition ``being generated by pure extremal elements'' is easily fulfilled. For instance, the main result of \cite{Cohen2008} shows that as soon as $\Char(k) > 5$, each simple Lie algebra admitting a single pure extremal element is, in fact, generated by its extremal elements. (See also \cite[Chapter~6]{FernandezLopez2019}.)
On the other hand, there exist interesting simple Lie algebras not admitting any extremal element, and in that case, studying the \emph{minimal inner ideals} instead is an interesting approach.
We refer, for instance, to our earlier work \cite{DeMedts2024} for examples of this situation.

In addition, notice that the existence of a Galois extension $k'/k$ such that the extremal geometry of the Lie algebra over $k'$ contains lines, is also easily fulfilled. Indeed, if $\Char(k) \neq 2$, then it follows from \cite[Theorem 1.1]{Cuypers2023} and the paragraph following this result that unless the Lie algebra is a symplectic Lie algebra (i.e., of the form $\mathfrak{fsp}(V, f)$ for some non-degenerate symplectic space $(V,f)$), there always even exists an extension $k'/k$ \emph{of degree at most $2$} fulfilling the requirement that the extremal geometry contains lines.
(It is a more subtle fact that the \emph{simplicity} of the Lie algebra is preserved after such a base extension, but we deal with this in \cref{le:simple'}.)

\medskip

Our \cref{main:B,main:C} are very similar in structure.
Notice the difference between the conditions between the two results.
In \cref{main:B}, the main condition is the fact that we require the existence of lines already over the base field $k$.
In \cref{main:C}, the main condition is the existence of a symplectic pair of extremal elements.
(Notice that by the preceding discussion, condition (i) in the statement of \cref{main:C} essentially only excludes symplectic Lie algebras.)

\begin{maintheorem}\label{main:B}
    Let $L$ be a simple Lie algebra defined over a field $k$ with $|k| \geq 4$.
    Assume that $L$ is generated by its pure extremal elements and that the extremal geometry contains lines.
    
    Then we can find two hyperbolic pairs $(x,y)$ and $(p,q)$ of extremal elements, with corresponding $5$-gradings
    \begin{align*}
    	L &= L_{-2}\oplus L_{-1}\oplus L_0\oplus L_1\oplus L_2, \\
    	L &= L'_{-2}\oplus L'_{-1}\oplus L'_0\oplus L'_1\oplus L'_2,
    \end{align*}
    such that the gradings intersect as in \cref{fig:cns}.
    Moreover, the structure of the Lie algebra induces various maps between the different components of this $5 \times 5$-grading, resulting in a ``twin cubic norm structure'' on the pair
    \[ J = L_{-1} \cap L'_{-1}, \quad J' = L_{-1} \cap L'_0. \]
    If the norm of this structure is not identically zero, then this results in a genuine cubic norm structure.
\end{maintheorem}

\begin{maintheorem}\label{main:C}
    Let $L$ be a simple Lie algebra defined over a field $k$ with $|k| \geq 3$.
	Assume that $L$ is generated by its pure extremal elements and that:
	\begin{enumerate}
		\item there exists a Galois extension $k'/k$ of degree at most $2$ such that the extremal geometry of $L\otimes k'$ contains lines;
		\item there exist symplectic pairs of extremal elements.
	\end{enumerate}
    Then we can find two hyperbolic pairs $(x,y)$ and $(p,q)$ of extremal elements, with corresponding $5$-gradings
    \begin{align*}
    	L &= L_{-2}\oplus L_{-1}\oplus L_0\oplus L_1\oplus L_2, \\
    	L &= L'_{-2}\oplus L'_{-1}\oplus L'_0\oplus L'_1\oplus L'_2,
    \end{align*}
    such that the gradings intersect as in \cref{fig:qa}.
    Moreover, the structure of the Lie algebra induces various maps between the different components of this $5 \times 5$-grading, resulting in a quadrangular algebra on the pair
    \[ V = L_{-1} \cap L'_{-1}, \quad X = L_{-1} \cap L'_0. \]
\end{maintheorem}

\cref{main:B} is shown in \cref{se:lines} culminating in \cref{recover:hex system prop} and \cref{main:C} is shown in \cref{se:symplectic} culminating in \cref{th:quadr main}.
For both results, we rely in a crucial way on \cref{main:A}.
For \cref{main:B}, we use this result to define the norm~$N$ and the adjoint~$\sharp$ of the cubic norm structure in \cref{recover:define norm cross etc}, relying on \cref{recover:determine J - unique}.
For \cref{main:C}, we use it to define the quadratic form $Q$ on $V$ in \cref{def:Q and T} (relying on \cref{recover quadr:unique Q}) and to define the map $\theta \colon X \times V \to V$ in \cref{def:theta}. Not surprisingly, the corresponding \cref{th:alg} is used in many places of our proof.

\subsection*{Acknowledgments}

This work has been initiated as part of the PhD project of the second author, supported by the Research Foundation Flanders (Belgium) (F.W.O.- Vlaanderen), 166032/1128720N \cite{MeulewaeterPhD}.

The first author has presented%
\footnote{A video of this presentation is available at \url{https://www.youtube.com/watch?v=fr3JaJ_ZHsw}.}
a preliminary version of these results at the Collège de France on the occasion of the ``Colloque pour honorer la mémoire de Jacques Tits'' (11--13 December 2023). We thank the organizers and the Fondation Hugo du Collège de France for the invitation.

We are grateful to Michel Brion and Erhard Neher for their comments on that occasion, which have encouraged us to write down our results in the current form.
We thank Vladimir Chernousov and Arjeh Cohen for the discussions about the origin of the terminology of extremal elements,
and Holger Petersson for providing us with the reference \cite{Faulkner2001} for cubic norm pairs.

We thank both referees for their valuable comments and their interesting suggestions for future research.

\section{Preliminaries}
\label{chapter:prelim}

Throughout, $k$ is an arbitrary commutative field, of arbitrary characteristic.
All algebras occurring in this paper are assumed to be $k$-algebras.
This preliminary section consists of three parts.
\Cref{ss:Lie} deals with Lie algebras and introduces the basic theory of extremal elements and the extremal geometry.
Most of this material is taken from~\cite{Cohen2006}.
We then introduce cubic norm structures in \cref{ss:cns} and quadrangular algebras in \cref{ss:qa}.

\subsection{Lie algebras}\label{ss:Lie}

\begin{definition}
	A \emph{$\Z$-grading} of a Lie algebra $L$ is a vector space decomposition $L=\bigoplus_{i\in \Z} L_i$ such that $[L_i,L_j]\leq L_{i+j}$ for all $i,j\in\Z$.
	If $n$ is a natural number such that $L_i=0$ for all $i\in\Z$ such that $|i|>n$ while $L_{-n}\oplus L_n\neq 0$, then we call this grading a \emph{$(2n+1)$-grading}.
	We call $L_{-n}$ and $L_n$ the \emph{ends} of this grading.
	The \emph{$i$-component} of $x\in\LL$ is the image of the projection of $x$ onto $\LL_i$.
	We also set $L_{\leq i}=\bigoplus_{j\leq i} L_j$ and $L_{\geq i}=\bigoplus_{j\geq i} L_j$.
\end{definition}



\begin{definition}[{\cite[Definition 14]{Cohen2006}}]\label{def of extremal}
	Let $\g$ be a Lie algebra over $k$.
	\begin{enumerate}
	    \item
        	A non-zero element  $x\in \g$ is called {\em extremal} if there is a map $g_x \colon \g \to k$, called the {\em extremal form} on $x$, such that for all $y,z \in \g$, we have
        	\begin{align}
            	\label{extr}
        			\big[x,[x,y]\big] &= 2g_x(y)x, \\
        		\label{P1}
        			\big[[x,y],[x,z]\big] &= g_x\big([y,z]\big)x+g_x(z)[x,y]-g_x(y)[x,z], \\
        		\label{P2}
        			\big[x,[y,[x,z]]\big] &= g_x\big([y,z]\big)x-g_x(z)[x,y]-g_x(y)[x,z].
        	\end{align}
        	The last two identities are called the {\em Premet identities}.
        	If the characteristic of $k$ is not $2$, then the Premet identities \cref{P1,P2} follow from~\cref{extr}; see \cite[Definition~14]{Cohen2006}.
        	Moreover, using the Jacobi identity, \cref{P1} and \cref{P2} are equivalent if \cref{extr} holds.

        	Note that the extremal form $g_x$ might not be unique if $\Char(k) = 2$.
	    \item
            We call  $x\in \g$ a {\em sandwich} or an \emph{absolute zero divisor} if $[x,[x,y]]=0$ and $[x,[y,[x,z]]]=0$ for all $y,z\in \g$.
            An extremal element is called {\em pure} if it is not a sandwich.
        \item
            The Lie algebra $L$ is \emph{non-degenerate} if it has no non-trivial sandwiches.
	    \item
            We denote the set of extremal elements of a Lie algebra $\g$ by $E(\g)$ or, if $\g$   is clear from the context, by $E$.
            Accordingly, we denote the set  $\{ k x \mid x\in E(\g)\}$ of {\em extremal points} in the projective space on $\g$ by $\mathcal{E}(\g)$ or $\mathcal{E}$.
	\end{enumerate}
\end{definition}
\begin{remark}\label{rem:gx}
    \begin{enumerate}
        \item
            If $x \in L$ is a sandwich, then the extremal form $g_x$ can be chosen to be identically zero;
            we adopt the convention from \cite{Cohen2006} that {\em $g_x$ is identically zero whenever $x$ is a sandwich in $\g$}.
        \item
            By \cite[Lemma 16]{Cohen2006}, the existence of two distinct functions $g_x$ and $g'_x$ satisfying the identities \cref{extr,P1,P2} implies that $\Char(k)=2$ and that $x$ is a sandwich. Combined with our convention in (i), this means that $g_x$ is always \emph{uniquely determined}.
    \end{enumerate}
\end{remark}

\begin{definition}\label{def inner ideal}
    Let $\LL$ be a Lie algebra.
    An \emph{inner ideal} of $\LL$ is a subspace $I$ of $\LL$ satisfying $[I,[I,\LL]]\leq I$.

    If $\Char(k) \neq 2$, then the $1$-dimensional inner ideals of $L$ are precisely the subspaces spanned by an extremal element.
	We will need inner ideals of dimension $>1$ only once, namely in the proof of \cref{recover:sum of extr}.
\end{definition}

Recall that if $D$ is a linear map from a Lie algebra $L$ to itself such that $D^n=0$ for some $n\in\mathbb N$ and $(n-1)!\in k^\times$, then $\exp(D)=\sum_{i=0}^{n-1}\frac{1}{i!}D^i$.
A crucial aspect of extremal elements is that they allow for exponential maps in any characteristic.
\begin{definition}[{\cite[p.\@~444]{Cohen2006}}]\label{def:exp}
    Let $x \in E$ be an extremal element, with extremal form~$g_x$. Then we define
    \[ \exp(x) \colon L \to L \colon y \mapsto y + [x,y] + g_x(y)x . \]
    We also write $\Exp(x) := \{ \exp(\lambda x) \mid \lambda \in k \}$.
\end{definition}
\begin{lemma}[{\cite[Lemma 15]{Cohen2006}}]\label{le:exp}
    For each $x \in E$, we have $\exp(x) \in \Aut(L)$.
    Moreover, $\exp(\lambda x) \exp(\mu x) = \exp((\lambda + \mu)x)$ for all $\lambda,\mu \in k$, so $\Exp(x)$ is a subgroup of $\Aut(L)$.
\end{lemma}
Notice that for an extremal element $x \in E$, we always have $\ad_x^3 = 0$.
In particular, if $\Char(k) \neq 2$, then $g_x(y) x = \tfrac{1}{2} [x, [x, y]]$, and we recover the usual exponential map $\exp(x) = \exp(\ad_x)$.

\medskip

In this paper, we will always assume that our Lie algebra $L$ is generated by its extremal elements.
This has some powerful consequences.
\begin{proposition}[{\cite[Proposition 20]{Cohen2006}}]\label{pr:g}
    Suppose that $L$ is generated by $E$. Then:
    \begin{enumerate}
        \item $L$ is linearly spanned by $E$.
        \item\label{pr:g:unique} There is a unique bilinear form $g \colon L \times L \to k$ such that $g_x(y) = g(x,y)$ for all $x \in E$ and all $y \in L$. The form $g$ is symmetric.
        \item The form $g$ associates with the Lie bracket, i.e., $g([x,y],z) = g(x,[y,z])$ for all $x,y,z \in L$.
    \end{enumerate}
\end{proposition}
\begin{proposition}[{\cite[Lemmas 21, 24, 25 and~27]{Cohen2006}}]\label{pr:2pts}
    Suppose that $L$ is generated by $E$. Let $x,y \in E$ be \emph{pure} extremal elements. Then exactly one of the following holds:
    \begin{enumerate}[\rm (a)]
        \item $kx = ky$;
        \item $kx \neq ky$, $[x,y] = 0$, and $\lambda x + \mu y \in E \cup \{ 0 \}$ for all $\lambda,\mu \in k$ (we write $x \sim y$);
        \item $kx \neq ky$, $[x,y] = 0$, and $\lambda x + \mu y \in E \cup \{ 0 \}$ if and only if $\lambda=0$ or $\mu=0$;
        \item $[x,y] \neq 0$ and $g(x,y)=0$, in which case $[x,y] \in E$ and $x \sim [x,y] \sim y$;
        \item $[x,y] \neq 0$ and $g(x,y) \neq 0$, in which case $x$ and $y$ generate an $\mathfrak{sl}_2(k)$-subalgebra.
    \end{enumerate}
    Moreover, we have $g(x,y) = 0$ in all four cases \rm{(a)}--\rm{(d)}.
\end{proposition}
This gives rise to the following definitions.
\begin{definition}\label{def:Ei}
    \begin{enumerate}
        \item
            Let $\langle x \rangle, \langle y \rangle \in \E(L)$ be two distinct pure extremal points.
            Then depending on whether the corresponding elements $x,y$ are in case (b), (c), (d) or (e) of \cref{pr:2pts}, we call the pair $\langle x \rangle, \langle y \rangle$ \emph{collinear}, \emph{symplectic}, \emph{special} or \emph{hyperbolic}, respectively.
            Following the notation from \cite{Cohen2006}, we denote the five possible relations between points of $\E$ by $\E_{-2}$ (equal), $\E_{-1}$ (collinear), $\E_0$ (symplectic), $\E_1$ (special), $\E_2$ (hyperbolic), respectively.
            We also use $E_{-2},E_{-1},E_0,E_1,E_2$ for the corresponding relations between elements of $E$.
            Moreover, we write $E_i(x)$ for the set of elements of $E$ in relation $E_i$ with $x$, and similarly for $\E_i(\langle x \rangle)$.

            We refer the reader to \cref{pr:Ei(x)} for the motivation for this notation.
        \item
            The \emph{extremal geometry} associated with $L$ is the point-line geometry with point set $\E(L)$ and line set
            \[ \F(L) := \{ \langle x,y \rangle \mid \langle x \rangle \text{ and } \langle y \rangle \text{ are collinear} \} .  \]
            It is a \emph{partial linear space}, i.e., any two distinct points lie on at most one line.
    \end{enumerate}
\end{definition}

A crucial ingredient for us will be the fact that each hyperbolic pair of extremal elements gives rise to a $5$-grading of the Lie algebra;
see \cref{pr:5gr} below.

\medskip

We finish this section with a subtle but important fact about exponential maps arising from a $5$-grading.

\begin{definition}\label{prelim:def alg}
    Assume that $\Char(k) \neq 2,3$ and let $L$ be a 5-graded finite-dimensional Lie algebra over $k$.
    We say that $\LL$ is \emph{algebraic} if for any $(x,s) \in \LL_{\sigma 1} \oplus \LL_{\sigma 2}$ (with $\sigma \in \{ +,- \}$), the linear endomorphism
    \[ \exp(\ad(x+s)) = \sum_{i=0}^{4} \frac{1}{i!} \ad(x+s)^i\]
    of $\LL$ is a Lie algebra automorphism.
\end{definition}

It is not difficult to see that when $\Char(k) \neq 2,3,5$, then any $5$-graded Lie algebra is algebraic; see \cite[Lemma 3.1.7]{Boelaert2019}.
If $\Char(k) = 5$, this is much more delicate; see \cite[\S 4.2]{Boelaert2019} and \cite[Theorem 2.10]{Stavrova2017}.

One of our key tools will precisely be an extension of the existence of such automorphisms to the case where $\Char(k)$ is arbitrary, including characteristic $2$~and~$3$; see \cref{def:l-exp,th:alg,recover:root groups descent,recover:Galois descent} below.

%

\subsection{Cubic norm structures}\label{ss:cns}

Cubic norm structures have been introduced by Kevin McCrimmon in \cite{McCrimmon1969a}.
We follow the approach from \cite[Chapter 15]{Tits2002}, which is equivalent but is more practical for our purposes.

\begin{definition}
\label{prelim:def hex}
	A \textit{cubic norm structure} is a tuple
	\[ (J,k,N,\sharp,T,\times,1),\]
	where $k$ is a field, $J$ is a vector space over $k$, $N \colon J\to k$ is a map called the \textit{norm}, $\sharp \colon J\to J$ is a map called the \textit{adjoint}, $T \colon J\times J\to k$ is a symmetric bilinear form called the \textit{trace}, $\times \colon J\times J\to J$ is a symmetric bilinear map called the \emph{Freudenthal cross product}, and $1$ is a non-zero element of $J$ called the \textit{identity} such that for all $\lambda\in k$, and all $a,b,c\in J$ we have:
	\begin{enumerate}[(i)]
		\item \label{prelim:hex axiom 1}$(\lambda a)^\sharp=\lambda^2a^\sharp$;
		\item \label{prelim:hex axiom 2}$N(\lambda a)=\lambda^3 N(a)$;
		\item \label{prelim:hex axiom 3}$T(a,b\times c)=T(a\times b,c)$;
		\item \label{prelim:hex axiom 4}$(a+b)^\sharp=a^\sharp+a\times b+b^\sharp$;
		\item \label{prelim:hex axiom 5}$N(a+b)=N(a)+T(a^\sharp,b)+T(a,b^\sharp)+N(b)$;
		\item \label{prelim:hex axiom 6}$T(a,a^\sharp)=3N(a)$;
		\item \label{prelim:hex axiom 7}$(a^\sharp)^\sharp=N(a)a$;
		\item \label{prelim:hex axiom 8}$a^\sharp\times (a\times b)=N(a)b+T(a^\sharp,b)a$;
		\item \label{prelim:hex axiom 9}$a^\sharp\times b^\sharp+(a\times b)^\sharp=T(a^\sharp,b)b+T(a,b^\sharp)a$;
		\item \label{prelim:hex axiom 10}$1^\sharp=1$;
		\item \label{prelim:hex axiom 11}$a=T(a,1)1-1\times a$.
	\end{enumerate}
	We call this cubic norm structure \textit{non-degenerate} if $\{a\in J\mid N(a)=0=T(a,J)=T(a^\sharp,J)\} = 0$.
\end{definition}

\begin{definition}\label{def:cns anis}
	We call $a\in J$ \textit{invertible} if $N(a)\neq 0$.
	We denote the set of all invertible elements of $J$ by $J^\times$.
	If all non-zero elements of $J$ are invertible, we call $J$ \textit{anisotropic}, otherwise we call it \textit{isotropic}.
\end{definition}

\begin{lemma}[{\cite[(15.18)]{Tits2002}}]
\label{prelim:hex less cond}
	In \cref{prelim:def hex}, condition \cref{prelim:hex axiom 3} is a consequence of \cref{prelim:hex axiom 4,prelim:hex axiom 5}.
	If $|k|>3$, then conditions \cref{prelim:hex axiom 6,prelim:hex axiom 8,prelim:hex axiom 9} are a consequence of \cref{prelim:hex axiom 1,prelim:hex axiom 2,prelim:hex axiom 4,prelim:hex axiom 5,prelim:hex axiom 7}.
\end{lemma}

\begin{remark}
    Every cubic norm structure can be made into a (quadratic) Jordan algebra,
    where the $U$-operator of the Jordan algebra is given by
    \begin{align*}
    	U_x(y)=T(y,x)x-y\times x^\sharp
    \end{align*}
    for all $x,y\in J$.
\end{remark}

\begin{definition}\label{prelim:def isot}
	\begin{enumerate}
		\item 
			Let $(J,k,N,\sharp,T,\times,\allowbreak 1)$ and $(J',k,N',\sharp',T',\times',\allowbreak 1')$ be two cubic norm structures.
			A vector space isomorphism $\varphi:J\to J'$ is an \textit{isomorphism} from $(J,k,N,\sharp,T,\times,\allowbreak 1)$ to $(J',k,N',\sharp',T',\times',\allowbreak 1')$ if $\varphi \circ\sharp=\sharp'\circ\varphi$ and $\varphi(1)=1'$.
		\item 
			Let $(J,k,N,\sharp,T,\times,1)$ be a cubic norm structure and let $d\in J^\times$.
			We define new maps $N_d$, $\sharp_d$, $T_d$, $\times_d$ and $1_d$ by
			\begin{align*}
				N_d(a) &:= N(d)^{-1} N(a) , \\
				a^{\sharp_d} &:= N(d)^{-1} U_d(a^\sharp) , \\
				T_d(a,b) &:= T\bigl( U_{N(d)^{-1} d^\sharp}(a), b \bigr) , \\
				a \times_d b &:= N(d)^{-1} U_d(a \times b) , \\
				1_d &:= d ,
			\end{align*}
			for all $a,b \in J$.
			Then $(J,k,N_d,\sharp_d,T_d,\times_d,1_d)$ is again a cubic norm structure; see \cite[(29.36)]{Tits2002}. 
			
			Two cubic norm structures $(J,k,N,\sharp,T,\times,1)$ and $(J',k,N',\sharp',T',\times',1')$ are called \textit{isotopic} if there exists an isomorphism from $(J',k,N',\sharp',T',\times',1')$ to $(J,k,N_d,\sharp_d,T_d,\times_d,1_d)$ for a certain $d\in J^\times$.
	\end{enumerate}
\end{definition}

\subsection{Quadrangular algebras}\label{ss:qa}

Quadrangular algebras have been introduced by Richard Weiss in \cite{Weiss2006} in the anisotropic case, and have been generalized to allow isotropic quadrangular algebras in \cite{Muehlherr2019}.
Our notation follows \cite{Muehlherr2019}, except that we use different letters for some of our spaces and maps.

\begin{definition}
\label{prelim: def quadr space}
	Let $V$ be a vector space over the field $k$.
	A \textit{quadratic form} $Q$ on~$V$ is a map $Q \colon V\to k$ such that 
	\begin{itemize}
		\item $Q(\lambda v)=\lambda^2Q(v)$ for all $\lambda\in k$ and $v\in V$;
		\item the map $T \colon V\times V\to k$ defined by 
				 \[ T(u,v)=Q(u+v)-Q(u)-Q(v) \]
				 for all $u,v \in V$ is bilinear.
	\end{itemize}
	We call $Q$ 
	\begin{itemize}
		\item \textit{anisotropic} if $Q(v)=0$ implies $v=0$;
		\item \textit{non-degenerate} if $V^\perp := \{ v\in V \mid T(v,V)=0 \} = 0$;
		\item \textit{non-singular} if it is either non-degenerate or $\dim(V^\perp)=1$ and $Q(V^\perp)\neq 0$;
		\item \textit{regular} if $\{ v \in V \mid Q(v)=0=T(v,V) \} = 0$.
	\end{itemize}
	A \textit{base point} for $Q$ is an element $e\in V$ such that $Q(e)=1$.
	Using this base point, we can define an involution $\sigma \colon V\to V$ by
	\begin{align}
		\label{prelim:def sigma}
			v^\sigma = T(v,e)e - v,
	\end{align} 
	for all $v\in V$.
\end{definition}

\begin{definition}[{\cite[Definition 2.1]{Muehlherr2019}}]
\label{prelim:def quadr alg}
	A quadrangular algebra is a tuple
	\[ (k,V,Q,T,e,X,\cdot,h,\theta),\]
	where 
	\begin{itemize}[noitemsep]
		\item $k$ is a field;
		\item $V$ is a vector space over $k$;
		\item $Q \colon V \to k$ is a regular quadratic form, with associated bilinear form $T$;
		\item $e \in V$ is a base point for $Q$;
		\item $X$ is a vector space over $k$;
		\item $\cdot \colon X\times V \to X \colon (a,v)\mapsto a \cdot v$ is a bilinear map;
		\item $h \colon X\times X\to V$ is a bilinear map;
		\item $\theta \colon X\times V\to V$ is a map;
	\end{itemize}
	 such that
	 \begin{enumerate}
	 	\item \label{prelim:def quadr alg ax 1} $a \cdot e=a$ for all $a\in X$;
	 	\item \label{prelim:def quadr alg ax 2} $(a\cdot v)\cdot v^\sigma = Q(v)a$ for all $a\in X$ and all $v\in V$;
	 	\item \label{prelim:def quadr alg ax 3} $h(a, b \cdot v) = h(b, a \cdot v) + T(h(a,b), e)v$ for all $a,b \in X$ and all $v\in V$;
	 	\item \label{prelim:def quadr alg ax 4} $T(h(a\cdot v, b), e) = T(h(a,b), v)$ for all $a,b \in X$ and all $v\in V$;
	 	\item \label{prelim:def quadr alg ax 5} For each $a \in X$, the map $v \mapsto \theta(a,v)$ is linear;
	 	\item \label{prelim:def quadr alg ax 6} $\theta(\lambda a, v) = \lambda^2 \theta (a, v)$ for all $a \in X$, $v \in V$ and $\lambda\in k$;
	 	\item \label{prelim:def quadr alg ax 7} There exists a function $\gamma \colon X\times X\to k$ such that
	 			\[ \theta(a+b, v) = \theta(a,v) + \theta(b,v) + h(a, b\cdot v) - \gamma(a,b)v \]
	 			for all $a,b\in X$ and all $v\in V$;
	 	\item \label{prelim:def quadr alg ax 8} There exists a function $\phi \colon X\times V \to k$ such that 
	 		\begin{multline*}
	 			\theta(a\cdot v, w)	= \theta(a, w^\sigma)^\sigma Q(v) - T(w, v^\sigma) \theta(a,v)^\sigma
					+ T(\theta(a,v), w^\sigma) v^\sigma + \phi(a,v)w
	 		\end{multline*}
			for all $a\in X$ and all $v,w \in V$;
		\item \label{prelim:def quadr alg ax 9} $a \cdot \theta(a,v) = (a \cdot \theta(a,e)) \cdot v$ for all $a\in X$ and all $v\in V$.	
	 \end{enumerate}
\end{definition}

\begin{notation}
\label{prelim:not pi}
	We set $\pi(a)=\theta(a,e)$ for all $a\in X$. 
\end{notation}

\begin{definition}[{\cite[Definition 2.3]{Muehlherr2019}}]
	We call a quadrangular algebra \textit{anisotropic} if $Q$ is anisotropic and $\pi(a)$ is a multiple of $e$ if and only if $a=0$.
\end{definition}

\begin{remark}
	The existence of $e \in V$ implies $V\neq 0$.
	However, contrary to the common convention, we allow $X=0$.
	In other words, we view quadratic spaces as special (degenerate) examples of quadrangular algebras.
\end{remark}

\section{Gradings from extremal elements}
\label{section grading}

In this section, we describe the $5$-grading associated with any pair of hyperbolic extremal elements and we collect several properties about this grading that we will need later.
We emphasize again that we do not make any assumptions on the field~$k$, so in particular, we allow $\Char(k) = 2$ or $3$.

\begin{notation}
    We will assume from now on that $L$ is a Lie algebra over $k$ \emph{generated by its pure extremal elements}.
    As in \cref{def of extremal}, we will write $E$ for the set of extremal elements of $L$ and $\E$ for the set of extremal points of $L$.
    We let $g$ be the symmetric bilinear form from \cref{pr:g}.
    If $V$ is any subspace of $L$, we write
    \[ V^\perp := \{ u \in L \mid g(u,v) = 0 \text{ for all } v \in V \} . \]
\end{notation}
\begin{remark}\label{rem:L simple pure}
	If, in addition, $L$ is \emph{simple}, then our assumptions imply that $L$ does not contain sandwich elements.
    Indeed, the sandwich elements are contained in the radical of the form $g$, which is an ideal of $L$. Since $L$ is generated by its pure extremal elements, the form $g$ cannot be identically zero, so the radical of $g$ must be trivial.
\end{remark}
In this section, we do not assume that $L$ is simple (except in \cref{recover:T nondeg,pr:RFS,Grading is bracket if lines}), but we will make this important assumption in the next sections.

\begin{definition}
    We define the \emph{normalizer} of $x$ as
    \[ N_L(x) := \bigl\{ l \in L \mid [x,l] \in \langle x \rangle \bigr \} . \]
\end{definition}

\begin{proposition}[{\cite[Proposition 22 and Corollary 23]{Cohen2006}}]	\label{pr:5gr}
	Let $x, y\in E$ such that $g(x,y)=1$.
	Then:
	\begin{enumerate}
	    \item
        	 $\g$ has a $\Z$-grading
        		\[ \g=\g_{-2}\oplus \g_{-1}\oplus \g_0\oplus \g_1\oplus \g_2,\]
        	with $\g_{-2}=\langle x\rangle$, $\g_{-1}=[x,U]$, $\g_0=N_{\g}(x)\cap N_{\g}(y)$, $\g_1=[y,U]$ and $\g_2=\langle y\rangle$, where
        	\[ U = \bigl\langle x, y, [x,y] \bigr\rangle^\perp . \]
	    \item\label{5gr:grading der}
        	Each $\g_i$ is contained in the $i$-eigenspace of $\ad_{[x,y]}$.
	    \item\label{5gr:iso}
            $\ad_x$ defines a linear isomorphism from $\g_1$ to $\g_{-1}$ with inverse $-\ad_y$.
        \item\label{5gr:indep}
            The filtration $L_{-2} \leq L_{\leq -1} \leq L_{\leq 0} \leq L_{\leq 1} \leq L$ only depends on $x$ and not on~$y$. More precisely, we have
            \begin{align*}
                L_{\leq 1} &= x^\perp, \\
                L_{\leq 0} &= N_L(x), \\
                L_{\leq -1} &= kx + [x, x^\perp], \\
                L_{-2} &= kx.
            \end{align*}
        \item\label{5gr:g0}
            In particular, $g_x(L_{\leq 1})=0$ and $g_y(L_{\geq -1})=0$.
	\end{enumerate}
\end{proposition}


%

These gradings are closely related to the five different relations $E_i$ that we have introduced in \cref{def:Ei}:
\begin{proposition}\label{pr:Ei(x)}
	Let $x$, $y$ and $L_i$ be as in \cref{pr:5gr}.
	For each $i\in [-2,2]$, we have
	\[ E_i(x)=(E\cap L_{\leq i}) \setminus L_{\leq i-1} . \]
\end{proposition}
\begin{proof}
	This is shown in the proof of \cite[Theorem 28]{Cohen2006}.
\end{proof}
\begin{corollary}\label{co:E-1}
	Let $x$, $y$ and $L_i$ be as in \cref{pr:5gr}.
	Then
	\[ E_{-1}(x) = \{ \lambda x + e \mid \lambda \in k, e \in E \cap L_{-1} \} . \]
\end{corollary}
\begin{proof}
    By \cref{pr:Ei(x)}, $E_{-1}(x) = E \cap (kx + L_{-1}) \setminus kx$.
    Now observe that an element $\lambda x + e$ (with $e \in L_{-1}$) is extremal if and only if $e$ is extremal, because if one of them is extremal, it is collinear with $x$ and hence all points on the line through $x$ are extremal points.
\end{proof}

\begin{lemma}\label{le:ayb}
	Let $x$, $y$ and $L_i$ be as in \cref{pr:5gr}.
    Let $a \in L_1$ and $b \in L_{-1}$.
    Then $[a, [y,b]] = g(a,b)y$.
\end{lemma}
\begin{proof}
    By \cref{pr:5gr}\cref{5gr:iso}, we can write $a = [y,c]$ for some $c \in L_{-1}$. By the Premet identity \cref{P1} and \cref{pr:5gr}\cref{5gr:g0}, we then get
    \begin{align*}
        [a, [y,b]]
        &= [[y,c], [y,b]] \\
        &= g_y([c,b])y + g_y(b) [y,c] - g_y(c) [y,b] \\
        &= g(y, [c,b])y = g([y,c], b)y = g(a,b)y . \qedhere
    \end{align*}
\end{proof}

The following lemma describes an automorphism that reverses the grading.

\begin{lemma}
\label{recover:switching grading}
	Let $x$, $y$ and $L_i$ be as in \cref{pr:5gr}.
	Consider the automorphism $\varphi=\exp(y)\exp(x)\exp(y)$.
	Then
	\begin{align*}
	   \varphi(x) &= y , \\
	   \varphi(l_{-1}) &= [y, l_{-1}] , \\
	   \varphi(l_0) &= l_0 + [x, [y, l_0]] , \\
	   \varphi(l_1) &= [x, l_1] , \\
	   \varphi(y) &= x ,
	\end{align*}
	for all $l_{-1} \in L_{-1}$, $l_0 \in L_0$, $l_1 \in L_1$.
\end{lemma}
\begin{proof}
	Since $g_x(y)=g_y(x)=1$, we get $\exp(x)(y) = y + [x,y] + x$ and $\exp(y)(x) = x - [x,y] + y$. Since $g_x([x,y]) = 0$ and $g_y([x,y]) = 0$, it now follows that $\varphi(x)=y$ and $\varphi(y)=x$.

	Next, let $l_1\in L_1$.
	Then $\exp(y)(l_1) = l_1$ and $\exp(x)(l_1) = l_1 + [x, l_1]$ because $g_x(L_1) = 0$ (by \cref{pr:5gr}\cref{5gr:g0}).
	By the same \cref{pr:5gr}\cref{5gr:g0}, we also have $g_y(L_{-1}) = 0$, hence
	\[ \varphi(l_1)=\exp(y)([x,l_1]+l_1)=[x,l_1]+ \bigl( l_1+[y,[x,l_1]] \bigr) = [x,l_1] , \]
	by \cref{pr:5gr}\cref{5gr:iso}.
	Similarly, $\varphi(l_{-1})=[y,l_{-1}]$ for all $l_{-1}\in L_{-1}$.

	Finally, let $l_0\in L_0$.
	Let $\lambda,\mu \in k$ be such that $[x,l_0]=\lambda x$ and $[y,l_0]=\mu y$.
	Then by the fact that $g$ associates with the Lie bracket, we have
	\[ \mu = g_x(\mu y) = g_x([y,l_0]) = -g_y([x,l_0]) = -g_y(\lambda x) = -\lambda. \]
	Hence
	\begin{align*}
		\varphi(l_0)
		&= \exp(y)\exp(x)(l_0 - \lambda y) \\
		&= \exp(y) (l_0 + \lambda x) - \lambda \varphi(y) \\
		&= (l_0 - \lambda y) + \lambda (x - [x,y] + y) - \lambda x \\
		&= l_0 - \lambda [x,y] = l_0 + [x,[y,l_0]] .
		\qedhere
	\end{align*}
\end{proof}
\begin{remark}\label{rem:swap}
    Because of \cref{recover:switching grading}, many results involving the $5$-grading can be ``swapped around''. We will often do this without explicitly mentioning.
\end{remark}

The automorphism group $\Aut(L)$ contains a torus preserving the grading. (This easy fact is, of course, true for any $\Z$-graded algebra.)
\begin{lemma}
\label{recover:autom of grading}
 	Let $x$, $y$ and $L_i$ be as in \cref{pr:5gr}.
	Let $\lambda\in k^\times$ be arbitrary.
	Consider the map $\varphi_\lambda \colon L\to L$ defined by $\varphi_\lambda(l_i)=\lambda^il_i$ for all $l_i\in L_i$, with $i\in [-2, 2]$.
	Then $\varphi_\lambda\in \Aut(L)$.
\end{lemma}
\begin{proof}
	Clearly $\varphi_\lambda$ is bijective.
	Consider $l_i\in L_i$ and $l_j\in L_j$ arbitrary, with $i,j\in [-2, 2]$, then
	\[ \varphi_\lambda ([l_i,l_j])=\lambda^{i+j}[l_i,l_j]=[\lambda^il_i,\lambda^jl_j]=[\varphi_\lambda(l_i),\varphi_\lambda(l_j)].\qedhere\]
\end{proof}

The following lemma will ensure the non-degeneracy of certain maps later on.

\begin{lemma}
\label{recover:T nondeg}
    Assume that $L$ is simple, and let $x$, $y$ and $L_i$ be as in \cref{pr:5gr}.
    Let $Z_i = \{ z \in L_i \mid [z, L_i] = 0 \}$ for $i = -1$ and $i = 1$.
    Then $Z_{-1} = Z_1 = 0$.
\end{lemma}
\begin{proof}
	Let $z\in Z_{-1}$ and $a\in L_{-1}$.
	Then $[a,z] = 0$ by definition of $Z_{-1}$, and $g_y(z) = g_y(a) = 0$ by \cref{pr:5gr}\cref{5gr:g0}.
	By the Premet identities, we then have
	\[ [[y,a],[y,z]]=g_y([a,z])+g_y(z)[y,a]-g_y(a)[y,z]=0. \]
	Since $[y, L_{-1}] = L_1$ by \cref{pr:5gr}\cref{5gr:iso}, this shows that $[y,Z_{-1}]\leq Z_1$.
	Similarly, $[x, Z_1] \leq Z_{-1}$, but then $Z_{-1} = [x, [y, Z_{-1}]] \leq [x, Z_1] \leq Z_{-1}$, so in fact
	\[ [y, Z_{-1}] = Z_1 \quad \text{and} \quad [x, Z_1] = Z_{-1}. \]
	Observe now that by the Jacobi identity,
	$[[L_0,Z_{-1}],L_{-1}] \leq [L_0,[Z_{-1},L_{-1}]]+[Z_{-1},[L_0,L_{-1}]]=0$ and thus $[L_0,Z_{-1}]\leq Z_{-1}$; similarly, we have $[L_0,Z_1]\leq Z_1$.
	Moreover,
	\[ [y,[Z_{-1},L_1]]=[[y,Z_{-1}],L_1]=[Z_1,L_1]=0 , \]
	hence the map $\varphi$ from \cref{recover:switching grading} fixes $[Z_{-1}, L_1]$.
	On the other hand, $\varphi(Z_{-1}) = [y, Z_{-1}] = Z_1$ and $\varphi(L_1) = [x, L_1] = L_{-1}$, and therefore
	$[Z_{-1},L_1] = [Z_1,L_{-1}]$.
	We can now see that
	\[ Z_{-1}\oplus [Z_{-1},L_1]\oplus Z_1 \]
	is an ideal of $L$.
	Because $L$ is simple, we conclude that $Z_{-1}=0$ and $Z_1=0$.
\end{proof}

The idea of looking at two different $5$-gradings obtained from hyperbolic pairs is a key ingredient for our results.
The following lemma is a very first (easy) instance of this.
\begin{lemma}
	\label{recover quadr:prop aut fixing x and y}
	Consider $x,y,a,b\in E$ such that $g_x(y)=1=g_a(b)$.
	Let
	\begin{align*}
	   L &= L_{-2}\oplus L_{-1}\oplus L_0\oplus L_1\oplus L_2 \quad \text{and} \\
	   L &= L'_{-2}\oplus L'_{-1}\oplus L'_0\oplus L'_1\oplus L'_2
	\end{align*}
    be the $5$-gradings corresponding to the hyperbolic pairs $(x,y)$ and $(a,b)$, respectively, as in \cref{pr:5gr}. 

	If $\varphi\in\Aut(L)$ maps $L_{-2}$ to $L'_{-2}$ and $L_2$ to $L'_2$, then $\varphi(L_i)=L'_i$ for all $i\in [-2,2]$.
\end{lemma}
\begin{proof}
    By assumption, there exist $\lambda,\mu\in k^\times$ such that $\varphi(x)=\lambda a$ and $\varphi(y)=\mu b$.
	We have $L_0=N_L(x)\cap N_L(y)$ and $L'_0=N_L(a)\cap N_L(b)$ by \cref{pr:5gr}, so $L_0$ is mapped to $L'_0$ by $\varphi$.

	Next, let $U = \langle x, y, [x,y] \rangle^\perp$ and $U' = \langle a, b, [a,b] \rangle^\perp$. By \cref{rem:gx} and \cref{pr:g}\cref{pr:g:unique}, the form $g$ is uniquely determined by the Lie algebra $L$, so $\varphi(U) = U'$. (Recall that ``$\perp$'' is with respect to $g$.)
	It follows that $\varphi$ maps $L_{-1} = [x,U]$ to $L'_{-1} = [a,U']$ and $L_1 = [y,U]$ to $L'_1 = [b, U']$.
\end{proof}

In \cref{Grading is bracket if lines}, we will show that if $L$ is simple and the extremal geometry has lines, we can write $x$ and $y$ as the Lie bracket of two extremal elements in $L_{-1}$ and $L_1$, respectively.
We first recall some geometric facts from \cite{Cohen2006}.

\begin{proposition}\label{pr:RFS}
    Assume that $L$ is simple and that the extremal geometry contains lines, i.e., $\F(L) \neq \emptyset$.
    Let $x,y \in \E$.
    Then:
    \begin{enumerate}
        \item\label{pr:RFS:-1} If $(x,y) \in E_{-1}$, then $E_{-1}(x) \cap E_{1}(y) \neq \emptyset$.
        \item\label{pr:RFS:0} If $(x,y) \in E_0$, then $E_0(x) \cap E_{2}(y) \neq \emptyset$.
        \item\label{pr:RFS:1} If $(x,y) \in E_1$, then $E_{-1}(x) \cap E_{2}(y) \neq \emptyset$.
        \item\label{pr:RFS:2} If $(x,y) \in E_2$, then $E_{-1}(x) \cap E_{1}(y) \neq \emptyset$, while $E_{-1}(x) \cap E_{\leq 0}(y) = \emptyset$.
    \end{enumerate}
\end{proposition}
\begin{proof}
    By \cref{rem:L simple pure}, $L$ does not contain sandwich elements.
    By \cite[Theorem~28]{Cohen2006}, $(\E, \F)$ is a so-called \emph{root filtration space}, which is either \emph{non-degenerate} or has an empty set of lines. By our assumption, the latter case does not occur.

    In particular, the statements of \cite[Lemmas 1 and 4]{Cohen2006} and \cite[Lemma 8]{Cohen2007} hold.
    Now claim (i) is \cite[Lemma 4(ii)]{Cohen2006} and claim (ii) is \cite[Lemma 8(ii)]{Cohen2007}.
    To show claim (iii), let $z := [x,y]$, so $(z,x) \in E_{-1}$. By (i), we can find some $u \in E_{-1}(x) \cap E_1(u)$.
    We can now invoke \cite[Lemma 1(v)]{Cohen2006} on the pairs $(y,x)$ and $(z,u)$ to conclude that $u \in E_{-1}(x) \cap E_2(y)$.

    Finally, to show (iv), we first observe that the non-degeneracy of the root filtration space implies that $E_{-1}(x) \neq \emptyset$ (this is condition (H) on \cite[p.~439]{Cohen2006}), i.e., there exists a line $\ell$ through $\langle x \rangle$.
    We now use the defining properties (D) and (F) of a root filtration space (see \cite[p.~435]{Cohen2006}) to see that $\E_1(y)$ has non-empty intersection with $\ell$, as required.
    The final statement $E_{-1}(x) \cap E_{\leq 0}(y) = \emptyset$ is just the defining property (D) itself.
\end{proof}

\begin{lemma}
\label{a in g_{-1}}
	Let $x$, $y$ and $L_i$ be as in \cref{pr:5gr}.
   	Then $E_{-1}(x) \cap E_1(y) = E \cap L_{-1}$.
  \end{lemma}

\begin{proof}
		By \cref{pr:Ei(x)}, we have $E_{-1}(x) = (E \cap L_{\leq -1}) \setminus L_{-2}$.
		Similarly (e.g., by applying \cref{recover:switching grading}, see \cref{rem:swap}), we have $E_1(y) = (E \cap L_{\geq {-1}}) \setminus L_{\geq 0}$.
		We conclude that $E_{-1}(x) \cap E_1(y) = E \cap L_{-1}$.
\end{proof}

\begin{lemma}
\label{recover:extr in 1 prop}
	Let $x$, $y$ and $L_i$ be as in \cref{pr:5gr}.
	Consider $e\in E\cap L_{1}$.
	Then $g_e(x)=0$ for all $x\in L_{i}$ with $i\neq -1$.
\end{lemma}
\begin{proof}
	By \cref{pr:5gr}\cref{5gr:g0}, we have $g_e(x)=0$ and $g_e(y)=0$.
	Moreover, for any $l\in L_1$, there exists $l'\in L_{-1}$ such that $l=[y,l']$ by \cref{pr:5gr}.
	Since $g$ associates with the Lie bracket, we have
	\[ g_e(l) = g_e([y,l']) = g([e,y],l') = 0 . \]
	Finally, let $l \in L_0$.
	By \cref{pr:5gr} again, there exist $e'\in L_{-1}$ such that $e=[y,e']$.
	Let $\lambda\in k$ such that $[y,l]=\lambda y$.
	Then
	\[ g_e(l) = g([y,e'],l)=-g(e',[y,l])=\lambda g_y(e')=0 . \qedhere \]
\end{proof}

\begin{proposition}
\label{Grading is bracket if lines}
    Assume that $L$ is simple and that the extremal geometry of $\g$ contains lines.
	Let $x, y$ and $\g_i$ be as in \cref{pr:5gr}.
	Then there exist $c,d\in E\cap L_{-1}$ such that $[c,d]=x$.
	Moreover,
	\[  [x,y] = [c+d, -[c,y]+[d,y]] \in [I_{-1}, I_1] , \]
	where $I_i$ is the subspace of $L_i$ spanned by $E \cap L_i$, for $i = \pm 1$.
\end{proposition}
\begin{proof}
    By \cref{pr:RFS}\cref{pr:RFS:2}, there exists an $a \in E_{-1}(x) \cap E_{1}(y)$.
    By \cref{a in g_{-1}}, $a\in E \cap \g_{-1}$.

	Now, by \cref{pr:RFS}\cref{pr:RFS:-1}, we can find some $b \in E_{-1}(x) \cap E_1(a)$.
	By \cref{co:E-1}, we can write $b = \lambda x + d$ with $d \in E \cap L_{-1}$.
	Moreover, $[a, d] = [a, b]$, and since the pair $(a,b)$ is special and both $a$ and $b$ are collinear with $x$, we must have $[a,b] = \lambda x$ for some $\lambda \in k^\times$.
	Now write $c = \lambda^{-1} a$, and we get $[c, d] = x$ as required.

	Next, observe that $g(c, y) = 0$ and $g(d, y) = 0$ because these pairs are special. Using the Jacobi identity, we get
	\begin{align*}
		[c+d,-[c,y]+[d,y]]&=-2g_{c}(y)c+2g_{d}(y)d+[c,[d,y]]-[d,[c,y]] \\
			  					  &=	[c,[d,y]]+[d,[y,c]]=-[y,[c,d]]\\
			  					  &=-[y,x]=[x,y],
	\end{align*}
	as claimed.
\end{proof}

\section{$l$-exponential automorphisms}
\label{se:l-exp}

In this section, we consider a simple Lie algebra over the field $k\neq\mathbb F_2$ which is generated by its pure extremal elements.
Such a Lie algebra admits hyperbolic pairs of extremal elements and by \cref{pr:5gr}, each such pair gives rise to a $5$-grading $L = L_{-2} \oplus L_{-1} \oplus L_0 \oplus L_1 \oplus L_2$.
The main goal if this section is to define automorphisms that behave like exponential automorphisms, with respect to elements of $L_1$ (or $L_{-1}$). If $l \in L_1$, then we call the corresponding automorphisms \emph{$l$-exponential}.
We introduce these in \cref{def:l-exp} and prove their existence and uniqueness (up to a one-parameter choice) in \cref{th:alg,recover:Galois descent}.

We first need an auxiliary result, namely, we show in \cref{recover:sum of extr} that if the extremal geometry of~$L$ contains lines, then $L_1$ is spanned by the extremal elements contained in~$L_1$ (and similarly for $L_{-1}$).
We will need this condition on the existence of lines in \cref{th:alg}, but in \cref{recover:Galois descent}, we will use a Galois descent technique that will allow us to get the same existence and uniqueness results under a much weaker condition, namely the existence of lines after extending the scalars by a Galois extension.

The existence and uniqueness of these automorphisms will play a key role in \cref{se:lines,se:symplectic}.

\begin{assumption}\label{ass:simple pure}
	We assume in this section that $L$ is a \emph{simple} Lie algebra \emph{generated by its set $E$ of pure extremal elements}.
	We also assume that $k \neq \mathbb F_2$.
    Notice that this implies that $L$ does not contain sandwich elements, and by \cite[Theorem~28]{Cohen2006}, the graph $(\E, \E_2)$ is connected, so $L$ contains many hyperbolic pairs.
\end{assumption}

\begin{notation}
    Let $(x,y) \in E_2$ be a hyperbolic pair and scale $x$ and $y$ so that $g(x,y)=1$.
    Let $L=L_{-2}\oplus L_{-1}\oplus L_0\oplus L_1\oplus L_2$ be the corresponding $5$-grading as in \cref{pr:5gr}, with $L_{-2} = \langle x \rangle$ and $L_2 = \langle y \rangle$.

    Let $I_{-1}$ be the subspace of $L_{-1}$ spanned by $E\cap L_{-1}$ and
    let $I_1$ be the subspace of $L_{1}$ spanned by $E\cap L_{1}$.
    Notice that the extremal geometry has no lines if and only if $E_{-1}$ and $E_1$ are empty, or equivalently, if and only if $I_{-1} = I_1 = 0$.
\end{notation}
We begin with a lemma that provides us with a rather ``small'' generating set for~$L$ whenever the extremal geometry has lines.

\begin{lemma}
\label{recover: y and I1}
	If $I_{-1}\neq \emptyset$, then $L$ is generated by $y$ and $I_{-1}$.
\end{lemma}
\begin{proof}
	Denote the subalgebra generated by $y$ and $I_{-1}$ by $I$.
	Since $I_{-1}\neq \emptyset$, there are lines in the extremal geometry.
	By \cref{Grading is bracket if lines}, $x \in [I_{-1},I_{-1}] \leq I$.
	By \cref{co:E-1}, all elements in $E_{-1}(x)$ are contained in $I$.

	Now observe that if $a,b \in I$, then also $\exp(a)(b) = b + [a,b] + g(a,b)a \in I$.
	In particular, the automorphism $\varphi$ introduced in \cref{recover:switching grading} maps $I$ to itself.
	This implies that also $I_1 \in I$ and all elements of $E_{-1}(y)$ belong to $I$.

	We can now replace $y$ by any $y' \in E_2(x) \cap E_{-1}(y)$; observe that both $y'$ and the ``new'' $I_{-1}$ w.r.t. $y'$ are contained in $I$.
	By repeating this procedure, we see that
	\begin{equation}\label{eq:conn-y}
	   \text{all elements of $E_2(x)$ \emph{connected} to $y$ are contained in $I$.}
	\end{equation}

	If the relation $E_0$ is non-empty, then we can invoke \cite[Lemma 2.20]{Cuypers2021}. It tells us that $\E_2(x)$ is connected, and we conclude that $E_2(x) \subseteq I$.
	If $z \in E_1(x)$, then by \cref{pr:RFS}\cref{pr:RFS:1}, we can find an element $y' \in E_{-1}(z) \cap E_2(x)$. Again, $y' \in I$ and we can replace $y$ by $y'$ to see that also the collinear point $z$ is contained in~$I$. Thus $E_1(x) \subseteq I$.
	Applying $\varphi$, we also get $E_1(y) \subseteq I$.
	The only extremal elements not considered yet are those in $C := E_0(x) \cap E_0(y)$; notice that $C = E \cap L_0$ by \cref{pr:Ei(x)}.
	So let $z \in C$ and let $\ell$ be any line through $z$ not completely contained in $C$.
	Now observe that $C$ is a subspace of the extremal geometry, i.e., if two collinear elements are contained in $C$, then the whole line through these elements lies in $C$. (This follows, for instance, from the fact that $L_0 = N_L(x) \cap N_L(y)$ by \cref{pr:5gr}).
	Hence $z$ is the unique point of $\ell$ in $C$. Since we assume $|k| \geq 3$ (see \cref{ass:simple pure}), $\ell$ contains at least $2$ other points, which we already know are contained in $I$, and we conclude that also $z \in I$.
	Since $L$ is generated by $E$, this shows that indeed $I=L$ in this case.

	Assume from now on $\E_0=\emptyset$. In this case, the extremal geometry is a generalized hexagon, and the relations $\E_{-1}$, $\E_1$ and $\E_2$ coincide with the relations ``distance $1$'', ``distance $2$'', and ``distance $3$'', respectively, in the collinearity graph.
	(Indeed, we have observed in the proof of \cref{pr:RFS} that $(\E, \F)$ is a non-degenerate root filtration space, and it has been observed in \cite[Example 3.5]{Cohen2012} that a non-degenerate root filtration space with $\E_0 = \emptyset$ is necessarily a generalized hexagon.)
	Recall that in a generalized hexagon, for any point $p$ and any line $\ell$, there is a \emph{unique} point on $\ell$ closest to $p$, called the \emph{projection} of $p$ on $\ell$.

	Let $z\in E_2(x)$ be arbitrary; our goal is to show that $z\in I$.
	If $(y,z) \in E_{\leq -1}$, then it follows from \cref{eq:conn-y} that $z \in I$.
	Next, assume that $(y,z) \in E_1$, so $\langle y\rangle$ and $\langle z\rangle$ are at distance $2$.
	Recall that the common neighbor of $\langle y\rangle$ and $\langle z\rangle$ is $\langle [y,z]\rangle$.
	By \cref{eq:conn-y}, we are done if $[y,z]\in E_2(x)$,
	so we may assume that $[y,z] \in E_1(x)$.
	Now consider any line through $\langle z \rangle$ not containing $\langle [y,z] \rangle$, and let $\langle p \rangle$ be the projection of $\langle x \rangle$ on that line, so $\dist(\langle x \rangle, \langle p \rangle) = 2$; let $q := [p,x]$.
	Observe now that
	\[\langle x\rangle,\langle [x,[y,z]]\rangle,\langle [y,z]\rangle,\langle z\rangle, \langle p \rangle, \langle q \rangle,\langle x\rangle\]
	forms an ordinary hexagon.
	Next, let $\langle b \rangle$ the projection of $\langle y \rangle$ on the line $\langle p, q \rangle$.
	Notice that $\langle b \rangle$ is different from $\langle p \rangle$ and $\dist(\langle y \rangle, \langle b \rangle) = 2$.
	If $\langle b \rangle = \langle q \rangle$, we replace $y$ by another $y'$ collinear with $y$ and $[y,z]$, which will then have a different projection on the line $\langle p,q \rangle$, so we may assume without loss of generality that $\langle b \rangle \neq \langle q \rangle$.
	Let $a := [b,y]$.
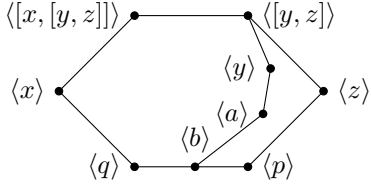
	\[
	\begin{tikzpicture}
        \draw (-1,0) -- (0,1) -- (1.5,1) -- (2.5,0) -- (1.5,-1) -- (0,-1) -- (-1,0);
        \draw (1.5,1) -- (1.8,.3) -- (1.7,-.3) -- (.8,-1);
        \foreach \x/\y/\l/\p in
        { -1/0/{$\langle x \rangle$}/left,
          0/1/{$\langle [x,[y,z]] \rangle$}/left,
          1.5/1/{$\langle [y,z] \rangle$}/right,
          2.5/0/{$\langle z \rangle$}/right,
          1.5/-1/{$\langle p \rangle$}/right,
          0/-1/{$\langle q \rangle$}/left,
          1.8/.3/{$\langle y \rangle$}/left,
          1.7/-.3/{$\langle a \rangle$}/left,
          .8/-1/{$\langle b \rangle$}/above
        } \node[inner sep=1pt,circle,draw,fill,label={\p:\l}] at (\x,\y) {};
    \end{tikzpicture}
	\]
	Since all neighbors of $\langle x \rangle$ and $\langle y \rangle$ are contained in~$I$, we see that $q \in I$ and $a \in I$, and hence also their common neighbor $\langle b \rangle$ belongs to~$I$.
	Now the line $\langle p, q \rangle$ already contains two points in $I$ (namely $\langle b \rangle$ and~$\langle q \rangle$), hence also $p \in I$.
	Finally, since $\langle p \rangle$ and $\langle [y,z] \rangle$ belong to $I$, their common neighbor $\langle z \rangle$ belongs to $I$.

	Finally, assume that $(y,z) \in E_2$, so $z\in E_2(x)\cap E_2(y)$.
	Let $\langle y\rangle,\langle a\rangle,\langle b\rangle,\langle z\rangle$ be a path of length $3$.
	If $\langle a,b\rangle$ is not completely contained in $E_{\leq 1}(x)$ we can apply the previous paragraph twice to obtain $z\in I$.
	So assume%
	\footnote{In fact, we may assume that this is true for all possible paths of length $3$ from $\langle y \rangle$ to $\langle z \rangle$, but this does not seem to lead to a simpler argument.}
	$a,b\in E_1(x)$, i.e., both $a$ and $b$ are at distance $2$ from $x$.
	Since a generalized hexagon does not contain cycles of length $\leq 5$, this can only happen if the line $\langle a,b\rangle$ contains a point $\langle c\rangle$ collinear with~$\langle x\rangle$.
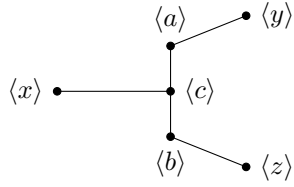
	\[
	\begin{tikzpicture}
        \draw (-1,0) -- (.5,0);
        \draw (1.5,-1) -- (.5,-.6) -- (.5,.6) -- (1.5,1);
        \foreach \x/\y/\l/\p in
        { -1/0/{$\langle x \rangle$}/left,
          .5/0/{$\langle c \rangle$}/right,
          .5/-.6/{$\langle b \rangle$}/below,
          .5/.6/{$\langle a \rangle$}/above,
          1.5/-1/{$\langle z \rangle$}/right,
          1.5/1/{$\langle y \rangle$}/right
        } \node[inner sep=1pt,circle,draw,fill,label={\p:\l}] at (\x,\y) {};
    \end{tikzpicture}
	\]
	Since $\Exp(x)$ acts sharply transitively on the set of points of the line $\langle a, b\rangle$ without~$\langle c\rangle$ (see \cref{def:exp}), we find a unique $\psi \in \Exp(x)$ such that $\psi(a) \in \langle b \rangle$.
	Let $z' := \psi(y)$, and notice that $\psi$ fixes $x$. Since $x,y \in I$, also $z' \in I$.
	Since $\psi$ fixes $x$, we get $z'\in E_2(x)$.
	Now $\langle b\rangle = \psi(\langle a\rangle)$ is a neighbor of both $\langle z\rangle$ and of $\langle z'\rangle$, and thus $z'\in E_{\leq 1}(z)$.
	Since $z'\in I$, the previous paragraph now implies $z\in I$.

	We have now shown that $E_2(x) \subseteq I$, and just as in the case where $\E_0 \neq \emptyset$, we conclude that $I = L$.
\end{proof}

\begin{corollary}
\label{recover: ideal}
	Suppose $I_{-1}\neq\emptyset$.
	If a non-trivial subspace $I \leq L$ satisfies $[I,y]\leq I$ and $[I,I_{-1}]\leq I$, then $I = L$.
\end{corollary}
\begin{proof}
	By \cref{recover: y and I1} and the Jacobi identity, $I$ is an ideal of $L$.
	Since $L$ is simple, this implies $I=L$.
\end{proof}

Another ingredient that we will need is the fact from \cite{Cohen2007} that if $\E_0 \neq \emptyset$, then each symplectic pair of points induces a subgeometry of the extremal geometry, known as a \emph{symplecton}.
We state the version from \cite[Proposition 2.12]{Cuypers2021}.
Recall that a \emph{polar space} is a partial linear space such that for each point~$p$ and each line~$\ell$, the point $p$ is collinear with either one or all points of $\ell$.
A polar space is called \emph{non-degenerate} if none if its points is collinear with all other points.
\begin{proposition}[{\cite[Proposition 2.12 and Lemma 2.18]{Cuypers2021}}]\label{pr:symplecta}
  	Assume that the extremal geometry contains lines and that $\E_0 \neq \emptyset$.
	Then the extremal geometry contains a collection  $\mathcal{S}$ of subspaces such that every pair of points $x,y$ with $(x,y)\in \E_0$ is contained in a unique element $S\in \mathcal{S}$.

	Moreover, for each $S\in \mathcal{S}$ we have:
	\begin{enumerate}
		\item\label{pr:symplecta:a} $S$ is a non-degenerate polar space (of rank $\geq 2$, i.e., containing lines).
		\item For all points $x,y\in S$, we have $(x,y)\in \E_{\leq 0}$.
		\item If a point $x$ is collinear to two non-collinear points of $S$, then $x \in S$.
		\item For each point $x$, the set of points in $\E_{\leq -1}(x)\cap S$ is either empty or contains a line.
		\item\label{pr:symplecta:e} If $x$ is a point with $\E_2(x) \cap S \neq \emptyset$, then $S \cap \E_0(x)$ contains a unique point.
	\end{enumerate}
	The elements of $\mathcal{S}$ are called symplecta.
\end{proposition}

We will need the following property of polar spaces.
\begin{lemma}
\label{recover: lemma polar}
	Let $S$ be a non-degenerate polar space.
	Let $x$ and $y$ be non-collinear points of $S$.
	Then $S$ is generated (as a point-line geometry) by $x$, $y$ and $x^\perp\cap y^\perp$, where $x^\perp$ denotes the set of points collinear with $x$.
\end{lemma}
\begin{proof}
	Let $X$ be the subspace generated by $x$, $y$ and $x^\perp\cap y^\perp$.
	Since every line through $x$ contains a point collinear to $y$, we have $x^\perp\subseteq X$, and similarly $y^\perp\subseteq X$.
	Consider $z\in X$ arbitrary.
	Let $\ell$ be any line through $x$.
	If $x^\perp\cap \ell\neq y^\perp\cap \ell$, then $\ell\subseteq X$.
	So we may assume that $\ell$ contains a unique point $a\in x^\perp\cap y^\perp$.
	By the non-degeneracy, we find $b\in x^\perp\cap y^\perp$ such that $a$ and $b$ are not collinear.
	Now $\ell$ and $by$ are opposite lines, and hence there exists a point $c$ on $\ell$ distinct from $a$ which is collinear to a point of $by$ distinct from $y$ and $b$.
	Hence $c$ lies on a line $m$ such that $m\cap y^\perp\in by\setminus \{b,y\}$ and thus $m\cap y^\perp\neq m\cap x^\perp$.
	As before, $m\subseteq X$.
	In particular, $c\in X$, so together with $a\in X$ we obtain $z\in X$.
\end{proof}

The next lemma gives a more precise description of the subspace $L_S$ of $L$ spanned by the elements in $S$.

\begin{lemma}
\label{recover: symp decomp Lie}
	Let $S$ be a symplecton containing $x$. Then
	\[ L_S=\langle x\rangle \oplus \langle \{\langle z\rangle\in S\mid z\in L_{-1}\}\rangle  \oplus \langle \{\langle z\rangle\in S\mid z\in L_{0}\}\rangle. \]
	Moreover, $\langle \{\langle z\rangle\in S\mid z\in L_{0}\}\rangle$ is $1$-dimensional.
\end{lemma}
\begin{proof}
	By \cref{pr:symplecta}\cref{pr:symplecta:e}, $S\cap \E_0(y)$ consists of a single point $\langle z\rangle$, and since $S \subseteq \E_{\leq 0}(x)$, this implies $z \in L_0$.
	By \cref{recover: lemma polar}, $L_S$ is spanned by $x$, $z$, and all points $\langle a\rangle$ collinear to both $\langle x\rangle $ and $\langle z\rangle$.
	For such a point $\langle a \rangle$, we have of course $a \in E_{-1}(x)$.
	Now $z \in E_{-1}(a) \cap E_0(y)$, so by \cref{pr:RFS}\cref{pr:RFS:2}, we have $(a,y) \not\in E_2$.
	Hence $a \in E_{-1}(x) \cap E_{\leq 1}(y) = L_{-1}$ by \cref{a in g_{-1}}.
\end{proof}

We are now ready to prove the first theorem of this section.

\begin{theorem}
\label{recover:sum of extr}
	Assume that $L$ is a simple Lie algebra over $k\neq\mathbb F_2$ generated by its pure extremal elements, and assume that the set of lines $\F$ is non-empty. Consider $x,y\in E$ with $g(x,y)=1$ and let $L=L_{-2}\oplus L_{-1}\oplus L_0\oplus L_1\oplus L_2$ be the associated $5$-grading.
	Then both $L_{-1}$ and $L_1$ are linearly spanned by the extremal elements contained in it.
\end{theorem}
\begin{proof}
	As before, let $I_{-1}$ and $I_1$ be the subspaces of $L$ linearly spanned by the extremal elements contained in $L_{-1}$ and $L_1$, respectively.
	Let
	\[ I = \langle x\rangle \oplus I_{-1}\oplus [I_{-1},I_1]\oplus I_1\oplus \langle y \rangle . \]
	We will use \cref{recover: ideal} to show that $I = L$.

    We first show that $[I, y] \leq I$.
	By \cref{Grading is bracket if lines}, we know that $[x,y]\in [I_{-1},I_1]$.
	Next, if $a\in E\cap L_{-1}$, then $a\in E_1(y)$ and hence $[a,y]\in I_{1}$.
	Since $[L_{\geq 0}, y] \leq L_2 = \langle y \rangle$, the claim $[I,y]\leq I$ is now clear.

    It remains to show that $[I,I_{-1}]\leq I$.
    Obviously $[\langle x\rangle \oplus I_{-1}\oplus I_1,I_{-1}]\leq I$.
    Similarly as before, $[\langle y\rangle,I_{-1}]\leq I_1\leq I$, so the only case left to prove is
    \[ [[I_{-1},I_1],I_{-1}]\leq I_{-1} . \]
    So consider arbitrary extremal elements $a,b\in E\cap L_{-1}$ and $c\in E\cap L_1$; our goal is to show that
    \begin{equation}\label{eq:acb}
        [[a,c],b]\in I_{-1}.
    \end{equation}
    Observe that $[a,b] \in L_{-2} = \langle x \rangle$, so by the Jacobi identity and the fact that $[x,c]$ is contained in $E\cap L_{-1}$, we have
    \[ [[a,c], b] \in I_{-1} \iff [[b,c], a] \in I_{-1} . \]
    Also observe that $c\in E_{\leq 0}(a)$ implies $[a,c]=0$ and that $c\in E_{\leq 0}(b)$ implies $[b,c]=0$.
    Hence we may assume from now on that $c\in E_{\geq 1}(a)\cap E_{\geq 1}(b)$.

    Since $a$ is extremal, $[[a,c], a] \in \langle a \rangle$ is again extremal, so \eqref{eq:acb} obviously holds if $(a,b) \in E_{-2}$.
    If $(a,b) \in E_{-1}$, then by \cite[Lemma 6.3]{Cuypers2021}, $J := \langle a,b\rangle$ is an \emph{inner ideal}, i.e., $[J, [J, L]] \leq J$; in particular, $[b,[a,c]] \in J \leq I_{-1}$.
    Next, if $(a,b) \in E_{0}$, then $a$ and $b$ are contained in a symplecton~$S$.
    The subspace $L_S$ spanned by all elements of $S$ is an inner ideal by \cite[Lemma 6.5]{Cuypers2021}, and hence $[b,[a,c]]\in L_S$.
    Since $[b,[a,c]]\in L_{-1}$, \cref{recover: symp decomp Lie} now implies that $[b,[a,c]] \in \langle \{ \langle z\rangle\in S\mid z\in L_{-1}\} \rangle \leq I_{-1}$.

    Now notice that $(a,b) \in E_2$ cannot occur, because both $a$ and $b$ are collinear to~$x$, so we can assume from now on that $(a,b) \in E_1$.
    In particular, $[a,b] \in \langle x \rangle \setminus \{ 0 \}$.

    If $c\in E_1(a)$, then $[a,c]\in E$, and since the extremal form $g$ is associative, we have $g(b,[a,c])=g([b,a],c)=0$, since $c\in E_1(x)$.
    Hence $[a,c]\in E_{\leq 1}(b)$, so either $[[a,c],b]\in E$ or $[[a,c],b]=0$.
    In either case, we obtain $[[a,c],b]\in I_{-1}$.
    The case $c\in E_1(b)$ is similar.

    The only case left to consider is $a\in E_1(b)$ and $c\in E_2(a)\cap E_2(b)$, and this case requires some more effort.
    We have $g(a,b)=0$, and by rescaling, we may assume that $g(a,c) = g(b,c) = 1$, and we also have $g(b, [a,c]) = 0$ as before.

    Let $\lambda,\mu\in k$ be arbitrary non-zero elements such that
    \begin{equation}\label{eq:lambda mu}
        \mu (\lambda - 1) = 1 ,
    \end{equation}
    which exist because $k \neq \mathbb F_2$, and consider the extremal element
    \[ d := \exp(\mu c)\exp(\lambda b)\exp(c)(a) .\]
    Then
    \begin{align*}
        d &= \exp(\mu c)\exp(\lambda b) \bigl(a + [c,a] + c \bigr) \\
        &= \exp(\mu c) \bigl( \underbrace{\lambda [b,a]}_{\in L_{-2}}
            + \underbrace{a + \lambda [b, [c,a]] + \lambda^2 b}_{\in L_{-1}}
            + \underbrace{[c,a] + \lambda [b,c]}_{\in L_0}
            + \underbrace{c\vphantom{[]}}_{\in L_1} \bigr) .
    \end{align*}
    We claim that $d \in L_{\leq -1}$.
    This could be checked by a direct computation, but it is easier to reverse the process, by starting from an arbitrary extremal element of the form
    \[ e := \lambda [b,a] + z \in \langle x \rangle + L_{-1} , \]
    compute $\exp(-\mu c)(e)$, and verify that with the correct choice of $z$, this is equal to $\exp(-\mu c)(d)$.
    We get
    \[ \exp(-\mu c)(e) = \underbrace{\lambda [b,a]}_{\in L_{-2}}
            + \underbrace{z - \lambda \mu [c,[b,a]]}_{\in L_{-1}}
            - \underbrace{\mu [c,z]}_{\in L_0}
            + \underbrace{\mu^2 g(c,z)c\vphantom{[]}}_{\in L_1} . \]
    Comparing the $L_{-1}$-components, we have no other choice than setting
    \begin{equation}\label{eq:z}
        z := a + \lambda^2 b + \lambda [b,[c,a]] + \lambda \mu [c,[b,a]] .
    \end{equation}
    Now, using the Premet identity \cref{P2} (see \cref{def of extremal}) and the fact that $g(c, [b,a]) = 0$, we get
    \begin{align*}
        [c,z]
        &= [c,a] + \lambda^2 [c,b] + \lambda [c,[b,[c,a]]] + \lambda \mu [c,[c,[b,a]]] \\
        &= [c,a] + \lambda^2 [c,b] - \lambda ([c,b] + [c,a]) \\
        &= (1 - \lambda) [c,a] + \lambda (\lambda - 1) [c,b],
    \end{align*}
    hence by \cref{eq:lambda mu}, $-\mu [c,z] = [c,a] - \lambda [c,b]$, so the $L_0$-components coincide. Finally,
    \begin{align*}
        g(c,z)
        &= g(c,a) + \lambda^2 g(c,b) + \lambda g(c,[b,[c,a]]) + \lambda \mu g(c,[c,[b,a]]) \\
        &= 1 + \lambda^2 - \lambda g([c, [c, a]], b) \\
        &= 1 + \lambda^2 - 2 \lambda g(c,a) g(c, b) = 1 + \lambda^2 - 2 \lambda = (1 - \lambda)^2 ,
    \end{align*}
    so again by \cref{eq:lambda mu}, $\mu^2 g(c,z)= 1$, and hence also the $L_1$-components coincide.
    Therefore, $e=d$, and we conclude that the element $e = \lambda [b,a] + z \in \langle x \rangle + L_{-1}$ is an extremal element.

    If $z = 0$, then by \cref{eq:z}, $[b,[c,a]]$ can be written as the sum of $3$ extremal elements contained in $L_{-1}$.
    Indeed, obviously $a, b\in E$, but also $[c,[b,a]]\in E$ since $\langle c\rangle$ and $\langle x\rangle=\langle [b,a]\rangle$ form a special pair.

    If $z \neq 0$, then $\langle e,x \rangle$ is a line of the extremal geometry.
    Since $z \in \langle e,x\rangle$, it is contained in $E$.
    So by \cref{eq:z} again, the element $[b,[c,a]]$ can be written as the sum of $4$ extremal elements contained in $L_{-1}$.

    This shows that $[b,[c,a]] \in I_{-1}$ in all possible cases.
\end{proof}

Now we will work towards showing the so-called \emph{algebraicity} of the Lie algebra, using the previous theorem.
We will give a precise definition of this property in \cref{recover:root groups descent} below, but loosely speaking it means that for any element in $L_{-1}$ there exist automorphisms depending on this element which behave nicely with respect to the $5$-grading.
Actually, if the characteristic is not $2$, $3$, or $5$, this property is easily shown (recall the discussion following \cref{prelim:def alg}), so it is not surprising that we sometimes have to handle the small characteristic cases a bit more carefully.

In the next two lemmas, we show that as soon as an extremal element has a certain form, it is contained in the image of a specific element under $\Exp(x)$.
We will use this later to show the uniqueness of certain automorphisms.

\begin{lemma}
\label{recover:char 2 uniqueness}
	Assume $\Char(k)\neq 3$.
	Let $l=l_{-2}+l_{-1}+e$ be an extremal element, with $l_i \in L_i$ and $e\in E\cap L_1$.
	Then $l_{-2} = 0$ and $l_{-1} \in \langle [x,e] \rangle$.
	In particular, $l=\exp(\lambda x)(e) = e + \lambda [x,e]$ for some $\lambda\in k$.
\end{lemma}
\begin{proof}
	By \cref{pr:5gr}\cref{5gr:grading der}, we have
	\[ [l,[l,[x,y]]]=[l,2l_{-2}+l_{-1}-e]=-3[l_{-2},e]-2[l_{-1},e] \in L_{-1} + L_0.\]
	On the other hand, since $l$ is extremal, we have $[l, [l, [x,y]]] \in \langle l \rangle$. By the previous identity, this element must have trivial $L_1$-component, but since $e \neq 0$ by assumption, this implies that $[l, [l, [x,y]]] = 0$, so
	\[ 3[l_{-2}, e] = 2[l_{-1}, e] = 0 . \]
	Since $\Char(k) \neq 3$, this implies $[l_{-2},e]=0$ and hence $l_{-2}=0$ by \cref{pr:5gr}\cref{5gr:iso}.
	If $l_{-1}=0$, then there is nothing to show, so assume $l_{-1}\neq 0$. Then
	\[ l = l_{-1} + e \in L_{-1} + L_1 , \quad l_{-1} \neq 0, \quad e \neq 0 . \]
	By \cref{pr:Ei(x)}, we have $l\in E_1(y)$ and thus $[l_{-1},y] = [l,y] \in E$.
	By \cref{pr:Ei(x)} again, $[l_{-1},y]\in E_1(x)$ and thus $l_{-1}=[x,[l_{-1},y]]\in E$.
	Since $l_{-1}$, $e$ and $l=l_{-1}+e$ are contained in $E$, \cref{pr:2pts} implies that $(l_{-1},e)\in E_{-1}$.
	Hence $\langle l_{-1}\rangle$ is the common neighbor of the special pair $(\langle x\rangle, \langle e\rangle)$ and thus $l_{-1}\in \langle [x,e]\rangle$.
\end{proof}

\begin{lemma}
\label{recover:uniqueness}
	Let $l=l_{-2}+l_{-1}+l_0+y\in E$ be an extremal element, with $l_i \in L_i$.
	Then $l_{-1} = 0$ and $l=\exp(\lambda x)(y)$ for some $\lambda\in k$.
\end{lemma}
\begin{proof}
	Let $\lambda,\mu \in k$ be such that $l_{-2}=\lambda x$ and $[l_0,x]=\mu x$.
	By \cref{pr:5gr}\cref{5gr:grading der}, we have
	\begin{align*}
		[l,[l,x]]
		&= [l, \mu x + [y,x]] \\
		&= \mu [l,x] + \bigl[ \lambda x+l_{-1}+l_0+y, [y,x] \bigr] \\
        &= (\mu^2 x+\mu[y,x])+(-2\lambda x-l_{-1}+2y).
	\end{align*}
	Since $l$ is extremal and $g(l, x) = g(y, x)=1$ by \cref{pr:5gr}\cref{5gr:g0}, we get
	\begin{equation}\label{eq:3eqs}
	   \mu^2 = 4\lambda, \quad 3 l_{-1} = 0, \quad \mu[y,x]=2l_0 .
	\end{equation}
	If $\Char(k)$ is not $2$ or $3$, then this already implies that $l =\exp(-\tfrac{1}{2} \mu x)(y)$.

	Assume now $\Char(k)=3$.
	We claim that also in this case, $l_{-1} = 0$; it then follows again that $l=\exp(-\tfrac{1}{2} \mu x)(y)$.
	To show the claim, suppose that $l_{-1} \neq 0$ and let $z := \exp(\tfrac{1}{2} \mu x)(l)$.
	Then $z$~is an extremal element, and we compute that $z = l_{-1} + y$.
	We apply the same technique as in the last paragraph of the proof of \cref{recover:char 2 uniqueness}:
	By \cref{pr:Ei(x)}, $z \in E_1(y)$ and thus $[l_{-1},y] = [z,y] \in E$.
	By \cref{pr:Ei(x)} again, $[l_{-1}, y]\in E_1(x)$ and thus $l_{-1} = [x, [l_{-1}, y]] \in E$.
	Since $l_{-1}$, $y$ and $z = l_{-1}+y$ are contained in $E$, \cref{pr:2pts} implies that $(l_{-1}, y) \in E_{-1}$, contradicting \cref{pr:Ei(x)}.

	Assume, finally, that $\Char(k)=2$. By \cref{eq:3eqs}, we have $l_{-1}=0$, so $l \in L_{-2} + L_0 + L_2$.
	By \cref{Grading is bracket if lines}, there exist $c,d \in E\cap L_{-1}$ such that $[c, d] = x$.
	By \cref{recover:extr in 1 prop}, we get for any $e\in E\cap L_{-1}$ that $g(l,e)=0$, and since $[y,e]\neq 0$, this implies that the pair $(l,e)$ is special; in particular, $[l,e] = [l_0,e] + [y,e] \in L_{-1} + L_1$ is extremal.
	By \cref{recover:char 2 uniqueness} applied on $[l,c]$ and $[l,d]$, we find $\lambda,\mu \in k$ such that $[l_0, c] = \lambda [x, [y, c]] = \lambda c$ and $[l_0, d] = \mu [x, [y, d]] = \mu d$.
	Let $\alpha := \exp(\lambda x)$. Then
	\[ \alpha([l,c])=[y,c], \quad \alpha([l,d])=[y,d] + (\lambda + \mu) d . \]
	By the Premet identity \cref{P1}, we have
	\begin{align*}
        &[[l,c],[l,d]] = g_l ([c,d])l + g_l (d)[l,c] + g_{l}(c)[l,d] = l \quad \text{and} \\
        &[[y,c],[y,d]] = g_y ([c,d])y + g_y (d)[l,c] + g_{y}(c)[l,d] = y .
	\end{align*}
	Hence 
	\begin{align*}
		\alpha(l)
		&= [\alpha([l,c]),\alpha([l,d])]=[[y,c], [y,d] + (\lambda + \mu)d] \\
		&= y + (\lambda + \mu) [[y, c], d] .
	\end{align*}
	Now note that by the Premet identity \cref{P2}, we have
	\[ [y, [d, [y, c]]] = g_y([d,c]) y + g_y(c) [y,d] + g_y(d) [y,c] = y \]
	and hence
	\[ [\alpha(l), y] = (\lambda + \mu) y , \]
	but then since $\Char(k) = 2$ and $\alpha(l)$ is an extremal element, we get
	\[ 0 = [\alpha(l), [\alpha(l), y]] = (\lambda + \mu)^2 y , \]
	so $\lambda = \mu$ and therefore $\alpha(l) = y$ and so $l = \alpha(y)$ (since $\alpha^2 = 1$).
\end{proof}

\begin{definition}
	Assume that the extremal geometry of $L$ has lines.
	We set
	\begin{align*}
	   E_+(x,y) &= \langle \exp(e) \mid e\in E \cap L_{\geq 1} \rangle\leq \Aut(L), \\
	   E_-(x,y) &= \langle \exp(e) \mid e\in E \cap L_{\leq -1} \rangle\leq \Aut(L).
	\end{align*}
\end{definition}
It is important to point out that these groups are \emph{generated} by elements of the form $\exp(e)$ and that an arbitrary element of $E_\pm(x,y)$ cannot be written as a single $\exp(e)$ in general.
In \cref{recover:root groups descent} below, we will be give a different definition of the groups $E_\pm(x,y)$ that also makes sense when the extremal geometry has no lines.

The following theorem is the key tool in our paper.
It allows us to deal with exponential maps even in characteristic $2$ and $3$.
%

\begin{definition}\label{def:l-exp}
    Let $\alpha \in \Aut(L)$ and $l \in L_1$.
    We say that $\alpha$ is an \emph{$l$-exponential automorphism} if there exist (necessarily unique) maps $q_{\alpha},n_{\alpha},v_{\alpha} \colon L \to L$ with
	\begin{equation}\label{eq:alg1}
    	q_{\alpha}(L_i)\subseteq L_{i+2}, \quad n_{\alpha}(L_i)\subseteq L_{i+3}, \quad v_{\alpha}(L_i)\subseteq L_{i+4},
	\end{equation}
    for all $i\in [-2,2]$, such that
	\begin{equation}\label{eq:alg2}
		\alpha(m)=m+[l,m]+q_{\alpha}(m)+n_{\alpha}(m)+v_{\alpha}(m)
	\end{equation}
    for all $m\in L$.
\end{definition}

\begin{theorem}
\label{th:alg}
	Assume that $L$ is a simple Lie algebra over $k\neq\mathbb F_2$ generated by its pure extremal elements, and assume that $\F\neq\emptyset$.
	Consider a hyperbolic pair $x,y\in E$ with $g(x,y)=1$ and let $L=L_{-2}\oplus L_{-1}\oplus L_0\oplus L_1\oplus L_2$ be the associated $5$-grading.
	Then:
	\begin{enumerate}
		\item \label{th:alg:comm} $\exp(a)\exp(b) = \exp([a,b])\exp(b)\exp(a)$ for all $a,b \in E \cap L_{\geq 1}$.
		\item \label{th:alg:a+b} $\exp(a)\exp(b) = \exp(a+b)$ for all $a \in E \cap L_{\geq 1}$ and $b \in L_2$.
		\item \label{th:alg:nilp} $[E_+(x,y), E_+(x,y)] = \Exp(y)$ and $[E_+(x,y), \Exp(y)] = 1$.
	\end{enumerate}
	Assume now that $l \in L_1$.
	\begin{enumerate}[resume]
		\item \label{th:alg:l-exp}
        	Write $l = e_1 + \dots + e_n$ with $e_i \in E \cap L_1$. Then
        	\[ \alpha := \exp(e_1) \dotsm \exp(e_n) \in E_+(x,y) \]
        	is an $l$-exponential automorphism.
     	\item \label{th:alg:unique}
        	If $\alpha$ is as in \cref{th:alg:l-exp} and $\beta$ is another $l$-exponential automorphism, then $\beta = \exp(z)\alpha$ for a unique $z\in L_{2}$. In particular, $\beta\in E_+(x,y)$.
        \item \label{th:alg:E+}
            Let $\alpha$ be any $l$-exponential automorphism.
            Then there exist $e_1,\dots,e_n \in E \cap L_1$ such that $l = e_1 + \dots + e_n$ and $\alpha = \exp(e_1) \dotsm \exp(e_n)$.
            Moreover,
            \[ E_+(x,y) = \bigl\{ \alpha\in \Aut(L)\mid \alpha \text{ is $m$-exponential for some } m \in L_{1} \bigr\} . \]
    	\item \label{th:alg:scalar}
            Let $\alpha$ be an $l$-exponential automorphism, with maps $q_{\alpha}, n_{\alpha}, v_{\alpha} \colon L \to L$.
            Then for each $\lambda\in k$, the map $\alpha_{\lambda}$ defined by
            \[ \alpha_{\lambda}(m)=m+\lambda[l,m]+\lambda^2q_{\alpha}(m)+\lambda^3n_{\alpha}(m)+\lambda^4v_{\alpha}(m) \]
            for all $m\in L$, is a $(\lambda l)$-exponential automorphism.
            In particular, $\alpha_\lambda \in E_+(x,y)$.
    	\item \label{th:alg:sum}
            If $\alpha$ is an $l$-exponential automorphism and $\beta$ is an $l'$-exponential automorphism, then $\alpha\beta$ is an $(l+l')$-exponential automorphism.
    	\item \label{th:alg:char}
            If $\Char (k)\neq 2$, then there is a unique $l$-exponential automorphism $\alpha$ with $q_{\alpha}(m) = \tfrac{1}{2} [l,[l,m]]$ for all $m \in L$.
            We denote this automorphism by $e_+(l)$. Its inverse is $e_+(-l)$.
	\end{enumerate}

\end{theorem}
\begin{proof}
	\begin{enumerate}
		\item 
            Let $a,b\in E \cap L_{\geq 1}$; notice that $g(a,b)=0$.
			Since $L$ is generated by $L_{\geq -1}$, it suffices to check that both sides coincide when applied on an element $l \in L_{\geq -1}$.
            We have
			\begin{align}
				\exp(a)\exp(b)(l)
				&= \exp(a)\bigl( l + [b,l] + g(b,l)b \bigr) \nonumber \\
                &= l + [a+b, l] + [a,[b,l]]  \nonumber \\
                &\hspace*{2.5ex} + g([a,b],l)a + g(b,l)[a,b] + g(a,l)a + g(b,l)b. \label{recover:algebraic eq 1}
			\end{align}
			Using the Premet identity \cref{P2}, we have
			\begin{align}
    			[[a,b],[a+b,l]]
    			&= [[a,b],[a,l]]-[[b,a],[b,l]] \nonumber \\
                &= g_a([b,l])a+g_a(l)[a,b]-g_a(b)[a,l]\nonumber\\
                &\hspace*{4ex} -g_b([a,l])b-g_b(l)[b,a]+g_b(a)[b,l]\nonumber\\
                &= g([a,b],l)a + g([a,b],l)b + (g(a,l)+g(b,l))[a,b]. \label{recover:algebraic eq 2}
			\end{align}
			Since $[a,b]\in L_2$, the automorphism $\exp([a,b])$ fixes all elements of $L_{\geq 1}$.
			Hence, by \cref{recover:algebraic eq 1} with the roles of $a$ and $b$ reversed, \cref{recover:algebraic eq 2} and the Jacobi identity, we see that
			\[ \exp(a)\exp(b)(l)=\exp([a,b])\exp(b)\exp(a)(l) \]
			as claimed.

		\item 
            We again check that both sides coincide when applied on an element $l \in L_{\geq -1}$.
            Comparing \cref{recover:algebraic eq 1} with the expression for $\exp(a+b)(l)$ and using the fact that $[a,b] = 0$ and that $b$ is a multiple of $y$, we see that this is equivalent with
            \[ [a, [y, l]] = g(a,l)y + g(y,l)a . \]
            Now both sides are equal to $0$ if $l \in L_{\geq 0}$.
            On the other hand, if $l \in L_{-2}$, then we can write $l = \lambda x$ with $\lambda \in k$ to see that both sides are equal to $\lambda a$, and if $l \in L_{-1}$, then the equality follows from \cref{le:ayb}.

		\item 
			By \cref{th:alg:comm}, we already know that $[E_+(x,y),E_+(x,y)]\leq \Exp(y)$.
			On the other hand, by \cref{Grading is bracket if lines}, the element $y$ can be written as $y = [c,d]$ for certain $c,d \in E \cap L_1$.
			By \cref{th:alg:comm} again, we see that $\exp([c,d]) \in [E_+(x,y),E_+(x,y)]$, and by rescaling, we get $\Exp(y) \leq [E_+(x,y),E_+(x,y)]$ as required.
			The second statement follows immediately from \cref{th:alg:comm} since $[a,b] = 0$ if $b \in L_2$.

		\item 
			Let $l\in L_1$.
			By \cref{recover:sum of extr}, we can write $l$ as the sum of a finite number of extremal elements in $L_1$:
			\[ l = e_1+\dots+e_n \quad \text{with } e_i\in E\cap L_1 . \]
			We will prove, by induction on $n$, that the automorphism
			\[ \alpha := \exp(e_1) \dotsm \exp(e_n) \in E_+(x,y) \]
			satisfies \cref{eq:alg2,eq:alg1} for the appropriate choices of maps $q_\alpha$, $n_\alpha$ and $v_\alpha$.
			(Notice that $\alpha$ depends on the choice of the elements $e_1,\dots,e_n$ as well as on their ordering.)

			If $n=1$, then we set
			\[ q_{\alpha}(m) := g_l(m)l, \quad n_{\alpha}(m) := 0, \quad v_{\alpha}(m) := 0 , \]
			for all $m \in L$.
			By \cref{def:exp}, we see that \cref{eq:alg2} is satisfied.
			By \cref{recover:extr in 1 prop}, $q_\alpha(L_i) = 0$ unless $i = -1$, so \cref{eq:alg1} holds as well.

			Now assume $n>1$, let $e := e_1$ and let $l'=e_2+\dots+e_n$, so $l = e + l'$.
			Let $\alpha' := \exp(e_2) \dotsm \exp(e_n)$, so $\alpha := \exp(e)\alpha'$.
			By the induction hypothesis, there exist maps $q_{\alpha'}, n_{\alpha'}$ and $v_{\alpha'}$ satisfying \cref{eq:alg2,eq:alg1}.
			We set
			\begin{align}
				q_{\alpha}(m) &:= q_{\alpha'}(m)+g_{e}(m)e+[e,[l',m]], \label{recover:alg induction P1} \\
				n_{\alpha}(m) &:= n_{\alpha'}(m)+g_{e}([l',m])e+[e,q_{\alpha'}(m)], \label{recover:alg induction P2} \\
				v_{\alpha}(m) &:= v_{\alpha'}(m)+[e,n_{\alpha'}(m)], \label{recover:alg induction P3}
			\end{align}
			for all $m\in L$.
			By \cref{recover:extr in 1 prop} and the fact that \cref{eq:alg1} holds for $\alpha'$, we get
			\begin{align*}
                &[e,v_{\alpha'}(m)]\in L_{\geq (1+4-2)}=L_{\geq 3}=0 \quad \text{and} \\
    			&g_{e}(q_{\alpha'}(m))=g_{e}(n_{\alpha'}(m))=g_{e}(v_{\alpha'}(m))=0
			\end{align*}
			for all $m \in L$,
			and by expanding
			\[ \alpha(m) = \exp(e)(\alpha'(m)) = \alpha'(m) + [e, \alpha'(m)] + g_e(\alpha'(m)) e, \]
			using the fact that \cref{eq:alg2} holds for $\alpha'$,
			we now see that \cref{eq:alg2} is satisfied for $\alpha$.
			By \cref{recover:extr in 1 prop} again, we see that also \cref{eq:alg1} holds for $\alpha$.

		\item 
			Let $l \in L_1$ and let $\beta\in\Aut(L)$ be any automorphism satisfying \cref{eq:alg2,eq:alg1}.
			Let $\alpha$ be the automorphism $\alpha = \exp(e_1) \dotsm \exp(e_n)$ as constructed in \cref{th:alg:l-exp} with respect to some choice $l = e_1 + \dots + e_n$, and observe that by the same construction,
			\[ \alpha^{-1} = \exp(-e_n) \dotsm \exp(-e_1) \]
			then satisfies \cref{eq:alg2,eq:alg1} for the element $-l = -e_n - \dots - e_1$.

			Now $\alpha$ and $\beta$ coincide on $L_{\geq 1}$.
			On the other hand, any element of $E_+(x,y)$ preserves the filtration $(L_{\geq i})$, so in particular, $\alpha^{-1}(L_{\geq 1}) = L_{\geq 1}$.
			Hence $\alpha \alpha^{-1}$ and $\beta \alpha^{-1}$ coincide on $L_{\geq 1}$, so $\gamma := \beta \alpha^{-1}$ fixes $L_{\geq 1}$.

			On the other hand, by \cref{eq:alg2} for $\beta$ and $\alpha^{-1}$, we have
			\[ \gamma(x) = \beta \alpha^{-1}(x) \in \beta(x - [l,x] + L_{\geq 0}) = x + L_{\geq 0}, \]
			so $\gamma(x)$ is an extremal element with trivial $L_{-1}$-component.
			By \cref{recover:uniqueness}, $\gamma(x) = \exp(z)(x)$ for some unique $z = \lambda y \in L_2$.
			Notice that also $\exp(z)$ fixes~$L_{\geq 1}$.
			Since $L$ is generated by $L_1$ and $x$, we conclude that $\gamma = \exp(z)$.

		\item 
            Let $\alpha \in E_+(x,y)$.
			We claim that $\alpha$ can be written as a product of elements of the form $\exp(a)$ with $a \in E \cap L_1$.
			Indeed, by definition, it is the product of elements $\exp(l)$ with $l \in E \cap L_{\geq 1}$, but each such $l$ can be written as $l = a+b$ with $a \in E \cap L_1$ and $b \in L_2$. By \cref{th:alg:a+b}, $\exp(l) = \exp(a) \exp(b)$, and by the proof of \cref{th:alg:nilp}, $\exp(b) = [\exp(c), \exp(d)]$ for certain $c,d \in E \cap L_1$; this proves our claim.
			Hence $\alpha = \exp(e_1) \dotsm \exp(e_n)$ for certain $e_1,\dots,e_n \in E \cap L_1$.
			Now let $m := e_1 + \dots + e_n$; then by \cref{th:alg:l-exp}, $\alpha$ is an $m$-exponential automorphism.
			This shows that any automorphism in $E_+(x,y)$ is $m$-exponential for some $m \in L_1$;
			the other inclusion ``$\supseteq$'' holds by \cref{th:alg:unique}.

			If $\alpha$ is any $l$-exponential automorphism, then in particular, $\alpha \in E_+(x,y)$, so by the previous paragraph, $\alpha$ is $m$-exponential for $m = e_1 + \dots + e_n$. Then by \cref{eq:alg2}, the $L_{-1}$-component of $\alpha(x)$ is equal to both $[m, x]$ and $[l, x]$, hence $m = l$, and we have shown the required decomposition.

		\item 
			By \cref{th:alg:E+}, we can write $\alpha = \exp(e_1) \dotsm \exp(e_n)$ and $l = e_1 + \dots + e_n$, where $e_i \in E \cap L_1$ .
			Now let $\beta := \exp(\lambda e_1) \dotsm \exp(\lambda e_n) \in E_+(x,y)$, with corresponding element
			$\lambda l = \lambda e_1 + \dots + \lambda e_n \in L_1$ and corresponding maps $q_\beta, n_\beta, v_\beta \colon L \to L$.
			Observe now that if we replace each $e_i$ by $\lambda e_i$ in the recursive formulas \cref{recover:alg induction P1,recover:alg induction P2,recover:alg induction P3} in the proof of \cref{th:alg:l-exp}, then this yields
			\[ q_\beta(m) = \lambda^2 q_\alpha(m), \quad
			   n_\beta(m) = \lambda^3 n_\alpha(m), \quad
			   v_\beta(m) = \lambda^4 v_\alpha(m), \]
			for all $m \in L$.
			We conclude that $\beta = \alpha_\lambda$ and hence $\alpha_\lambda \in E_+(x,y)$, which is then a $(\lambda l)$-exponential automorphism.

		\item 
            By \cref{th:alg:E+}, $\alpha = \exp(e_1) \dotsm \exp(e_n)$ and $\beta = \exp(f_1) \dotsm \exp(f_m)$ with $e_i, f_i \in E \cap L_1$ and with $l = e_1 + \dots + e_n$ and $l' = f_1 + \dots + f_m$.
            Hence $\alpha\beta = \exp(e_1) \dotsm \exp(e_n) \exp(f_1) \dotsm \exp(f_m)$, so by \cref{th:alg:l-exp}, $\alpha\beta$ is an $(l + l')$-exponential automorphism.

		\item 
			Assume $\Char(k)\neq 2$.
			We first show the existence of such an automorphism.
			As in the proof of part \cref{th:alg:l-exp}, we write $l = e_1 + \dots e_n$ with $e := e_1$ and $l' := e_2 + \dots + e_n$, and we proceed by induction on $n$.

			If $n=1$, then we choose $\alpha = \exp(e)$, which has $q_{\alpha}(m) = g_l(m)l = \tfrac{1}{2} [l,[l,m]]$ as required.
			Now assume $n>1$.
			By the induction hypothesis, there exists an $l'$-exponential automorphism $\alpha'$ with $q_{\alpha'}(m)=\tfrac{1}{2} [l',[l',m]]$ for all $m \in L$, and we first set $\beta := \exp(e) \alpha'$, so that we can invoke the formulas \cref{recover:alg induction P1,recover:alg induction P2,recover:alg induction P3}.
			In particular, we have
			\[ q_{\beta}(m)=\tfrac{1}{2} [l',[l',m]]+\tfrac{1}{2} [e,[e,m]]+[e,[l',m]] \]
			for all $m \in L$.
			Now set $\alpha := \exp(-\tfrac{1}{2} [e,l'])\beta$, and notice that $[e, l'] \in L_2$, so $\alpha$~is again an $l$-exponential automorphism.
			We get, using the Jacobi identity,
			\begin{align*}
				q_{\alpha}(m)	&=\tfrac{1}{2} [l',[l',m]]+\tfrac{1}{2} [e,[e,m]] + [e,[l',m]] +\tfrac{1}{2} [m,[e,l']] \\
							  	&=\tfrac{1}{2} [l'+e,[l'+e,m]]=\tfrac{1}{2} [l,[l,m]],
			\end{align*}
			so $\alpha$ satisfies the required assumptions, proving the existence.

			To prove uniqueness, assume that $\alpha$ and $\beta$ are two $l$-exponential automorphisms with $q_\alpha = q_\beta$; then in particular, $q_\alpha(x) = q_\beta(x)$.
			By \cref{th:alg:unique}, we have $\beta = \exp(z) \alpha$ for some $z = \lambda y \in L_2$.
			By expanding $\exp(z)$ ---see also \cref{rem:beta} below--- we get $q_\beta(x) = q_\alpha(x) + [z,x]$, but $[z,x] = \lambda [y,x]$, which is $0$ only for $\lambda = 0$, hence $\beta = \alpha$.

			Finally, let $\gamma := e_+(-l)e_+(l)$. By \cref{th:alg:sum}, $\gamma$ is a $0$-exponential automorphism, so by \cref{th:alg:unique}, we have $\gamma = \exp(z)$ for some $z = \lambda y \in L_2$ again.
			However,
			\begin{align*}
                q_\gamma(x)
                &= q_{e_+(-l)}(x) + q_{e_+(l)}(x) + [-l, [l, x]] \\
                &= \tfrac{1}{2}[-l, [-l, x]] + \tfrac{1}{2}[l, [l, x]] + [-l, [l, x]] = 0 ,
			\end{align*}
			so $\lambda = 0$ and hence $\gamma = \id$.
		\qedhere
	\end{enumerate}
\end{proof}

\begin{remark}\label{rem:beta}
    \begin{enumerate}
        \item\label{rem:beta:other}
            If $\alpha$ is an $l$-exponential automorphism for some $l \in L_1$, then by \cref{th:alg}\cref{th:alg:unique}, any other $l$-exponential automorphism $\beta$ is of the form $\beta = \exp(z) \alpha$ for some $z = \mu y \in L_2$. By expanding $\exp(z)$, we get the explicit formulas
            \begin{align*}
                q_\beta(m) &= q_\alpha(m) + [z, m], \\
                n_\beta(m) &= n_\alpha(m) + [z, [l, m]], \\
                v_\beta(m) &= v_\alpha(m) + [z, q_\alpha(m)] + g_z(m) z,
            \end{align*}
            for all $m \in L$.
            In particular, $\alpha = \beta$ if and only if $q_\alpha(x) = q_\beta(x)$. (See also the second to last paragraph of the proof of \cref{th:alg}\cref{th:alg:char}.)
        \item\label{rem:beta:product}
            If $\alpha$ is an $l$-exponential automorphism and $\beta$ is a $l'$-exponential automorphism, then we get the following explicit formulas for the $(l+l')$-exponential automorphism $\gamma := \alpha\beta$ from \cref{th:alg}\cref{th:alg:sum}:
            \begin{align*}
                q_\gamma(m) &= q_\alpha(m) + [l, [l', m]] + q_\beta(m), \\
                n_\gamma(m) &= n_\alpha(m) + q_\alpha([l',m]) + [l, q_\beta(m)] + n_\beta(m), \\
                v_\gamma(m) &= v_\alpha(m) + n_\alpha([l',m]) + q_\alpha(q_\beta(m)) + [l, n_\beta(m)] + v_\beta(m),
            \end{align*}
            for all $m \in L$.
    \end{enumerate}
\end{remark}

We now turn our attention to the case where the extremal geometry does not necessarily contain lines.

\begin{definition}
\label{recover:root groups descent}
	Let $L$ be a Lie algebra over $k$ with extremal elements $x$ and $y$ with $g(x,y)=1$, with corresponding $5$-grading as in \cref{pr:5gr}. Define
    \[ E_+(x,y) = \bigl\{ \alpha\in \Aut(L)\mid \alpha \text{ is $l$-exponential for some } l \in L_{1} \bigr\} . \]
	By \cref{th:alg}\cref{th:alg:unique}, this definition is consistent with our earlier definition when the extremal geometry has lines.
	We call $L$ \emph{algebraic} (with respect to this $5$-grading) if for each $l \in L_1$, there exists an $l$-exponential automorphism $\alpha \in \Aut(L)$.
\end{definition}

Now we can extend the results of \cref{th:alg} to a larger class of Lie algebras, using a Galois descent argument.

\begin{theorem}
\label{recover:Galois descent}
	Let $L$ be a simple Lie algebra.
	Assume that for some Galois extension $k'/k$ with $k'\neq\mathbb F_2$, $L_{k'} := L \otimes_k k'$ is a simple Lie algebra generated by its pure extremal elements and has $\F(L_{k'})\neq\emptyset$.

	Consider a hyperbolic pair $x,y\in E$ and let $L=L_{-2}\oplus L_{-1}\oplus L_0\oplus L_1\oplus L_2$ be the associated $5$-grading.

	Then $L$ is algebraic with respect to this grading.
\end{theorem}
\begin{proof}
	Consider $l\in L_1$ arbitrary.
	To simplify the notation, we will identify each $m \in L$ with $m \otimes 1 \in L \otimes k'$.

	By \cref{th:alg}\cref{th:alg:unique}, we find an $l$-exponential automorphism $\alpha$ of $L\otimes k'$.
	Since $[x,y]\neq 0$, we can find a basis $\mathcal B$ for $L_0$ containing $[x,y]$.
	Now, by \cref{eq:alg1}, we have $q_{\alpha}(x)\in L_0 \otimes k'$, so we can write
	\[ q_{\alpha}(x) = [x,y]\otimes \lambda+b_1\otimes\lambda_1+\dots +b_n\otimes\lambda_n \]
	with $b_1,\dots,b_n\in \mathcal B\setminus \{[x,y]\}$ and $\lambda,\lambda_1,\dots,\lambda_n\in k'$.

	Let $\beta := \exp(\lambda y)\alpha$ and notice that $\beta$ is again an $l$-exponential automorphism of $L \otimes k'$.
	By \cref{rem:beta}, we have $q_{\beta}=q_{\alpha}+\ad_{\lambda y}$, hence
	\[ q_{\beta}(x) = q_\alpha(x) + [y,x] \otimes \lambda =b_1\otimes\lambda_1+\dots+ b_n\otimes\lambda_n . \]

	Now each $\sigma \in$ $\Gal(k'/k)$ acts on the Lie algebra $L \otimes k'$ by sending each $m \otimes \lambda$ to $m\otimes \lambda^\sigma$.
	Since $\sigma(l) = l$, we have, for each $m'\in L\otimes k'$, that
	\begin{multline*}
	   (\sigma\circ \beta\circ \sigma^{-1})(m')
	       = m' + [l,m'] + (\sigma \circ q_{\beta} \circ \sigma^{-1})(m') \\
	       + (\sigma \circ n_{\beta_{l}} \circ \sigma^{-1})(m') + (\sigma \circ v_{\beta_{l}} \circ \sigma^{-1})(m') ,
	\end{multline*}
	so $\gamma := \sigma\circ \beta\circ \sigma^{-1}$ is again an $l$-exponential automorphism of $L \otimes k'$, with corresponding maps $q_{\gamma}=\sigma\circ q_{\beta}\circ \sigma^{-1}$, $n_{\gamma}=\sigma\circ n_{\beta}\circ \sigma^{-1}$ and $v_{\gamma}=\sigma\circ v_{\beta}\circ \sigma^{-1}$.
	Hence \cref{th:alg}\cref{th:alg:unique} implies that there exists $\mu \in k'$ such that $\gamma=\exp(\mu y)\beta$, and by \cref{rem:beta} again, we have $q_\gamma = q_\beta + \ad_{\mu y}$.
	In particular,
	\begin{align*}
		b_1 \otimes \lambda_1^\sigma + \dots + b_n \otimes \lambda_n^\sigma
		  &= q_{\gamma}(x) = [y,x] \otimes \mu +q_{\beta}(x) \\
		  &= [y,x] \otimes \mu + b_1 \otimes \lambda_1 + \dots + b_m \otimes \lambda_n .
	\end{align*}
	Since $b_1,\dots, b_n$ and $[x,y]$ are linearly independent by construction, this implies $\mu=0$ and thus $\gamma = \beta$.
	Since $\sigma \in \Gal(k'/k)$ was arbitrary, we get $\beta=\sigma\circ \beta\circ \sigma^{-1}$ for all $\sigma \in \Gal(k'/k)$, and therefore
	\[ q_{\beta}=\sigma\circ q_{\beta}\circ \sigma^{-1}, \quad n_{\beta}=\sigma\circ n_{\beta}\circ \sigma^{-1}, \quad v_{\beta}=\sigma\circ v_{\beta}\circ \sigma^{-1}, \]
	for all $\sigma \in \Gal(k'/k)$.
	Since $k'/k$ is a Galois extension and thus $\Fix(\Gal(k'/k))=k$, this implies that the maps $q_\beta$, $n_\beta$ and $v_\beta$ stabilize $L$, proving that the restriction of $\beta$ to $L$ is an automorphism of $L$. Since $\beta$ is an $l$-exponential automorphism, we are done.
%
%
%
\end{proof}

\begin{remark}
    Under the assumptions of \cref{recover:Galois descent}, many of the conclusions of \cref{th:alg} now remain valid, simply by extending the scalars to $k'$. In particular, we are allowed to use \cref{th:alg}\cref{th:alg:unique,th:alg:scalar,th:alg:sum,th:alg:char} in this more general setting.
    (On the other hand, notice that \cref{th:alg}\cref{th:alg:l-exp,th:alg:E+} do not necessarily make sense over $k$ because $E \cap L_1$ might be empty.)
\end{remark}

\begin{assumption}\label{ass:ext}
    For the rest of this section, we assume that $L$ is a simple Lie algebra over $k$ with extremal elements $x$ and $y$ with $g(x,y)=1$ and that $k'/k$ is a Galois extension such that $L\otimes k'$ is a simple Lie algebra generated by its pure extremal elements, with $\mathcal F(L\otimes k')\neq\emptyset$ and $|k'|\geq 3$.
\end{assumption}

In the next two lemmas we determine how $l$-exponential automorphisms commute, extending \cref{th:alg}\cref{th:alg:comm}.

\begin{lemma}
\label{recover:commutator}
	Let $l,l' \in L_1$ and let $\alpha$ and $\beta$ be an $l$-exponential and $l'$-exponential automorphism of $L$, respectively.
	Then $\alpha\beta = \exp([l,l'])\beta\alpha$.
\end{lemma}
\begin{proof}
    Notice that by \cref{th:alg}\cref{th:alg:sum}, both $\alpha\beta$ and $\exp([l,l'])\beta\alpha$ are $(l+l')$-exponential automorphisms.
    By the uniqueness statement of \cref{rem:beta}, it remains to show that $q_{\alpha\beta}(x) = q_{\exp([l,l'])\beta\alpha}(x)$.
    By \cref{rem:beta}\cref{rem:beta:product}, we have
    \[ q_{\alpha\beta}(x) = q_\alpha(x) + [l, [l',x]] + q_\beta(x) , \]
    and by \cref{rem:beta}\cref{rem:beta:other}, we have
    \[ q_{\exp([l,l'])\beta\alpha}(x) = q_\beta(x) + [l', [l,x]] + q_\alpha(x) + [[l,l'], x] , \]
    so by the Jacobi identity, these expressions are indeed equal.
\end{proof}

If $\Char(k)\neq 2$ we can be more precise.
\begin{lemma}
\label{th:alg group operation}
	Assume $\Char(k)\neq 2$ and let $l,l' \in L_1$.	Then
	\[ e_+(l) e_+(l') = \exp(\tfrac{1}{2} [l,l']) \, e_+(l+l') .\]
\end{lemma}
\begin{proof}
    Let $\alpha := e_+(l)$, $\beta := e_+(l')$, $\gamma := e_+(l+l')$ and $z := \tfrac{1}{2} [l,l']$.
	Exactly as in the proof of \cref{recover:commutator}, we only have to show that $q_{\alpha\beta}(x) = q_{\exp(z) \gamma}(x)$.
	Using \cref{rem:beta}, we get
    \[ q_{\alpha\beta}(x) = q_\alpha(x) + q_\beta(x) + [l, [l',x]] = \tfrac{1}{2}[l, [l, x]] + \tfrac{1}{2}[l', [l', x]] + [l, [l', x]] , \]
    and
    \[ q_{\exp(z) \gamma}(x) = q_\gamma(x) + \tfrac{1}{2} [[l,l'], x] = \tfrac{1}{2}[l+l', [l+l', x]] + \tfrac{1}{2}[[l,l'], x] , \]
    and again, the equality of these two expressions follows from the Jacobi identity.
\end{proof}

If $\Char(k)\neq 2,3$, we now get an explicit description of the maps $n_{e_+(l)}$ and $v_{e_+(l)}$.
In particular, we recover in this case (and under our running \cref{ass:ext}) the fact that the Lie algebra is algebraic with respect to the given $5$-grading, as defined in \cref{prelim:def alg}.

\begin{corollary}
\label{recover:char not 2 3}
    Assume that $\Char(k) \neq 2$.
    Let $l \in L_1$ and let $\alpha := e_+(l)$.
    Then
	\[ 6 n_\alpha(m) = [l,[l,[l,m]]] \quad \text{and} \quad 24 v_\alpha(m) = [l,[l,[l,[l,m]]]] \]
    for all $m \in L$.
\end{corollary}
\begin{proof}
	Let $\beta := \alpha_{2}$ be as in \cref{th:alg}\cref{th:alg:scalar} and let $\gamma := e_+(2l)$ as in \cref{th:alg}\cref{th:alg:char}.
	Then $q_\beta = 2^2 q_\alpha$ while $q_\gamma(m) = \tfrac{1}{2}[2l, [2l, m]] = 4 q_\alpha(m)$ for all $m \in L$, so by the uniqueness part of \cref{th:alg}\cref{th:alg:char}, we have $\beta = \gamma$.
	On the other hand, we have $e_+(2l) = e_+(l + l) = e_+(l)^2$ by \cref{th:alg group operation}, so $\alpha_2 = \alpha^2$.

	By \cref{th:alg}\cref{th:alg:scalar}, we have $n_{\alpha_2} = 8 n_\alpha$, while by \cref{rem:beta}\cref{rem:beta:product}, we have
	\[ n_{\alpha^2}(m) = 2 n_\alpha(m) + q_\alpha([l,m]) + [l, q_\alpha(m)] \]
	for all $m \in L$.
	Since $q_\alpha(m) = \tfrac{1}{2}[l, [l, m]]$, we conclude that $6 n_\alpha(m) = [l, [l, [l, m]]]$ for all $m \in L$.

    Similarly, we let $\delta := e_+(-l) e_+(l)$ and we use the fact that $\alpha_{-1} = \alpha^{-1}$.
    On the one hand, $\delta = \id$, so $v_\delta = 0$, while on the other hand, by \cref{rem:beta}\cref{rem:beta:product} again,
    \[ v_{\delta}(m) = 2 v_\alpha(m) - n_\alpha([l,m]) + q_\alpha(q_\alpha(m)) - [l, n_\alpha(m)] \]
    for all $m \in L$.
    Multiplying by $12$, we get $24 v_\alpha(m) = [l,[l,[l,[l,m]]]]$ as required.
\end{proof}

Our final result of this section shows that $E_+(x,y)$ acts sharply transitively on the set of all extremal points with non-zero $L_{-2}$-component.

\begin{proposition}
\label{recover:sharp trans opp}
	Let $L$ be a simple Lie algebra.
	Assume that for some Galois extension $k'/k$ with $k'\neq\mathbb F_2$, $L_{k'} := L \otimes_k k'$ is a simple Lie algebra generated by its pure extremal elements and has $\F(L_{k'})\neq\emptyset$.

	Consider a hyperbolic pair $x,y\in E$ with $g(x,y)=1$ and let $L=L_{-2}\oplus L_{-1}\oplus L_0\oplus L_1\oplus L_2$ be the associated $5$-grading.

	Then every extremal element with $L_{-2}$-component equal to $x$ can be written as $\varphi (x)$ for a unique $\varphi\in E_+(x,y)$.
\end{proposition}

\begin{proof}
	Let $e$ be an extremal element with $L_{-2}$-component equal to $x$.
	We first show the existence of an automorphism $\varphi \in E_+(x,y)$ with $\varphi(x) = e$.

	Let $e_{-1}$ be the $L_{-1}$-component of $e$.
	By \cref{pr:5gr}\cref{5gr:iso}, there is a (unique) $a \in L_1$ such that $[a,x] = e_{-1}$.
	By \cref{recover:Galois descent}, there exists an $a$-exponential automorphism $\alpha\in E_+(x,y)$.
	Then $\alpha^{-1}(e) = x + l_0 + l_1 + l_2$, with $l_i\in L_i$.
	By \cref{recover:uniqueness}, there exists $\lambda\in k$ such that $\exp(\lambda y)(x) = \alpha^{-1}(e)$.
	This shows that $\varphi := \alpha \exp(\lambda y) \in E_+(x,y)$ maps $x$ to $e$, as required.

	To show uniqueness, assume that also $\varphi'\in E_+(x,y)$ satisfies $\varphi'(x)=e=\varphi(x)$.
	In particular, $\varphi$ and $\varphi'$ are $l$-exponential automorphisms for the same $l \in L_1$.
	This implies that also $\varphi(m) = m + [l,m] = \varphi'(m)$ for all $m\in L_{1}$.
	Since $L$ is generated by $\{ x \} \cup L_1$, we conclude that $\varphi = \varphi'$.
\end{proof}


\begin{remark}
	Note that we did not make any assumptions on the dimension of the Lie algebra $L$ in this section.
	In particular, the results also hold for infinite-dimensional Lie algebras.
\end{remark}

\section{Extremal geometry with lines --- recovering a cubic norm structure}
\label{se:lines}

In this section, we prove that if $L$ is a simple Lie algebra over a field $k$ with $|k| \geq 4$, which is generated by its pure extremal elements and such that its extremal geometry contains lines, then $L_1$ can be decomposed as $k\oplus J\oplus J' \oplus k$, for a certain ``twin cubic norm structure'' $(J, J')$.
If the norm of this twin cubic norm structure is not the zero map, then $J' \cong J$ and $J$ gets the structure of a genuine cubic norm structure (depending on the choice of a ``base point'').
At the end of this section, we also sketch how the Lie algebra $L$ can be completely reconstructed from this cubic norm structure.


\begin{assumption}\label{ass:lines}
    Throughout the whole section, we assume that $L$ is a simple Lie algebra defined over a field $k$ with $|k| \geq 4$, such that $L$ is generated by its pure extremal elements and such that the extremal geometry contains lines.

    We fix a pair of extremal elements $x,y \in E$ with $g(x,y)=1$ and we consider the corresponding $5$-grading $L=L_{-2}\oplus L_{-1}\oplus L_0\oplus L_1\oplus L_2$ as in \cref{pr:5gr}.
\end{assumption}

\begin{notation}\label{recover:notation e12}
	Since the extremal geometry contains lines, \cref{Grading is bracket if lines} implies that there exist extremal elements $c$ and $d$ contained in $L_{-1}$ such that $[c,d]=x$.
	Set $p := [y,c]$ and $q := -[y,d]$, so $p,q \in L_1$.
	By \cref{Grading is bracket if lines} again, we see that $[x,y]=[c+d,p+q] = [c,q] + [d,p]$. (Notice that $[c,p]=0$ and $[d,q]=0$.)
	Observe that both $(c,d)$ and $(p,q)$ are special pairs, with $[c,d] = x$ and $[q,p] = y$.
\end{notation}

We start with the following observation, needed to obtain a second grading on the Lie algebra which we will use often in this section.

\begin{lemma}
\label{recover:e1 e2 hyperbolic}
	We have $g(p, d) = 1$ and $g(q, c) = 1$.
\end{lemma}
\begin{proof}
	We have $g(p, d) = g([y, c], d) = g(y, [c, d]) = g(y,x) = 1$ and $g(q, c) = g(-[y, d], c) = -g(y, [d, c]) = g(y, x) = 1$.
\end{proof}

We are now ready to define a subspace $J$ which will turn out to have the structure of a cubic norm structure.
Using the grading related to $(c,q)\in E_2$, we are able to describe a decomposition of $L_{-1}$ into $4$ parts.

\begin{notation}
\label{recover:notation J}
	We denote the $5$-grading of $L$ obtained by considering the pair $(c,q)\in E_2$ in \cref{pr:5gr} by
	\begin{equation}\label{recover:other grading}
		L=L'_{-2}\oplus L'_{-1}\oplus L'_{0}\oplus L'_{1}\oplus L'_{2},
	\end{equation}
	where $L'_{-2}=\langle c\rangle$ and $L'_2=\langle q\rangle$.
	We will occasionally also need a third grading arising from the hyperbolic pair $(d,p) \in E_2$, denoted by
	\begin{equation}\label{recover:other grading 2}
		L=L''_{-2}\oplus L''_{-1}\oplus L''_{0}\oplus L''_{1}\oplus L''_{2},
	\end{equation}
	where $L''_{-2}=\langle d\rangle$ and $L''_2=\langle p\rangle$.

	We set
	\begin{equation*}
		J = L_{-1}\cap L'_{-1},\ J'=L_{-1}\cap L'_0.
	\end{equation*}
\end{notation}



\begin{lemma}
    We have
    \begin{alignat*}{2}
        &x \in L_{-2} \cap L'_{-1} \cap L''_{-1}, \qquad
        &&y \in L_2 \cap L'_{1} \cap L''_{1}, \\
        &c \in L_{-1} \cap L'_{-2} \cap L''_{1}, \qquad
        &&d \in L_{-1} \cap L'_{1} \cap L''_{-2}, \\
        &p \in L_1 \cap L'_{-1} \cap L''_{2}, \qquad
        &&q \in L_1 \cap L'_{2} \cap L''_{-1}.
    \end{alignat*}
\end{lemma}
\begin{proof}
    By our setup in \cref{recover:notation e12}, the extremal points $\langle x\rangle,\langle c \rangle,\langle p \rangle,\langle y \rangle,\langle q \rangle,\langle d \rangle,\langle x \rangle$ form an ordinary hexagon in the extremal geometry, where all pairs at distance two are special. Now all containments follow either directly from the grading (for the $-2$- and $2$-components), or they follow from \cref{a in g_{-1}} (for the $-1$- and $1$-components).
\end{proof}

We now obtain a decomposition of $L_{-1}$ into $4$ parts, by intersecting with the $L'_i$-grading \cref{recover:other grading} arising from the hyperbolic pair $(c,q)$.

\begin{proposition}
\label{recover:L1 decomposition}
	We have a decomposition
	\begin{align*}
		L_{-1}=\langle c\rangle \oplus J\oplus J'\oplus \langle d\rangle.
	\end{align*}
	Moreover, $L_{-1}\cap L'_1=\langle d\rangle$.
\end{proposition}
\begin{proof}
	Since $c \in L_{-1}$, we have $g_{c}(L_{-1})=0$ by \cref{recover:extr in 1 prop}.
	On the other hand, $c\in L'_{-2}$, so we get $L_{-1} \leq L'_{\leq 1}$ by \cref{pr:5gr}\cref{5gr:indep}.
	
	Now let $l\in L_{-1}$ be arbitrary, and write $l = l'_{-2}+l'_{-1}+l'_0+l'_1$ with each $l'_i \in L'_i$.
	Since $L'_{-2}=\langle c\rangle\leq L_{-1}$, we may assume $l'_{-2}=0$.
	
	Now observe that by \cref{pr:5gr}\cref{5gr:grading der}, $[[c,q], l] = -l'_{-1} + l'_1$.
	On the other hand, $[c,q]\in L_0$, so $[[c,q],l]\in L_{-1}$, and we deduce that
	\begin{equation}\label{eq:cql}
	   -l'_{-1} + l'_1 \in L_{-1}.
	\end{equation}

	Next, let $\lambda \in k$ be any non-zero scalar and consider the automorphism $\varphi_\lambda \in \Aut(L)$ defined in \cref{recover:autom of grading}, but with respect to the $5$-grading \cref{recover:other grading}.
	Since $x\in L'_{-1}$ and $y\in L'_1$, we have $\varphi_\lambda(x)=\lambda^{-1} x$ and $\varphi_\lambda(y)=\lambda y$.
	Then \cref{recover quadr:prop aut fixing x and y} implies $\varphi_\lambda(L_{-1})=L_{-1}$.
	In particular, $\varphi_\lambda(l)=\lambda^{-1}l'_{-1}+l'_0+\lambda l'_1$ is contained in $L_{-1}$, so subtracting $l \in L_{-1}$, we see that
	\[ (\lambda^{-1}-1)l'_{-1}+(\lambda-1)l'_1\in L_{-1} \]
	for all non-zero $\lambda \in k$. Combining this with \cref{eq:cql}, we get
	$(\lambda^{-1} + \lambda - 2)l'_{-1}\in L_{-1}$ for all non-zero $\lambda\in k$.
	If $\Char (k)\neq 2$, we can take $\lambda=-1$ to get $l'_{-1}\in L_{-1}$, while
	if $\Char(k)=2$, then $|k|>2$ implies that we can find $\lambda$ such that $\lambda+\lambda^{-1}\neq 0$.
	In both cases, we conclude that $l'_{-1}\in L_{-1}$, and then by \eqref{eq:cql} again, also $l'_1 \in L_{-1}$.
	Since $l = l'_{-1}+l'_0+l'_1$, also $l'_0 \in L_{-1}$.
	Hence
	\[
        L_{-1}
        = \underbrace{(L_{-1} \cap L'_{-2})}_{\langle c \rangle} \, \oplus \, \underbrace{(L_{-1} \cap L'_{-1})}_{J} \,
            \oplus \, \underbrace{(L_{-1} \cap L'_0)}_{J'} \, \oplus \; {(L_{-1} \cap L'_1)} .
	\]
	Finally, let $l'_1\in L_{-1} \cap L'_1$ be arbitrary.
	By \cref{pr:5gr}\cref{5gr:iso}, we get $[q,[c,l'_1]] = -l'_1$.
	Since $[c,l'_1]\in L_{-2} = \langle x\rangle$, this implies that $l'_1$ is a multiple of $[q,x]=-d$.
	Hence $L_{-1} \cap L'_1 = \langle d\rangle$.
\end{proof}

Using this decomposition we obtain some more information on the Lie bracket.

\begin{corollary}
\label{recover:L1 other decomposition}
	We have a decomposition
	\begin{align*}
		L_1=\langle p\rangle \oplus [y,J]\oplus [y,J']\oplus \langle q\rangle.
	\end{align*}
	Moreover,
	\begin{align}
		&\langle p\rangle =L_1\cap L'_{-1},&& [y,J]=L_1\cap L'_{0}, 	&& [y,J']=L_1\cap L'_1;	\label{recover:L1 identity 1}	\\
		&[J,J]=0, 							 && [J,c]=[J,d]=0, 		&&[J,p]=0;			\label{recover:L1 identity 2}	\\
		&[J',J']=0, 						 && [J',c]=[J',d]=0, 	&&[J',q]=0;			\label{recover:L1 identity 3} 	\\
		&[J,[J,q]]\leq J',  				 && [J',[J',p]]\leq J.								\label{recover:L1 identity 4}
	\end{align}
\end{corollary}
\begin{proof}
    By \cref{pr:5gr}\cref{5gr:iso}, the map $\ad_y \colon L_{-1} \to L_1$ is an isomorphism of vector spaces, so the decomposition is clear.
	The identities \cref{recover:L1 identity 1} now follow from this new decomposition.

    Next, by combining both gradings, we have
    \begin{alignat*}{2}
        & [J,J] \leq L_{-2} \cap L'_{-2}, \quad
        && [J,c] \leq L_{-2} \cap L'_{-3}, \\
        & [J,d] \leq L_{-2} \cap L'_{0}, \quad
        && [J,p] \leq L_0 \cap L'_{-2},
    \end{alignat*}
    but since $L_{-2} = \langle x \rangle \leq L'_{-1}$ and $L'_{-2} = \langle c \rangle \leq L_{-1}$, all four intersections are trivial, proving \cref{recover:L1 identity 2}. The proof of \cref{recover:L1 identity 3} is similar.

	Finally, $[J, [J, q]] \leq L_{-1-1+1} \cap L'_{-1-1+2} = L_{-1} \cap L'_{0} = J'$, and similarly, we have $[J', [J', p]] \leq L_{-1} \cap L'_{-1} = J$, showing \cref{recover:L1 identity 4}.
\end{proof}

\begin{lemma}\label{recover:alternative}
We have
	\begin{align}
		J &=\{ l\in L_{-1} \mid [c,l] = [d,l] = [p,l] = 0 \}, \label{recover:alternative for J} \\
	 	J' &=\{ l\in L_{-1} \mid [c,l] = [d,l] = [q,l] = 0 \}. \label{recover:alternative for J'}
	\end{align}
\end{lemma}
\begin{proof}
	By \cref{recover:L1 identity 2}, it is clear that $J$ is contained in the set on the right hand side of~\cref{recover:alternative for J}.

	Conversely, let $l\in L_{-1}$ such that $[c,l] = [d,l] = [p,l] = 0$ and use \cref{recover:L1 decomposition} to write $l = \lambda c + j + j' + \mu d$, with $\lambda,\mu \in k$, $j \in J$ and $j' \in J'$.
	Since $[c,d]\neq 0$, it already follows from $[c,l] = 0$ that $\mu = 0$ and from $[d,l] = 0$ that $\lambda = 0$, so we already have $l = j + j'$ and it remains to show that $j' = 0$.
	
	Since $[p,l] = 0$ and $[p,j] = 0$, we have $[p, j'] = 0$.
	Recall from \cref{recover:notation e12} that $[x,y] = [c,q] + [d,p]$ and $[d, j'] = 0$.
	Since $j'\in L_{-1} \cap L'_0$, we can use \cref{pr:5gr}\cref{5gr:grading der} twice to get
	\[ -j' = [[x,y],j'] = [[c,q],j'] + [[d,p],j'] = [d,[p,j']] = 0, \]
	as claimed.
	The proof of \cref{recover:alternative for J'} is completely similar.
\end{proof}

The subspace $L_0$ also has a decomposition, but this time it is a decomposition into $3$ parts.

\begin{proposition}
	We have a decomposition
	\begin{equation}\label{recover:decomp L0}
		L_0=(L_0\cap L'_{-1})\oplus (L_0\cap L'_0)\oplus (L_0\cap L'_1).
	\end{equation}
	Moreover,
	\begin{equation}\label{recover:decomp L0 eq}
		L_0\cap L'_1 = [q,J] \quad \text{and} \quad
		L_0\cap L'_{-1}	= [p,J'] . 
	\end{equation}
\end{proposition}
\begin{proof}
	By \cref{recover:extr in 1 prop} applied to $c\in L_{-1}$ we get $g_{c}(L_0)=0$.
	By \cref{pr:5gr}\cref{5gr:indep}, this implies $L_{0}\leq L'_{-2}\oplus L'_{-1}\oplus L'_0\oplus L'_1$.
	Similarly, $g_q(L_0) = 0$ implies $L_0\leq L'_{-1}\oplus L'_0\oplus L'_1\oplus L'_2$.
	Hence $L_0\leq L'_{-1}\oplus L'_0\oplus L'_1$.
	
	We can now apply the same technique as in the proof of \cref{recover:L1 decomposition}. So let $l \in L_0$ be arbitrary and write $l = l'_{-1} + l'_0 + l'_1$ with each $l'_i \in L'_i$.
	Again, $[[c,q],l] = - l'_{-1} + l'_1 \in L_0$, and using the automorphisms $\varphi_\lambda$, we get $\lambda_{-1} l'_{-1} + l'_0 + \lambda l'_1 \in L_0$ for all $\lambda \in k^\times$, yielding $l'_i \in L_0$ for all three values of $i$, exactly as in the proof of \cref{recover:L1 decomposition}. This proves \cref{recover:decomp L0}.

    To show \cref{recover:decomp L0 eq}, notice that the inclusions $[q, J] \leq L_0 \cap L'_1$ and $[p,J'] \leq L_0 \cap L'_{-1}$ are clear from the gradings.
	Conversely, let $l_0\in L_0\cap L'_1$ be arbitrary and let $l := - [c, l_0]$. Then $l \in L_{-1} \cap L'_{-1} = J$.
	By \cref{pr:5gr}\cref{5gr:iso}, we have $l_0 = [q,l]$, so indeed $l_0 \in [q, J]$.
	Similarly, let $l_0 \in L_0 \cap L'_{-1}$ be arbitrary and let $l := -[d, l_0]$. Then $l \in L_{-1} \cap L'_0 = J'$.
	By \cref{pr:5gr}\cref{5gr:iso} again, we have $l_0 = [p,l]$, so indeed $l_0 \in [p, J']$.
\end{proof}

In the next lemma we deduce some connections between the two gradings \cref{recover:other grading,recover:other grading 2}.

\begin{lemma}
	We have
	\begin{align}
		L_{-1}\cap L''_{-1}&=J', 			& L_{-1}\cap L''_0&=J, 			&L_{-1}\cap L''_{1}&=\langle c\rangle,
																\label{recover:third gr eq 1}\\
		L_0\cap L''_{-1}&=[q,J], 		&L_0\cap L''_0&=L_0\cap L'_0, 	&L_0\cap L''_{1}&=[p,J'].
																\label{recover:third gr eq 2}
	\end{align}
\end{lemma}
\begin{proof}
    Notice that in the setup in \cref{recover:notation e12}, we can replace $c$ with $d$, $d$ with $-c$, $p$~with $-q$ and $q$ with $p$.
    In the resulting decompositions introduced in \cref{recover:notation J}, this interchanges the two gradings \cref{recover:other grading,recover:other grading 2}.
    If we subsequently set $K := L_{-1} \cap L''_{-1}$ and $K' := L_{-1} \cap L''_0$, then it follows from \cref{recover:alternative} that $K = J'$ and $K' = J$.
    By \cref{recover:L1 decomposition} applied on this new setup, we get
	\[ L_{-1} = \langle d\rangle \oplus K \oplus K' \oplus \langle c\rangle \]
	with $\langle c\rangle=L_{-1}\cap L''_{1}$.
	This shows \cref{recover:third gr eq 1}.

	By \cref{recover:decomp L0 eq} applied on our new setup, we get $L_0\cap L''_1 = [p, J']$ and $L_0\cap L''_{-1}=[q,J]$.
	Finally, consider $l\in L_0\cap L''_0$.
	By \cref{pr:5gr}\cref{5gr:grading der}, we have $[[x,y], l] = 0$ and $[[d,p], l] = 0$.
	Recall from \cref{recover:notation e12} that $[x,y] = [c,q] + [d,p]$, so we get
	\[ 0 = [[x,y],l] = [[c,q],l]+[[d,p],l]=[[c,q],l].\]
	Hence $l$ is contained in the $0$-eigenspace of $\ad_{[c,q]}$, which is $L'_0$ by \cref{pr:5gr}\cref{5gr:grading der} again.
	Hence $L_0\cap L''_0\leq L_0\cap L'_0$, and $L_0\cap L'_0\leq L_0\cap L''_0$ follows similarly.
	This shows \cref{recover:third gr eq 2}.
\end{proof}

We have summarized all the information about the intersections of the different gradings in \cref{fig:cns}.
\begin{figure}[ht!]
\[
\scalebox{.85}{%
	\begin{tikzpicture}[x=28mm, y=28mm, 	label distance=-3pt]
		\node[myblob=8mm] at (0,3) (N03) {$\langle x\rangle$};
		\node[myblob=8mm] at (1,1) (N11) {$\langle d \rangle$};
		\node[myblob=16mm] at (1,2) (N12) {$J'$};
		\node[myblob=16mm] at (1,3) (N13) {$J$};
		\node[myblob=8mm] at (1,4) (N14) {$\langle c \rangle$};
		\node[myblob=16mm] at (2,1) (N21) {$[q, J]$};
		\node[myblob=24mm] at (2,2) (N22) {$L_0\cap L'_0$};
		\node[myblob=16mm] at (2,3) (N23) {$[p, J']$};
		\node[myblob=8mm] at (3,0) (N30) {$\langle q \rangle$};
		\node[myblob=16mm] at (3,1) (N31) {$[y,J']$};
		\node[myblob=16mm] at (3,2) (N32) {$[y,J]$};
		\node[myblob=8mm] at (3,3) (N33) {$\langle p \rangle$};
		\node[myblob=8mm] at (4,1) (N41) {$\langle y\rangle$};
		\path[ugentred]
			(0,4.5) node (C0) {\Large $L_{-2}$}
			(1,4.5) node (C1) {\Large $L_{-1}$}
			(2,4.5) node (C2) {\Large $L_{0}$}
			(3,4.5) node (C3) {\Large $L_{1}$}
			(4,4.5) node (C4) {\Large $L_{2}$};
		\path[ugentblue]
			(4.5,0) node (R0) {\Large $L'_{2}$}
			(4.5,1) node (R1) {\Large $L'_{1}$}
			(4.5,2) node (R2) {\Large $L'_{0}$}
			(4.5,3) node (R3) {\Large $L'_{-1}$}
			(4.5,4) node (R4) {\Large $L'_{-2}$};
		\path[ugentyellow]
			(2.8,3.2) node (H0) {$L''_{2}$}
			(1.5,3.65) node (H1) {$L''_{1}$}
			(1.5,2.65) node (H2) {$L''_{0}$}
			(.53,2.65) node (H3) {$L''_{-1}$}
			(.8,1.2) node (H4) {$L''_{-2}$};
		\draw[myedge,ugentblue]
			(N11) -- (N21) -- (N31) -- (N41)
			(N12) -- (N22) -- (N32)
			(N03) -- (N13) -- (N23) -- (N33);
		\draw[myedge,ugentblue,opacity=.15]
			(-.25,0) -- (N30) -- (R0)
			(-.25,1) -- (N11) (N41) -- (R1)
			(-.25,2) -- (N12) (N32) -- (R2)
			(-.25,3) -- (N03) (N33) -- (R3)
			(-.25,4) -- (N14) -- (R4);
		\draw[myedge,ugentred]
			(N11) -- (N12) -- (N13) -- (N14)
			(N21) -- (N22) -- (N23)
			(N30) -- (N31) -- (N32) -- (N33);
		\draw[myedge,ugentred,opacity=.15]
			(0,-.25) -- (N03) -- (C0)
			(1,-.25) -- (N11) (N14) -- (C1)
			(2,-.25) -- (N21) (N23) -- (C2)
			(3,-.25) -- (N30) (N33) -- (C3)
			(4,-.25) -- (N41) -- (C4);
		\draw[myedge,ugentyellow]
			(N03) --  (N12) -- (N21) -- (N30)	
			(N13) -- (N22) -- (N31) 
			(N14) --  (N23) -- (N32) -- (N41);
	\end{tikzpicture}
}
\]
\caption{Intersecting gradings for cubic norm structures}\label{fig:cns}
\end{figure}
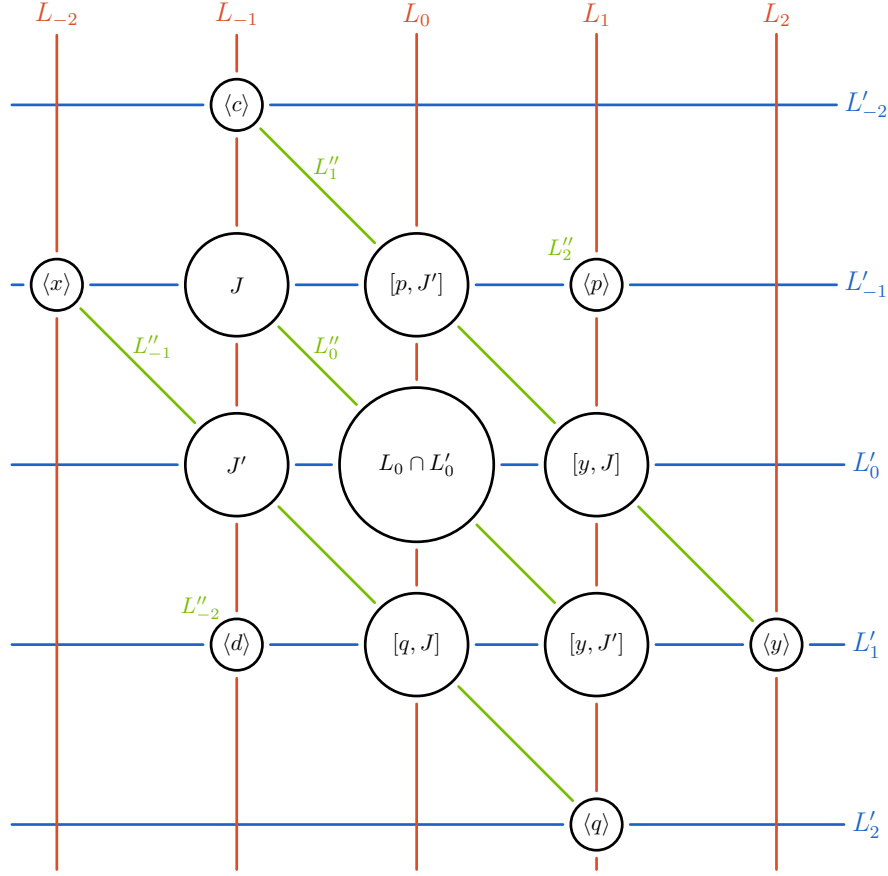

Now that we have a complete understanding of how the different gradings intersect, we are ready to invoke our results from \cref{se:l-exp}.
The next result on the uniqueness of certain extremal elements will play a key role in defining the maps involved in the cubic norm structure.

\begin{proposition}\label{recover:determine J - unique}
    \begin{enumerate}
        \item \label{J1}
        	For every $a\in J$, there is a unique extremal element of the form
        	\[ a_{-2}+a_{-1}+[a,q]+q,\]
        	with $a_{-2}\in L_{-2}$ and $a_{-1}\in J'$.
        	This element is equal to $\alpha(q)$ for a unique $a$\dash exponential automorphism $\alpha \in E_-(x,y)$.
        \item \label{J2}
        	For every $b\in J'$, there is a unique extremal element of the form
        	\[ b_{-2}+b_{-1}+[b,p]+p,\]
        	with $b_{-2}\in L_{-2}$ and $b_{-1}\in J'$.
        	This element is equal to $\beta(q)$ for a unique $b$\dash exponential automorphism $\beta \in E_-(x,y)$.
    \end{enumerate}
\end{proposition}
\begin{proof}
    We show \cref{J1}; the proof of \cref{J2} is completely similar.
	By \cref{th:alg}\cref{th:alg:l-exp}, there exists an $a$-exponential automorphism $\alpha$, so we have a corresponding extremal element
	\[ \alpha(q)=a_{-2}+a_{-1}+[a,q]+q , \]
	with $a_{-1}\in L_{-1}$ and $a_{-2}\in L_{-2}$.
	Notice that $\alpha(d) = d + [a,d] = d$ because $[J,d] = 0$ by \cref{recover:L1 identity 2}.
	By \cref{pr:5gr}\cref{5gr:indep}, this implies that $\alpha$ preserves the filtration $L''_{\leq i}$, so in particular, $\alpha(q) \in L''_{\leq -1}$.	
	Hence $a_{-1}\in L_{-1} \cap L''_{\leq -1} = \langle d\rangle\oplus J'$ (see \cref{fig:cns}).
	
	Now, since $[x,q]=[x,-[y,d]]=-[[x,y],d]=d$, there is a (unique) $\lambda\in k$ such that $\exp(\lambda x)(\alpha(q))$ has its $L_{-1}$-component in $J'$.
	By replacing $\alpha$ with $\exp(\lambda x)\alpha$, we find the required $a$-exponential automorphism and the required extremal element.

    We now show uniqueness, so suppose that $e = b_{-2} + b_{-1} + [a,q] + q$ is an extremal element such that $b_{-2}\in L_{-2}$ and $b_{-1}\in J'$.
    Let $\beta := \alpha_{-1}$ as defined in \cref{th:alg}\cref{th:alg:scalar}, so $\beta$ is a $(-a)$-exponential automorphism.
    In particular, $f := \beta(e)$ has trivial $L_0$-component.
	Moreover, the $L_{-1}$-component of $f$ is equal to $b_{-1} + a_{-1} - [a,[a,q]]$.
	By assumption, $a_{-1}, b_{-1} \in J'$ and by \cref{recover:L1 identity 4}, also $[a,[a,q]]\in J'$, so the $L_{-1}$-component of $f$ is contained in $J'\leq L'_0$. Hence
	\[ f = \mu x + c_{-1} + q, \quad \text{with } \mu \in k \text{ and } c_{-1} \in J' , \]
	is an extremal element.
	By \cref{recover:uniqueness} applied to the grading \cref{recover:other grading}, we get $\mu = 0$ and $f = \exp(\lambda c)(q)$ for certain $\lambda\in k$.
	Considering the $L_0$-components of these extremal elements yields $0=\lambda [c,q]$ and thus $\lambda=0$ and $q = f = \beta(e)$.
	
	Next, since both $\beta^{-1}$ and $\alpha$ are $a$-exponential automorphisms, \cref{th:alg}\cref{th:alg:unique} implies $\beta^{-1} = \exp(\nu x)\alpha$ for some $\nu \in k$, so in particular, $e = \exp(\nu x)\alpha(q)$.
	However, $[x,q] = d \not\in J'$, so because the $L_{-1}$-components of both $e$ and $\alpha(q)$ are contained in $J'$, this can only happen if $\lambda=0$, and therefore $e = \alpha(q)$, proving the required uniqueness of the extremal element.

    Notice that the uniqueness of the corresponding automorphism $\alpha \in E_-(x,y)$ is now clear, because any other $\alpha' \in E_-(x,y)$ with $\alpha'(q) = \alpha(q)$ would again be $a$\dash exponential (for the same $a$) and therefore differ from $\alpha$ by an element of the form $\exp(\lambda x)$, forcing $\lambda = 0$ as before.
\end{proof}

%
%

As promised, we can now use \cref{recover:determine J - unique} to define some maps.

\begin{definition}\label{recover:define norm cross etc}
	We define maps $N \colon J\to k$, $\sharp \colon J\to J'$, $T \colon J\times J'\to k$ and $\times \colon J\times J\to J'$ as follows:
	\begin{align*}
        N(a)x &= a_{-2}, \text{ for any } a\in J \text{ with } a_{-2} \text{ as in \cref{recover:determine J - unique}\cref{J1}};\\
        a^\sharp &= a_{-1}, \text{ for any } a\in J \text{ with } a_{-1} \text{ as in \cref{recover:determine J - unique}\cref{J1}};\\
        T(a,b)x &= [a,b], \text{ for any } a\in J, b\in J';\\
        a\times b &= [a,[b,q]], \text{ for any } a, b\in J.
    \end{align*}
	Note that $T$ is a bilinear map and that $\times$ is a symmetric bilinear map because $[J,J]=0$.
	Similarly, we define maps $N' \colon J'\to k$, $\sharp' \colon J'\to J$, $T' \colon J'\times J \to k$ and $\times' \colon J'\times J'\to J$:
	\begin{align*}
        N'(b)x &= b_{-2}, \text{ for any } b\in J' \text{ with } b_{-2} \text{ as in \cref{recover:determine J - unique}\cref{J2}}; \\
        b^{\sharp'} &= -b_{-1}, \text{ for any } b\in J' \text{ with } b_{-1} \text{ as in \cref{recover:determine J - unique}\cref{J2}}; \\
        T'(b,a)x &= [a,b], \text{ for any } a\in J, b\in J';\\
        a\times' b &= -[a,[b,p]], \text{ for any } a, b\in J'.
    \end{align*}
    (Notice the two minus signs. Observe that $T'(b,a) = T(a,b)$ for all $a \in J$, $b \in J'$.)
\end{definition}

Using \cref{recover:determine J - unique}, we can deduce some identities.
\begin{lemma}
	Let $a,b\in J$ and $\lambda\in k$. Then
	\begin{align}
		(\lambda a)^\sharp &=\lambda^2a^\sharp ; \label{recover:cubic axioms 1} \\
		N(\lambda a)	   &=\lambda^3 N(a) ;\label{recover:cubic axioms 2}\\
		(a+b)^\sharp 	   &=a^\sharp+a\times b+b^\sharp ;\label{recover:cubic axioms 3}\\
		N(a+b)			   &=N(a)+T(b,a^\sharp)+T(a,b^\sharp)+N(b);\label{recover:cubic axioms 4}\\
		(a^\sharp)^{\sharp'}&=N(a)a;\label{recover:cubic axioms 5}\\
		N'(a^\sharp) &= N(a)^2\label{recover:cubic axioms 6}.
	\end{align}
	Similar identities hold for $a,b \in J'$.
\end{lemma}
\begin{proof}
	Equations \cref{recover:cubic axioms 1,recover:cubic axioms 2} follow immediately from \cref{recover:determine J - unique} and \cref{th:alg}\cref{th:alg:scalar}.

	Now let $a,b \in J$ and consider the corresponding automorphisms $\alpha$ for $a$ and $\beta$ for $b$ as in \cref{recover:determine J - unique}.
	Then $\alpha(q)=N(a)x+a^\sharp+[a,q]+q$ and $\beta(q)=N(b)x+b^\sharp+[b,q]+q$.
	The elements $\alpha\beta(q)$ and $\beta\alpha(q)$ are also extremal, and
	\begin{align*}
		 \alpha\beta(q) &= \bigl( N(b)x+T(a,b^\sharp)x+q_{\alpha}([b,q])+N(a)x \bigr) \\
		 				         &\hspace*{24ex} + (b^\sharp+a\times b+a^\sharp )+[b+a,q]+q, \\
		 \beta\alpha(q) &= \bigl( N(a)x+T(b,a^\sharp)x+q_{\beta}([a,q])+N(b)x \bigr) \\
		 				         &\hspace*{24ex} + (a^\sharp+b\times a+b^\sharp )+[a+b,q]+q.
	\end{align*}
	Since $a^\sharp + a\times b + b^\sharp \in J'$, we can use the uniqueness in \cref{recover:determine J - unique} applied on $a+b \in J$ to obtain \cref{recover:cubic axioms 3} and
	\begin{align*}
		N(a+b)x&=N(a)x+T(b,a^\sharp)x+q_{\beta}([a,q])+N(b)x\\
			  &=N(a)x+T(a,b^\sharp)x+q_{\alpha}([b,q])+N(b)x.
	\end{align*}
	In particular, we have
	\begin{align}\label{recover:N(a+b)}
	   T(a,b^\sharp)x - q_{\beta}([a,q]) = T(b,a^\sharp)x - q_{\alpha}([b,q]).
	\end{align}
	Since $b\in J$ is arbitrary, we can replace it by $\lambda b$ for arbitrary $\lambda\in k$.
	By \cref{th:alg}\cref{th:alg:scalar}, \cref{recover:cubic axioms 1}, and the linearity of $T$, the left hand side of \cref{recover:N(a+b)} is quadratic in~$\lambda$, while the right hand side is linear in $\lambda$.
	Since $|k|\geq 3$, this implies that both the left and the right hand side of \cref{recover:N(a+b)} are $0$.
	So we obtain \cref{recover:cubic axioms 4}.

	We now show \cref{recover:cubic axioms 5} and \cref{recover:cubic axioms 6} simultaneously. Let $a\in J$ be arbitrary and	consider the corresponding extremal element
	$e = N(a)x + a^\sharp + [a,q] + q$.
	Since $a^\sharp\in J' \leq L_0'$ and $p \in L'_{-1}$, we get $g(p, a^\sharp) = 0$ by \cref{recover:extr in 1 prop}.
	Notice that $e \in L''_{-1}$, so $(e,p) \in E_1$ because of the grading \cref{recover:other grading 2}.
	It follows that
    \[ [e, p] = -N(a)c + [a^\sharp,p] + [[a,q],p] + y \in E . \]
    Notice  that $[[a,q],p] = [a,y]$ by the Jacobi identity, since $[a,p] \in [J,p] = 0$.
	Applying the automorphism $\varphi$ from \cref{recover:switching grading}, we get
	\[ x + [x,[a,y]] + \varphi([a^\sharp,p]) - N(a)p \in E . \]
	Recall that $\varphi(l_0) = l_0 + [x, [y, l_0]]$ for all $l_0 \in L_0$.
	However, $[y, [a^\sharp, p]] = 0$ because of the grading (see again \cref{fig:cns}), so $\varphi([a^\sharp,p]) = [a^\sharp, p]$. Combined with $[x, [a,y]] = a$ by \cref{pr:5gr}\cref{5gr:iso}, we get
	\[ x + a + [a^\sharp,p] - N(a)p \in E . \]

	First assume $N(a)\neq 0$.
	Applying the automorphism $\varphi_\lambda$ from \cref{recover:autom of grading} with $\lambda = -N(a)^{-1}$, we get
	\[ N(a)^2 x - N(a) a + [a^\sharp,p] + p \in E . \]
	By definition of $\sharp'$, this implies $(a^\sharp)^{\sharp'}=N(a)a$ and $N'(a^\sharp) = N(a)^2$, so we obtain \cref{recover:cubic axioms 5,recover:cubic axioms 6} if $N(a)\neq 0$.

	Now assume $N(a)=0$. If  $a^\sharp=0$, then \cref{recover:cubic axioms 5,recover:cubic axioms 6} are trivially satisfied, so assume $a^\sharp\neq 0$. Then
	\[ e = a^\sharp + [a,q] + q . \]
	In particular, $e \in L_{\geq -1} \setminus L_{\geq 0}$, so by \cref{pr:Ei(x)}, we get $(e, y) \in E_1$ and hence $[a^\sharp,y]=[e,y]\in E$.
	It follows that also $a^\sharp = [x,[a^\sharp,y]] \in E$.
	Since $g(p, a^\sharp) = 0$, it follows that $\exp(a^\sharp)(p) = [a^\sharp,p] + p \in E$, and thus $(a^\sharp)^{\sharp'}=0$ and $N'(a^\sharp)=0$, proving  \cref{recover:cubic axioms 5,recover:cubic axioms 6} also in this case.
\end{proof}

\begin{lemma}
	For all $a,b,c\in J$ we have
	\begin{align}
		\label{recover:cubic P6}T(a,a^\sharp)&=3N(a);\\
		\label{recover:cubic P7}a\times a   &=2a^\sharp;\\
		\label{recover:cubic P8}T(c,a\times b)&=T(a,b\times c);\\
		\label{recover:cubic P9}a^\sharp\times'(a\times b)&=N(a)b+T(b,a^\sharp)a;\\
		\label{recover:cubic P10}a^\sharp\times'b^\sharp &=-(a\times b)^{\sharp'}+T(b,a^\sharp)b+T(a,b^\sharp)a.
	\end{align}
	Similar identities hold for $a,b,c \in J'$.
\end{lemma}
\begin{proof}
	By \cref{recover:cubic axioms 1,recover:cubic axioms 2,recover:cubic axioms 3,recover:cubic axioms 4,recover:cubic axioms 5,recover:cubic axioms 6}, this follows in exactly the same way as in \cite[(15.16), (15.18)]{Tits2002}. (Recall from \cref{ass:lines} that we assume $|k| \geq 4$.)
\end{proof}

\begin{remark}\label{rem:twin}
	The identities \cref{recover:cubic axioms 1,recover:cubic axioms 2,recover:cubic axioms 3,recover:cubic axioms 4,recover:cubic axioms 5,recover:cubic axioms 6,recover:cubic P6,recover:cubic P7,recover:cubic P8,recover:cubic P9,recover:cubic P10} we have shown so far may serve as an axiom system for ``cubic norm pairs'' $(J, J')$.
	Although this seems to be a natural counterpart for cubic norm structures in the context of \emph{Jordan pairs} (rather than Jordan algebras), this notion seems to have appeared only once in the literature in a paper by John Faulkner \cite{Faulkner2001}, although the setup of a paired structure also appears already in \cite{Springer1962}.
	(We thank Holger Petersson for pointing this out to us.)
	A recent contribution to the theory of cubic norm pairs, also related to $G_2$-graded Lie algebras, is Michiel Smet's preprint \cite{Smet2025}.
\end{remark}

At this point, we have, in principle, gathered enough information to \emph{reconstruct} the Lie algebra from our data.
Since this is not our main focus, we will only sketch the procedure.
We will need the following concrete computations, providing us with information about the Lie bracket with the ``middle part'' $L_0 \cap L'_0$, which is the most difficult piece to control.
(Notice that $[J, [y, J']]$ and $[J', [y, J]]$ belong to this middle part.) 

\begin{lemma}
\label{recover:Lie bracket explicit}
	Let $a,b,e \in J$ and $a',b',e' \in J'$ be arbitrary. Then
	\begin{align}
		\label{recover:bracket 3 J 1} [a,[b,[y,e]]] &= T(a, b \times e)c ; \\
		\label{recover:bracket 3 J 5} [a',[b',[y,e']]] &= T(b' \times' e', a')d ; \\
		\label{recover:bracket 3 J 4} [a,[b',[y,e']]] &= - a \times (b' \times' e') ; \\
		\label{recover:bracket 3 J 6} [a',[b,[y,e]]] &= a' \times' (b \times e) ; \\
		\label{recover:bracket 3 J 3} [a,[b',[y,e]]] &= b' \times' (a \times e) - T(a, b')e ; \\
		\label{recover:bracket 3 J 8} [a',[b,[y,e']]] &= - b \times (a' \times' e') + T(b, a')e ; \\
		\label{recover:bracket 3 J 2} [a,[b,[y,e']]] &= e' \times' (a \times b) - T(a, e')b - T(b, e')a ; \\
		\label{recover:bracket 3 J 7} [a',[b',[y,e]]] &= - e \times (a' \times' b') + T(e, a')b' + T(e, b')a' .
	\end{align}
\end{lemma}
\begin{proof}
	By the grading (see \cref{fig:cns}), we have $[a,[b,[y,e]]] \in L_{-1} \cap L'_{-2} = \langle c \rangle$, so $[a,[b,[y,e]]] = \lambda c$ for some $\lambda \in k$.
	Since $[c,d]=x$ and $[J,d]=0$, we get
	\[ \lambda x = [d,[a,[b,[e,y]]]] = [a,[b,[e,[d,y]]]] = [a,[b,[e,q]]] = [a, b\times e] = T(a, b \times e)x \]
	and hence \cref{recover:bracket 3 J 1} holds.
	The proof of \cref{recover:bracket 3 J 5} is similar.
	
	Next, since $y = [q,p]$, $[q, e'] = 0$ and $[q, b'] = 0$, we get
	\begin{align*}
	   [a, [b', [y, e']]]
	   & = [a, [b', [[q, p], e']]] \\
	   & = [a, [b', [q, [p, e']]]] \\
	   & = [a, [q, [b', [p, e']]]] = - a \times (b' \times' e'),
	\end{align*}
	proving \cref{recover:bracket 3 J 4}.
	The proof of \cref{recover:bracket 3 J 6} is similar.

    Now, by the Jacobi identity and \cref{recover:bracket 3 J 6}, we have
	\[ [a, [b', [y, e]]] = [b', [a, [y,e]]] + [[a,b'], [y,e]] = b' \times' (a \times e) + [T(a,b')x, [y,e]], \]
	so we get \cref{recover:bracket 3 J 3}.
	The proof of \cref{recover:bracket 3 J 8} is similar.
    
	Finally, by applying the Jacobi identity twice and using $[a,[b,e']] \in L_{-3}=0$ and~\cref{recover:bracket 3 J 3}, we get
	\begin{align*}
        [a,[b,[y,e']]]
            &= [y,[a,[b,e']]] + [a,[[b,y],e']] + [[a,y],[b,e']] \\
            &= 0 + [a, [e', [y, b]]] + [[a,y], T(b,e')x] \\
            &= e' \times' (a \times b) - T(a, e')b - T(b, e')a,
	\end{align*}
	showing \cref{recover:bracket 3 J 2}.
	The proof of \cref{recover:bracket 3 J 7} is similar.
\end{proof}

Here is the promised sketch of the reconstruction result.
\begin{corollary}\label{co:cns reconstruct}
    The Lie bracket on $L$ can be completely recovered from the maps $T$, $\times$ and $\times'$ alone.
\end{corollary}
\begin{proof}
    Notice that, by \cref{fig:cns}, the algebra $L$ has a decomposition into $13$ pieces.
    Now let
    \[ K_0 := \langle [x,y], [q,c] \rangle \oplus [J, [y, J']] , \]
    and let $K$ be the subspace of $L$ spanned by $K_0$ together with the $12$ remaining pieces of $L$.
    It is not difficult to check (but it requires some case analysis) that $K$ is an ideal of $L$.
    Since $L$ is simple, it follows that $K = L$, so in particular, $K_0 = L_0 \cap L'_0$.
    Now observe that we can identify $K_0$ with the Lie algebra $\tilde K_0$ consisting of the corresponding inner derivations $\ad_l$ (for each $l \in K_0$) restricted to the sum of the $12$ remaining pieces; see, e.g., \cite[Construction 4.1.2]{Boelaert2019}. (Since $L$ is simple, the corresponding Lie algebra homomorphism $K_0 \to \tilde K_0$ is injective and hence an isomorphism.)
    
	The Lie algebra $L$ can now be reconstructed from six one-dimensional subspaces, three copies of $J$, three copies of $J'$, and a copy of $\tilde K_0$, assembled exactly as in \cref{fig:cns}. It is now a routine (but somewhat lengthy) verification, relying on \cref{recover:Lie bracket explicit}, that that the Lie bracket between every two of the $13$ pieces is completey determined by the maps $T$, $\times$ and $\times'$.
\end{proof}

So far, we have been able to endow the pair $(J, J')$ with a ``twin cubic norm structure'' using the various maps between them.
(See also \cref{rem:twin}.)
When the norm $N$ is not identically zero, we can go one step further, by selecting a ``base point'' to identify the two parts of this twin structure, which then results in a genuine cubic norm structure.

\begin{lemma}\label{le:cns:basept}
	If $N\neq 0$, then we can re-choose $c$ and $d$ in \cref{recover:notation e12} in such a way that $N(z)=1$ for some $z \in J$.
	We call $z$ a \emph{base point} for $N$.
\end{lemma}
\begin{proof}
	Let $z\in J$ such that $N(z)\neq 0$.
	If we replace $c$ by $\tilde c = N(z)c$ and $d$ by $\tilde d = N(z)^{-1}d$ in \cref{recover:notation e12}, and correspondingly replace $p$ by $\tilde p = N(z)p$ and $q$ by $\tilde q = N(z)^{-1}q$, then all conditions in \cref{recover:notation e12} remain satisfied.
	Since we have only replaced these elements by a scalar multiple, this does not affect any of the decompositions.
	
	Now denote the norm obtained from \cref{recover:define norm cross etc} with these new choices by~$\tilde N$.
	Since $N(z)x+z^\sharp+[z,q]+q\in E$, we get $x+N(z)^{-1}z^\sharp+[z, \tilde q] + \tilde q\in E$ and it follows that indeed $\tilde N(z)=1$.
\end{proof}

Assume for the rest of this section that $N\neq 0$ and that $z \in J$ is a base point for~$N$.
\begin{definition}
\label{recover:def sigma 12}
	We define maps $\sigma \colon J\to J'$ and $\sigma' \colon J'\to J$ by
	\begin{align*}
	   \sigma(a) &= T(a,z^\sharp)z^\sharp - z\times a, \\
	   \sigma'(b) &= T(z,b)z - b \times' z^\sharp,
	\end{align*}
	for all $a \in J$ and all $b \in J'$.
\end{definition}

\begin{lemma}
\label{recover:sigma are inverse}
	We have $\sigma'\circ\sigma=\id_J$ and $\sigma\circ\sigma'=\id_{J'}$.
\end{lemma}
\begin{proof}
	Let $a \in J$ be arbitrary.
	Then by \cref{recover:cubic axioms 5,recover:cubic P6,recover:cubic P7,recover:cubic P8,recover:cubic P9} and since $N(z)=1$,
	\begin{align*}
		\sigma'(\sigma(a))&= T(z,T(a,z^\sharp)z^\sharp)z-T(z,z\times a)z-T(a,z^\sharp)z^\sharp\times'z^\sharp+ (z\times a) \times' z^\sharp \\
							 &= T(a,z^\sharp)T(z,z^\sharp)z-T(a,z\times z)z-2T(a,z^\sharp)(z^\sharp)^{\sharp'} +N(z)a+T(a,z^\sharp)z \\
							 &= 3T(a,z^\sharp)z-2T(a,z^\sharp)z-2T(a,z^\sharp)z +a+T(a,z^\sharp)z \\
							 &= a.
	\end{align*}
	The proof of the second equality is similar.
\end{proof}

\begin{lemma}
\label{recover:sharp commutes}
	We have $\sigma'\circ\sharp=\sharp'\circ\sigma$.
\end{lemma}
\begin{proof}
	Let $a \in J$ be arbitrary. Then using \cref{recover:cubic P10}, we get
	\begin{align*}
		\sigma'(a^\sharp)&=T(z,a^\sharp)z-a^\sharp\times' z^\sharp \\
						  &=T(z,a^\sharp)z+(a\times z)^{\sharp'}-T(z,a^\sharp)z-T(a,z^\sharp)a \\
						  &=(a\times z)^{\sharp'}-T(a,z^\sharp)a.
	\end{align*}
	On the other hand,
	using $N(z)=1$, \cref{recover:cubic axioms 1,recover:cubic axioms 3,recover:cubic axioms 5,recover:cubic P9}, we have
	\begin{align*}
		\sigma(a)^{\sharp'}&=(T(a,z^\sharp)z^\sharp-z\times a)^{\sharp'} \\
							 &=T(a,z^\sharp)^2(z^\sharp)^{\sharp'}-T(a,z^\sharp)z^\sharp\times'(z\times a)+(z\times a)^{\sharp'} \\
							 &=T(a,z^\sharp)^2z-T(a,z^\sharp)a-T(a,z^\sharp)^2z+(z\times a)^{\sharp'}\\
							 &=(a\times z)^{\sharp'}-T(a,z^\sharp)a.
							 \qedhere
	\end{align*}
\end{proof}

We are now ready to define the cubic norm structure and prove the required defining identities; see \cref{prelim:def hex}.

\begin{definition}
\label{recover:hex system constr}
	We define maps $T_J:J\times J\to k$, $\times_J:J\times J\to J$ and $\sharp_J:J\to J$ by setting
	\begin{align}
		T_J(a,b) &= T(a,\sigma(b)),\\
		a\times_J b &= \sigma'(a\times b) = \sigma(a) \times' \sigma(b), \label{eq:hex2} \\
		a^{\sharp_J} &= \sigma'(a^\sharp) = \sigma(a)^{\sharp'}, \label{eq:hex3}
	\end{align}
	for all $a,b \in J$. (The rightmost equalities in \cref{eq:hex2,eq:hex3} hold by \cref{recover:sharp commutes}.)
\end{definition}

\begin{theorem}
\label{recover:hex system prop}
	Let $L$ be as in \cref{recover:notation e12} and $J$ be as in \cref{recover:notation J}.
	Assume that the map $N$ defined in \cref{recover:define norm cross etc} is not identically zero.
	Then the data $(J,k,N,\sharp_J,\allowbreak T_J,\times_J,z)$ forms a non-degenerate cubic norm structure.
\end{theorem}
\begin{proof}
	Recall that by \cref{prelim:hex less cond}, we only have to show that the defining identities \cref{prelim:hex axiom 1,prelim:hex axiom 2,prelim:hex axiom 4,prelim:hex axiom 5,prelim:hex axiom 7,prelim:hex axiom 10,prelim:hex axiom 11} of \cref{prelim:def hex} are satisfied.

	First, note that $\times_J$ is symmetric and bilinear by construction.
	For $T_J$, we have
	\begin{align*}
		T_J(a,b)&=T(a,T(b,z^\sharp)z^\sharp)-T(a,z\times b)=T(a,z^\sharp)T(b,z^\sharp)-T(z,a\times b)
	\end{align*}
	by \cref{recover:cubic P8}, and hence $T_J$ is symmetric and bilinear.

	Identity \cref{prelim:hex axiom 1} follows from the linearity of $\sigma'$ and \cref{recover:cubic axioms 1}.
	Identity \cref{prelim:hex axiom 2} is precisely~\cref{recover:cubic axioms 2}.
	Identity \cref{prelim:hex axiom 4} follows from the linearity of $\sigma'$ and \cref{recover:cubic axioms 3}.
	Identity \cref{prelim:hex axiom 5} follows from \cref{recover:sigma are inverse,recover:cubic axioms 4}.
	Identity \cref{prelim:hex axiom 7} follows from \cref{recover:sigma are inverse,recover:sharp commutes,recover:cubic axioms 5}.
	Identity \cref{prelim:hex axiom 10} follows from
	\[ \sigma'(z^\sharp)=T(z,z^\sharp)z-z^\sharp\times'z^\sharp=3N(z)z-2N(z)z=z, \]
	using \cref{recover:cubic P6,recover:cubic P7,recover:cubic axioms 5} and $N(z)=1$.
	Identity \cref{prelim:hex axiom 11} follows from
	\begin{align*}
		z \times_J a=\sigma'(z\times a)&=T(z,z\times a)z-z^\sharp \times' (z\times a) \\
						   &=T(a,z\times z)z-N(z)a-T(a,z^\sharp)z \\
						   &=T(a,z^\sharp)z-a=T_J(a,z)-a
	\end{align*}
	for all $a \in J$,
	using \cref{recover:cubic P7,recover:cubic P8,recover:cubic P9}, $\sigma(z)=z^\sharp$ and $N(z)=1$.
	Finally, the cubic norm structure is non-degenerate: if $a \in J$ is such that $T_J(a, J) = 0$, then we have $T(a, J') = 0$, i.e., $[a, J'] = 0$, but it then follows from the grading that $[a, L_{-1}] = 0$; \cref{recover:T nondeg} then implies that $a=0$.
\end{proof}


\begin{remark}
    A different choice of base point would give rise to a different cubic norm structure, but it is not hard to show (and not surprising) that the resulting cubic norm structures are \emph{isotopic}; see \cref{prelim:def isot}.
\end{remark}

\begin{remark}\label{rem:MQ}
    If the cubic norm structure is \emph{anisotropic} (see \cref{def:cns anis}), then the extremal geometry is, in fact, a generalized hexagon. (The key point here is the fact that the extremal geometry is a generalized hexagon if and only if there are no symplectic pairs, a fact that we have exploited already in the proof of \cref{recover: y and I1}. The existence of symplecta turns out to be equivalent to the existence of \emph{extremal} elements in $L'_{-1} \cap L_{-1} = J$, and such elements always have norm $0$.)
    
    With some more effort, we can use the techniques developed in this section to show that this generalized hexagon is isomorphic to the Moufang hexagon obtained from the cubic norm structure as in \cite[(16.8)]{Tits2002} and we can determine the root groups and the commutator relations explicitly in terms of the Lie algebra.
    Since this result is not so surprising but still requires a substantial amount of work, we omit the details. See \cite[\S 4.4.2]{MeulewaeterPhD}.
\end{remark}

\begin{remark}
    If the map $N$ is identically zero, then also the maps $\sharp$, $\times$, $N'$, $\sharp'$ and~$\times'$ are identically zero.
    In this case, it can be shown that the extremal geometry is isomorphic to a geometry of the form $\Gamma(V, W^*)$ as defined in \cite[Example~2.8]{Cuypers2021}.
    Note that these examples also occur when the Lie algebra is infinite-dimensional. In the finite-dimensional case, these are precisely so-called \emph{root shadow spaces} of type $A_{n, \{1,n\}}$.
\end{remark}

\section{Extremal geometry with symplectic pairs --- recovering a quadrangular algebra}
\label{se:symplectic}

In this section, we continue to assume that $L$ is a simple Lie algebra generated by its pure extremal elements, but we make different assumptions on the existence of certain extremal elements and extremal lines.

\begin{assumption}
	\label{ass:quadr}
	We assume that $L$ is a simple Lie algebra, defined over a field $k$ with $|k|>2$, such that $L$ is generated by its set $E$ of pure extremal elements, and such that:
	\begin{enumerate}
		\item \label{ass:quadr i} there exists a Galois extension $k'/k$ of degree at most $2$ such that the extremal geometry of $L\otimes k'$ contains lines;
		\item \label{ass:quadr ii} there exist symplectic pairs of extremal elements.
	\end{enumerate}
\end{assumption}

This time, we will show that $L_{-1}$ can be decomposed as $V \oplus X \oplus V'$ into $3$ parts, where $V' \cong V$, and that $(V,X)$ can be made into a quadrangular algebra.
This will require a substantial amount of work, mostly because we are including the case $\Char(k)=2$.
Notice, however, that we will never require a case distinction, not even in the computations in our proofs.
For convenience, we have included the $\Char(k)\neq 2$ description of $\theta$ in \cref{pr:theta char not 2}, but we never use this result.

%
%
%

We first point out that $L$ remains simple after base extension to $k'$.
\begin{lemma}\label{le:simple'}
    The Lie algebra $L_{k'} := L \otimes_k k'$ is simple.
\end{lemma}
\begin{proof}
    Essentially, this is shown in the proof of \cite[Theorem 5.7]{Cuypers2021}.
    However, the statement in \emph{loc.\@ cit.\@} assumes $\Char(k) \neq 2$, and extending it to all characteristics requires some subtle changes in the final step, so we give a complete proof.
    
    Let $\sigma$ be the unique non-trivial element of $\Gal(k'/k)$.
    We first recall%
    \footnote{This fact is reproven in \cite[Lemma 6.9]{Cuypers2021} in this specific situation with an ad-hoc argument, but holds in a much more general setting.}
    that, by \emph{Speiser's Lemma} (see, e.g., \cite[Lemma 2.3.8]{GS17}), each $\sigma$-invariant subspace $W$ of $L_{k'}$ can be written as $W = V \otimes_k k'$ for some subspace $V$ of $L$.
    
    Now let $I$ be a non-trivial ideal of $L_{k'}$ and let $a \in I$ with $a \neq 0$.
    Then, as the extremal elements in $E$ linearly span $L$, they also span $L_{k'}$ over $k'$, so we can write
    \[  a = x_1 \otimes \lambda_1+\cdots+x_k\otimes \lambda_k \]
    where $x_j \in E$ are linearly independent and each $\lambda_j \in k'^\times$.
    After replacing $a$ with a scalar multiple, we can assume $\lambda_1+\lambda_1^\sigma\neq 0$.
    Then $a+a^\sigma$ is a non-trivial element in $I+I^\sigma$ fixed by $\sigma$.
    Therefore, the subspace spanned by the elements in $I+I^\sigma$  fixed by $\sigma$ forms a non-trivial ideal of $L$, which by simplicity of $L$ equals $L$.
    This implies that  $I+I^\sigma = L_{k'}$.
    
    Since both $I$ and $I^\sigma$ are ideals, $I\cap I^\sigma$ is an ideal which is stabilized by $\sigma$.
    By Speiser's Lemma, $I \cap I^\sigma = V \otimes_k k'$ for some subspace $V$ of $L$, but then $V$ is an ideal of $L$.
    By the simplicity of $L$, this implies that either $I = L_{k'}$, or $I\cap I^\sigma=0$.
    So we may assume the latter, and therefore $L_{k'} = I \oplus I^\sigma$.
    In particular, $[I,I^\sigma] \leq I \cap I^\sigma = 0$ and thus
    \begin{equation}\label{eq:II}
        [L_{k'}, L_{k'}] = [I,I]\oplus [I^\sigma,I^\sigma] .
    \end{equation}
    Since $L$ is simple, $L=[L,L]$ and hence $L_{k'} = [L_{k'}, L_{k'}]$, so it follows from \cref{eq:II} that $[I,I]=I$.
    
    Now let $x\in E$ be arbitrary.
    By the Premet identity \cref{P1} with $y,z \in I$, we get $g_x([y,z])x \in I$ for all $y,z \in I$.
    Since $I = [I,I]$, it follows that $g_x(a) x \in I$ for all $a \in I$.
    It thus follows that $x \in I$ if we can find an $a \in I$ with $g(x,a) \neq 0$.
    This is indeed always possible: since $g$ is non-degenerate, there exists a $y \in L$ with $g(x,y) \neq 0$, but since $L_{k'} = I + I^\sigma$, we either find $a \in I$ with $g(x,a) \neq 0$ or we find $b \in I^\sigma$ with $g(x,b) \neq 0$. In the latter case, however, we simply apply $\sigma$, and we find $g(x, b^\sigma) \neq 0$ with $b^\sigma \in I$.
    
    Since $x\in E$ was arbitrary, we have shown that $E$ is contained in $I$.
    However, $E$ generates $L_{k'}$, so we get that $I = L_{k'}$ as required.
    (In fact, this is now a contradiction to the fact that $I \cap I^\sigma = 0$.)
\end{proof}

\begin{proposition}
\label{recover quadr:existence ext elements}
	There exist extremal elements $x$, $y$, $c$ and $d$ such that $g(x,y)=g(c,d)=1$ and $(x,c),(c,y),(y,d),(d,x) \in E_0$.
\end{proposition}
\begin{proof}
	Assume first that the extremal geometry of $L$ contains lines.
	By \cref{ass:quadr}\cref{ass:quadr ii}, there exists an $x\in E$ and a symplecton $S$ of the extremal geometry containing $\langle x\rangle$.
	Recall from \cref{ass:simple pure} that $(\E, \E_2)$ is connected, so there exist $y\in E$ such that $g(x,y)=1$.
	By \cref{pr:symplecta}\cref{pr:symplecta:e}, there exists $c\in E$ such that $\langle c\rangle \in S$ and $(c,y)\in E_0$.
	Since $S \subseteq \E_{\leq 0}(x)$, \cref{pr:RFS}\cref{pr:RFS:2} implies that also $(x,c)\in E_0$.
	
	By \cref{pr:RFS}\cref{pr:RFS:0} applied on $(x,c)$, there exists a symplecton $T$ containing $x$ and containing an extremal point hyperbolic with $\langle c\rangle$.
	By \cref{pr:symplecta}\cref{pr:symplecta:e} again, we find $d\in E$ such that $\langle d\rangle \in T$ with $(d,y) \in E_0$, and again, \cref{pr:RFS}\cref{pr:RFS:2} implies that also $(d,x)\in E_0$.
	It remains to show that $(c,d)\in E_2$.
	Let $A := T \cap \E_{\leq 0}(c)$. By \cref{pr:RFS}\cref{pr:RFS:2}, $A = T \cap \E_{0}(c)$.
	By \cref{pr:symplecta}\cref{pr:symplecta:e} now, $A$~consists of the single point $\langle x \rangle$.
	We can now invoke \cite[Lemma 2.17]{Cuypers2021} to see that the point $\langle d \rangle$ is indeed contained in $\E_2(c)$.
	By rescaling $c$, we obtain $g(c,d)=1$.

	Assume now that the extremal geometry of $L$ does not contain lines.
	Then $\E \times \E = \E_{-2} \cup \E_0 \cup \E_2$.
	By \cref{ass:quadr}\cref{ass:quadr i}, the extremal geometry $\Gamma' = (\E', \F')$ of $L_{k'}$ contains lines.
	Recall from \cref{le:simple'} that $L_{k'}$ is again simple, so it satisfies \cref{ass:simple pure}.
	Let $\Gal(k'/k) = \langle \sigma \rangle$. Then the involution $\sigma$ also acts on $\Gamma'$ and fixes all points of $\E$.
	
	We start now with a pair of extremal elements $(a,b) \in E_0$.
	Recall from \cref{ass:simple pure} that $(\E, \E_2)$ is connected, so there exist $e \in E$ such that $(a,e) \in E_2$.
	Notice that by \cref{pr:symplecta}, there is a unique symplecton $S$ through $a$ and $b$ in $\Gamma'$, but since $a$ and $b$ are fixed by $\sigma$, also $S$ is fixed (setwise) by $\sigma$.
	By \cref{pr:symplecta}\cref{pr:symplecta:e}, there is a unique point $\langle x \rangle \in S$ in relation $\E'_0$ with $\langle e \rangle$.
	Then also $\langle x \rangle^\sigma$ has this property, so we must have $\langle x \rangle^\sigma = \langle x \rangle$.
	By Hilbert's Theorem 90, we may replace $x$ by a $k'$-multiple of $x$ to get $x^\sigma = x$ and therefore $x \in E$.
	Notice that now $(x, a) \in E_0$ and that $S$ is the unique symplecton in $\Gamma'$ through $\langle x \rangle$ and $\langle a \rangle$.
	We also have $(x, e) \in E_0$; let $T$ be the unique sympecton in $\Gamma'$ through $\langle x \rangle$ and $\langle e \rangle$.
	Notice that $S \cap T = \langle x \rangle$.
	
	We now apply the connectedness of $(\E, \E_2)$ again to find an element $y \in E_2(x)$.
	Then again, there is a unique point $\langle c \rangle \in S$ in relation $\E'_0$ with $y$, and similarly,
	there is a unique point $\langle d \rangle \in T$ in relation $\E'_0$ with $y$.
	As before, we get $\langle c \rangle^\sigma = \langle c \rangle$ and $\langle d \rangle^\sigma = \langle d \rangle$ and we may rescale $c$ and $d$ so that $c \in E$ and $d \in E$.
	Now observe that $(c,e) \not\in E_0$, since otherwise both $\langle x \rangle$ and $\langle c \rangle$ would be in relation $\E'_0$ with~$\langle e \rangle$, contradicting \cref{pr:symplecta}\cref{pr:symplecta:e} applied on $S$ and $e$. Since $\E \times \E = \E_{-2} \cup \E_0 \cup \E_2$, this implies that $(c,e) \in E_2$. Similarly, it now follows that $(c,d) \not\in E_0$, since otherwise both $\langle x \rangle$ and $\langle d \rangle$ would be in relation $\E'_0$ with~$\langle c \rangle$, again contradicting \cref{pr:symplecta}\cref{pr:symplecta:e}, now applied on $T$ and $c$. We conclude that $(c,d) \in E_2$.
	It now only remains to rescale the elements to see that $x$, $y$, $c$ and $d$ satisfy all the required assumptions.
\end{proof}

\begin{notation}
\label{recover quadr:two gradings}
	Let $x$, $y$, $c$, $d$ be as in \cref{recover quadr:existence ext elements}.
	We denote the $5$-grading on~$L$ associated with the hyperbolic pair $(x,y)$ as in \cref{pr:5gr} by
	\begin{align}
		L_{-2}\oplus L_{-1}\oplus L_0\oplus L_1\oplus L_2, \label{recover quadr:grading 1}
	\end{align}
	with $L_{-2}=\langle x\rangle$ and $L_2 = \langle y \rangle$.
	Similarly, we denote the $5$-grading on~$L$ associated with the hyperbolic pair $(c,d)$ as in \cref{pr:5gr} by
	\begin{align}
		L'_{-2}\oplus L'_{-1}\oplus L'_0\oplus L'_1\oplus L'_2, \label{recover quadr:grading 2}
	\end{align}
	with $L'_{-2}=\langle c \rangle$ and $L'_2 = \langle d \rangle$.
	
	Now set
	\[ V := L_{-1}\cap L'_{-1}, \quad X := L_{-1}\cap L'_0, \quad V':=L_{-1}\cap L'_1, \quad X' := L_0 \cap L'_{-1}. \]
\end{notation}

\begin{lemma}
    We have $c,d \in L_0$ and $x,y \in L'_0$.
\end{lemma}
\begin{proof}
	Since $(x,c),(c,y) \in E_0$, it follows immediately from \cref{pr:Ei(x)} that $c\in L_0$.
	The proof of the other three statements is similar.
\end{proof}

\begin{lemma}
\label{recover quadr:prop g}
	We have $g_c(L_{-2}\oplus L_{-1}\oplus L_1\oplus L_2)=0$ and $g_d(L_{-2}\oplus L_{-1}\oplus L_1\oplus L_2) = 0$.
\end{lemma}
\begin{proof}
    We show the claim for $g_c$.
	Since $(c,x),(c,y) \in E_0$, we get $g_c(x)=0$ and $g_c(y)=0$.
	Next, let $l\in L_{-1}$ be arbitrary and recall from \cref{pr:5gr}\cref{5gr:iso} that  $[x,[y,l]]=-l$.
	Using the associativity of $g$, we get
	\[ -g_c(l)=g(c,[x,[y,l]])=g([c,x],[y,l])=g(0,[y,l])=0,\]
	so $g_c(L_{-1}) = 0$.
	Similarly, $g_c(L_{1}) = 0$.
\end{proof}

\begin{proposition}
\label{recover quadr:decomp in 3}
	We have decompositions
	\begin{align*}
	   L_{-1} &= V \oplus X \oplus V' , \\
	   L_{1} &= [y,V] \oplus [y,X] \oplus [y,V'], \\[1ex]
	   L'_{-1} &= V \oplus X' \oplus [y,V] , \\
	   L'_{1} &= V' \oplus [d,X'] \oplus [y,V'].
	\end{align*}
\end{proposition}
\begin{proof}
	By \cref{recover quadr:prop g}, we get $g_d(L_{-1})=0$ and hence $L_{-1} \leq L'_{\geq -1}$ by \cref{pr:5gr}\cref{5gr:g0} applied to the grading \cref{recover quadr:grading 2}.
	Similarly, it follows from $g_c(L_{-1})=0$ that $L_{-1} \leq L'_{\leq 1}$.
	Hence $L_{-1} \leq L'_{-1}\oplus L'_0\oplus L'_1$.

    The rest of the proof of the decomposition of $L_{-1}$ now follows exactly the same method as in the proof of \cref{recover:L1 decomposition}, using the grading element $[c,d]\in L_0$, so we can safely omit the details.
    The second decomposition follows from the first by applying the isomorphism $\ad_y$ from $L_{-1}$ to $L_1$; see \cref{pr:5gr}\cref{5gr:iso}.
    
    The other two decompositions follow in a similar fashion by interchanging the roles of the two gradings \cref{recover quadr:grading 1,recover quadr:grading 2}.
%
%
\end{proof}

We have summarized all the information about the intersections of the two gradings in \cref{fig:qa}.
Notice that the four subspaces $V$, $V' = [d,V]$, $[y,V]$ and $[y,V'] = [y,[d,V]]$ are all isomorphic, by \cref{pr:5gr}\cref{5gr:iso}.
(See also the proof of \cref{co:easy} below.)
For the same reason, $X \cong [y,X]$ and $X' \cong [d,X']$.
We will see later that also the subspaces $X$ and $X'$ are isomorphic; see \cref{co:XX'}.
Notice that in \cref{fig:qa}, we have also identified a third grading (diagonally), but this grading is not directly necessary for our purposes.
See, however, \cref{rem:qa 3rd grading} below.

\begin{figure}[ht!]
\[
	\scalebox{.85}{%
	\begin{tikzpicture}[x=28mm, y=28mm, 	label distance=-3pt]
		\node[myblob=7mm] at (0,2) (N02) {$\langle x\rangle$};
		\node[myblob=14mm] at (1,1) (N11) {$V'$};
		\node[myblob=20mm] at (1,2) (N12) {$X$};
		\node[myblob=14mm] at (1,3) (N13) {$V$};
		\node[myblob=7mm] at (2,0) (N20) {$\langle d\rangle$};
		\node[myblob=20mm] at (2,1) (N21) {$[d, X']$};
		\node[myblob=24mm] at (2,2) (N22) {$L_0\cap L'_0$};
		\node[myblob=20mm] at (2,3) (N23) {$X'$};
		\node[myblob=7mm] at (2,4) (N24) {$\langle c\rangle$};
		\node[myblob=14mm] at (3,1) (N31) {$[y,V']$};
		\node[myblob=20mm] at (3,2) (N32) {$[y,X]$};
		\node[myblob=14mm] at (3,3) (N33) {$[y,V]$};
		\node[myblob=7mm] at (4,2) (N42) {$\langle y\rangle$};
		\path[ugentred]
			(0,4.5) node (C0) {\Large $L_{-2}$}
			(1,4.5) node (C1) {\Large $L_{-1}$}
			(2,4.5) node (C2) {\Large $L_{0}$}
			(3,4.5) node (C3) {\Large $L_{1}$}
			(4,4.5) node (C4) {\Large $L_{2}$};
		\path[ugentblue]
			(4.5,0) node (R0) {\Large $L'_{2}$}
			(4.5,1) node (R1) {\Large $L'_{1}$}
			(4.5,2) node (R2) {\Large $L'_{0}$}
			(4.5,3) node (R3) {\Large $L'_{-1}$}
			(4.5,4) node (R4) {\Large $L'_{-2}$};
		\draw[myedge,ugentblue]
			(N11) -- (N21) -- (N31)
			(N02) -- (N12) -- (N22) -- (N32) -- (N42)
			(N13) -- (N23) -- (N33);
		\draw[myedge,ugentblue,opacity=.15]
			(-.25,0) -- (N20) -- (R0)
			(-.25,1) -- (N11) (N31) -- (R1)
			(-.25,2) -- (N02) (N42) -- (R2)
			(-.25,3) -- (N13) (N33) -- (R3)
			(-.25,4) -- (N24) -- (R4);
		\draw[myedge,ugentred]
			(N11) -- (N12) -- (N13)
			(N20) -- (N21) -- (N22) -- (N23) -- (N24)
			(N31) -- (N32) -- (N33);
		\draw[myedge,ugentred,opacity=.15]
			(0,-.25) -- (N02) -- (C0)
			(1,-.25) -- (N11) (N13) -- (C1)
			(2,-.25) -- (N20) (N24) -- (C2)
			(3,-.25) -- (N31) (N33) -- (C3)
			(4,-.25) -- (N42) -- (C4);
		\draw[myedge,ugentyellow]
			(N02) --  (N13) -- (N24)	
			(N12) --  (N23) 
			(N11) --  (N22) -- (N33)	
			(N21) --  (N32) 
			(N20) --  (N31) -- (N42)	;
	\end{tikzpicture}
	}
\]
\caption{Intersecting gradings for quadrangular algebras}\label{fig:qa}
\end{figure}
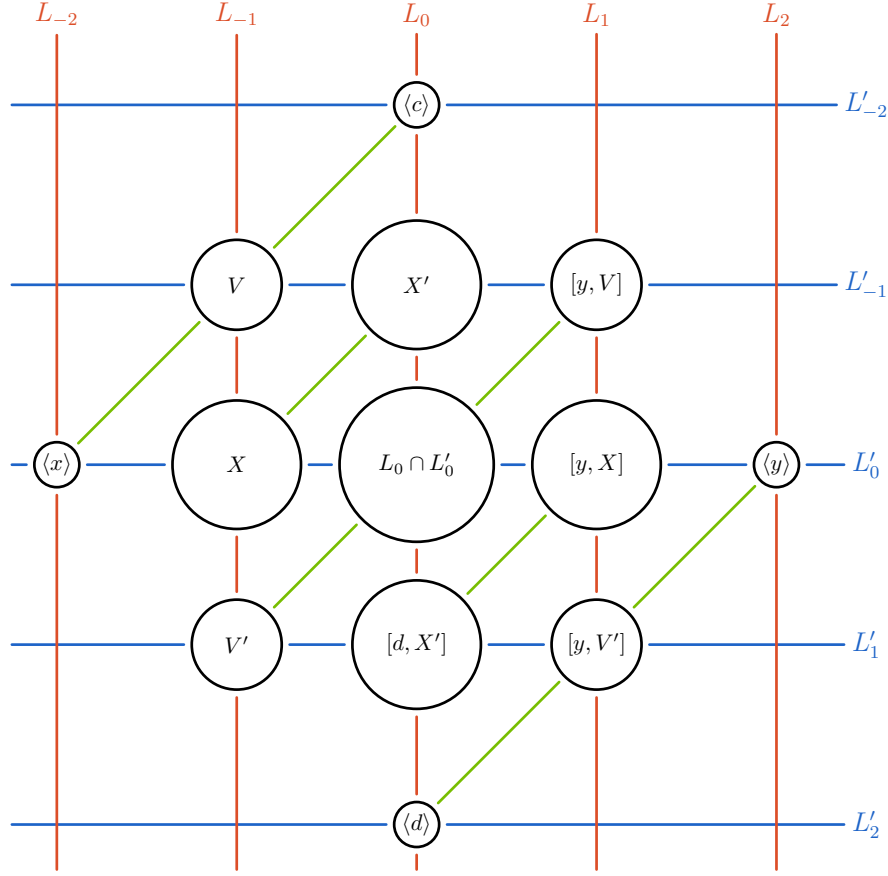

By combining both gradings, we already obtain some more information on the Lie bracket.

\begin{corollary}\label{co:easy}
	The following identities and inclusions hold:
	\begin{alignat}{2}
		\label{recover quadr:decoomp eq 1} &[c,V']=V, && [d,V]=V', \\
		\label{recover quadr:decoomp eq 5} &[c,V]=0, && [d,V']=0, \\
		\label{recover quadr:decoomp eq 6} &[c,X]=0, && [d,X]=0, \\
		\label{recover quadr:decoomp eq 2} &[V,V]=0, \quad [V',V']=0, \quad && [V,X]=0, \quad [V',X]=0, \\
		\label{recover quadr:decoomp eq 3} &[X,[X,[y,V]]\leq V, && [X,[X,[y,V']]\leq V', \\
		\label{recover quadr:decoomp eq 4} &[X,[V,[y,V']]]\leq X, && [X,[V',[y,V]]]\leq X.
	\end{alignat}
\end{corollary}
\begin{proof}
    By \cref{pr:5gr}\cref{5gr:iso} applied on the grading \cref{recover quadr:grading 2}, the map $\ad_d$ is an isomorphism from $L'_{-1}$ to $L'_1$ with inverse $-\ad_c$.
    On the other hand, since $c,d \in L_0$, these maps preserve the grading \cref{recover quadr:grading 1}.
    It follows from \cref{recover quadr:decomp in 3} that $\ad_d$ induces an isomorphism from $V$ to $V'$, with inverse $-\ad_c$. This shows \cref{recover quadr:decoomp eq 1}.

    The relations \cref{recover quadr:decoomp eq 5,recover quadr:decoomp eq 6,recover quadr:decoomp eq 2,recover quadr:decoomp eq 3,recover quadr:decoomp eq 4} all follow immediately from combining the two gradings, as can be seen from \cref{fig:qa}.
%
%
%
\end{proof}

We will now exploit the uniqueness of certain extremal elements to define a quadratic form on the vector space $V$.

\begin{lemma}
\label{recover quadr:unique Q}
	For every $v \in V$, there is a unique $\lambda\in k$ such that
	\begin{align}
	\label{recover quadr:eq unique Q}
		l_v := \lambda x + [v,d] + d
	\end{align}
	is an extremal element.
	Moreover, $l_v=l_{v'}$ if and only if $v=v'$, where $v,v'\in V$.
\end{lemma}
\begin{proof}
    Since $v \in L_{-1}$, \cref{recover:Galois descent} implies that there exists a $v$-exponential automorphism $\alpha \in \Aut(L)$.
    Then $\alpha(d)$ is an extremal element, and we have $\alpha(d) = q_\alpha(d) + [v,d] + d$ with $q_\alpha(d) \in L_{-2} = \langle x \rangle$, showing the existence.
    
	Assume now that that there are two distinct scalars such that \cref{recover quadr:eq unique Q} holds.
	Then the $2$-dimensional subspace $\langle x, [v,d] + d \rangle$ contains three distinct extremal elements, so \cref{pr:2pts} implies that this is a line of the extremal geometry.
	This contradicts \cref{co:E-1}.

	Finally, if $l_v = l_{v'}$ for certain $v,v' \in V$, then $[v,d]=[v',d]$.
	\Cref{pr:5gr}\cref{5gr:iso} applied on the grading \cref{recover quadr:grading 2} then implies that $v = v'$.
\end{proof}

\begin{definition}
\label{def:Q and T}
	Define the map $Q \colon V \to k$ by setting $Q(v)$ to be the unique $\lambda\in k$ from \cref{recover quadr:unique Q}, i.e., such that the element
	\begin{equation}\label{eq:Q-extr}
		l_v = Q(v)x + [v,d] + d
	\end{equation}
	is an extremal element.
	Define the bilinear form $T \colon V\times V \to k$ by
	\[ T(u,v) x = [u,[v,d]] \]
	for all $u,v \in V$.
\end{definition}

\begin{remark}\label{rem:v-exp}
	The uniqueness aspect, although easy to prove, is somewhat subtle, because the $v$-exponential automorphism $\alpha$ that we have used in the proof of \cref{recover quadr:unique Q} is not unique.
	Notice, however, that we will be able to single out a unique $v$-exponential automorphism in \cref{def:v-exp} below.
\end{remark}

\begin{lemma}
\label{recover quadr:Q quadr sp}
	The map $Q$ is a quadratic form on $V$, with corresponding bilinear form $T$.
\end{lemma}
\begin{proof}
	Let $\lambda\in k$ and $u,v \in V$.
	Let $\alpha,\beta$ be a $u$-exponential and a $v$-exponential automorphism, respectively.
	Then $q_{\alpha}(d) + [u,d] + d \in E$ and $q_\alpha(d) = Q(u)x$.
	By \cref{th:alg}\cref{th:alg:scalar}, $\lambda^2 q_{\alpha}(d) + \lambda[u,d] + d \in E$.
	Hence $Q(\lambda u)=\lambda^2 Q(u)$.

    Next, we have
    \begin{align*}
    	\alpha(\beta(d))
    	&= \alpha( Q(v)x + [v,d] + d) \\
    	&= Q(v)x + [v,d] + [u,[v,d]] + Q(u)x + [u,d] + d \\
    	&= \bigl( Q(v) + T(u,v) + Q(u) \bigr) x + [u+v, d] + d \in E,
    \end{align*}
	hence $Q(u+v) = Q(v) + T(u,v) + Q(u)$.
	The map $T$ is bilinear by construction, hence $Q$ is a quadratic form on $V$ with corresponding bilinear form $T$.
\end{proof}

We will see in \cref{le:Q not 0} below that $Q$ is not the zero map.
The following lemma is a first step, useful in its own right.
\begin{lemma}
\label{recover quadr:Q zero}
	Let $v \in V \setminus \{ 0 \}$ be such that $Q(v) = 0$.
	Then $v$, $[v,d]$, $[y,v]$ and $[y,[v,d]]$ are extremal elements.
\end{lemma}
\begin{proof}
	By definition of $Q$ and \cref{recover quadr:Q quadr sp}, we have $\lambda [v,d] + d \in E$ for all $\lambda\in k$.
	Note that $[v,d] \neq 0$ since $\ad_d$ induces an isomorphism from $V$ to $V'$.
	By \cref{pr:2pts}, this implies that $[v,d]\in E$.
	Since $\ad_c$ maps extremal elements in $L'_1$ to extremal elements in $L'_{-1}$ (by \cref{recover:switching grading}),
	it follows that also $v = [c, [v,d]] \in E$.
	Similarly, $\ad_y$ maps extremal elements in $L_{-1}$ to extremal elements in $L_1$, and thus also $[y,v]$ and $[y, [v,d]]$ are extremal elements.
\end{proof}

\begin{lemma}\label{recover quadr:Q with x}
	Let $v \in V$.
	Then
	\begin{alignat}{2}
		&Q(v)x + v + c \in E, \quad && x + v + Q(v)c \in E, \label{eq:Q-xa} \\
		&Q(v)x + [v,d] + d \in E, \quad && x + [v,d] + Q(v)d \in E, \label{eq:Q-xb} \\
		&Q(v)y + [y,v] + c \in E, \quad && y + [y,v] + Q(v)c \in E, \label{eq:Q-ya} \\
		&Q(v)y + [y,[v,d]] + d \in E, \quad && y + [y,[v,d]] + Q(v)d \in E. \label{eq:Q-yb}
	\end{alignat}
	Moreover, in each case, $Q(v)$ is the unique scalar that makes the statement true.
\end{lemma}
\begin{proof}
	Assume first that $Q(v) = 0$.
	Then by \cref{recover quadr:Q zero}, we have $v \in E$.
	By \cref{co:E-1}, this element is collinear with $c$ and with $x$, so in particular, $v + c \in E$ and $x + v \in E$.
	The proof of the other 6 statements is similar.
	
	Assume now that $Q(v) \neq 0$.
	By definition, $Q(v)x + [v,d] + d \in E$.
	By replacing $v$ with $Q(v)^{-1} v$, we also get $Q(v)^{-1}x + Q(v)^{-1} [v,d] + d \in E$, and multiplying by $Q(v)$ now gives
	$x + [v,d] + Q(v) d \in E$. This proves \cref{eq:Q-xb}.
	
	Next, let $\varphi \in \Aut(L)$ be the automorphism obtained from \cref{recover:switching grading} applied to the grading \cref{recover quadr:grading 2}.
	In particular, $\varphi(d) = c$ and $\varphi(x) = x + [c, [d, x]] = x$, so we get
	\[ \varphi \bigl( Q(v)x+[v,d]+d \bigr) = Q(v)x+[c,[v,d]]+c = Q(v)x+v+c \in E . \]
	We can again replace $v$ by $Q(v)^{-1} v$ to get $x + v + Q(v)c \in E$ as well.
	This proves \cref{eq:Q-xa}.
	
	We now apply the automorphism $\psi \in \Aut(L)$ obtained from \cref{recover:switching grading} applied to the grading \cref{recover quadr:grading 1}.
	Then in exactly the same way as in the previous paragraph, applying $\psi$ on \cref{eq:Q-xa} yields \cref{eq:Q-ya} and applying $\psi$ on \cref{eq:Q-xb} yields \cref{eq:Q-yb}.
	
	Finally, the uniqueness statement follows in each case in exactly the same way as in the proof of \cref{recover quadr:unique Q}.
\end{proof}
The previous lemma will, in particular, allow us to single out a specific $v$\dash exponential automorphism, as announced in \cref{rem:v-exp} above.
\begin{definition}\label{def:v-exp}
	Let $v \in V$.
	Then by \cref{recover quadr:Q with x} applied on $-v$, we have $y + [v,y] + Q(v)c \in E$.
	By \cref{recover:sharp trans opp}, there exists a \emph{unique} automorphism $\alpha_v \in E_-(x,y)$ such that
	\begin{equation}\label{eq:alpha-v y}
		\alpha_v(y) = y + [v,y] + Q(v)c .
	\end{equation}
	By \cref{th:alg}\cref{th:alg:E+}, $\alpha_v$ is an $\ell$-exponential automorphism for some $\ell \in L_{-1}$, but then $[\ell, y] = [v,y]$ and therefore, by applying $\ad_x$, we get $\ell = v$, so $\alpha_v$ is a $v$\dash exponential automorphism.
	Notice that by \cref{recover quadr:unique Q,def:Q and T}, we also have
	\begin{equation}\label{eq:alpha-v d}
		\alpha_v(d) = d + [v,d] + Q(v)x .
	\end{equation}
\end{definition}
In a similar way, we can define a unique $[y,v]$-exponential automorphism for each $v \in V$, this time with respect to the $L_i'$-grading \cref{recover quadr:grading 2}. We will need these automorphisms in \cref{le:Theta_av,pr:hardest} below.
\begin{definition}\label{def:yv-exp}
	Let $v \in V$ and consider $[y,v] \in [y,V]$.
	By \cref{recover quadr:Q with x}, we have $d + [y, [v,d]] + Q(v)y \in E$.
	By \cref{recover:sharp trans opp}, there exists a \emph{unique} automorphism $\beta_v \in E_-(c,d)$ such that
	\[ \beta_v(d) = d + [y, [v,d]] + Q(v)y = d + [[y,v], d] + Q(v)y . \]
	In particular, by \cref{th:alg}\cref{th:alg:E+}, $\beta_v$ is an $\ell$-exponential automorphism for some $\ell \in L'_{-1}$, but then $[\ell, d] = [[y,v], d]$, so by applying $\ad_c$, we get $\ell = [y,v]$, so $\beta_v$ is a $[y,v]$-exponential automorphism with respect to the $L_i'$-grading \cref{recover quadr:grading 2}.
	Notice that $\beta_v(x) = x + [[y,v], x] + \lambda c \in E$ for some $\lambda \in k$, so $x + v + \lambda c \in E$, hence by the uniqueness part of \cref{recover quadr:Q with x}, we have
	\[ \beta_v(x) = x + v + Q(v) c . \]
\end{definition}

It is a slightly subtle fact that $\alpha_v$ is not only $v$-exponential with respect to the $L_i$-grading \cref{recover quadr:grading 1}, but also with respect to the $L'_i$-grading \cref{recover quadr:grading 2}, as we now show.
\begin{proposition}\label{pr:v-exp both gradings}
	Let $v \in V = L_{-1} \cap L'_{-1}$.
	Then $\alpha_v$ is also $v$-exponential with respect to the $L'_i$-grading \cref{recover quadr:grading 2}.
\end{proposition}
\begin{proof}
	First apply \cref{ass:quadr}\cref{ass:quadr i} and extend $k$ to a larger field $k'$ such that the extremal geometry $E(L \otimes k')$ contains lines.
	Let $\alpha_v$ be the $v$-exponential automorphism with respect to the $L_i$-grading \cref{recover quadr:grading 1} as in \cref{def:v-exp}.
	On the other hand, let $\alpha'_v$ be the unique $v$-exponential automorphism with respect to the $L'_i$-grading~\cref{recover quadr:grading 2} such that
	\[ \alpha'_v(d) = d + [v,d] + Q(v)x , \]
	similarly to what we have done in \cref{def:yv-exp}, and notice that also
	\[ \alpha'_v(y) = y + [v,y] + Q(v)c = \alpha_v(y) . \]
	We claim that $\alpha_v$ and $\alpha'_v$ also coincide on $L_{-1} = V \oplus X \oplus V'$; it will then follow from \cref{recover: y and I1} that $\alpha_v = \alpha'_v$.
	Our claim is obvious by the grading for elements of $V \oplus X$, so it only remains to show that $\alpha_v([d,w]) = \alpha'_v([d,w])$ for all $w \in V$. This is also clear, however, because $\alpha_v$ and $\alpha'_v$ coincide on both $\langle d \rangle$ and $V$.
	
	We conclude that indeed $\alpha_v = \alpha'_v$, so in particular, $\alpha_v$ is also $v$-exponential with respect to the $L'_i$-grading \cref{recover quadr:grading 2}.
\end{proof}

\begin{lemma}\label{le:T nondeg}
	We have $V \neq 0$ and $V' \neq 0$.
	Moreover, the bilinear form $T \colon V \times V \to k$ is non-degenerate.
\end{lemma}
\begin{proof}
	Notice that by \cref{recover quadr:decoomp eq 1}, we have $V \neq 0$ if and only if $V' \neq 0$.
	Suppose now that $V = V' = 0$; then $L_{-1} = X$ and $L_1 = [y, X]$ (see \cref{fig:qa}), so
	\[ L_{-2} \oplus L_{-1} \oplus L_1 \oplus L_2 \leq L'_0 . \]
	In particular, by \cref{pr:5gr}\cref{5gr:grading der}, $\ad_{[c,d]}$ maps $L_{-2} \oplus L_{-1} \oplus L_1 \oplus L_2$ to $0$.
	Since $L$ is generated by $L_{-2} \oplus L_{-1} \oplus L_1 \oplus L_2$, this implies that the map $\ad_{[c,d]}$ is zero on all of $L$, a contradiction. Hence $V \neq 0$ and $V' \neq 0$.
	
	Now let $v \in V \setminus \{ 0 \}$ be arbitrary, and assume that $T(v,u) = 0$ for all $u \in V$.
	Then by definition of $T$, this implies that $[v, V'] = 0$, but since also $[v, X] = 0$ and $[v, V] = 0$ by the grading, we then have $[v, L_{-1}] = 0$. This contradicts \cref{recover:T nondeg}.
\end{proof}

\begin{lemma}\label{le:Q not 0}
	There exists $v\in V$ such that $Q(v)\neq 0$.
\end{lemma}
\begin{proof}
	Assume that $Q(v) = 0$ for all $v \in V$.
	By \cref{recover quadr:Q zero}, there exist elements in $E \cap L_{-1}$, so by \cref{co:E-1}, the extremal geometry contains lines.
	In fact, each non-zero $v \in V$ is an extremal element collinear with $x$ and, by \cref{recover quadr:Q with x}, also collinear with $c$.
	Moreover, it now follows from the grading and \cref{co:E-1} that $E_{-1}(x) \cap E_{-1}(c) = V \setminus \{ 0 \}$.
	
	By \cref{pr:symplecta}, on the other hand, there is a (unique) symplecton $S$ containing $\langle x \rangle$ and $\langle c \rangle$.
	By \cref{pr:symplecta}\cref{pr:symplecta:a}, $S$ is a non-degenerate polar space of rank $\geq 2$, so we can find two distinct non-collinear points $\langle u \rangle$ and $\langle v \rangle$ both collinear to both $\langle x \rangle$ and $\langle c \rangle$.
	By the previous paragraph, $u$ and $v$ belong to $V \setminus \{ 0 \}$.
	However, this implies that also $u+v \in V \setminus \{ 0 \}$, so by assumption, $Q(u+v) = 0$, so also $u+v$ is an extremal element.
	This contradicts \cref{pr:2pts} because $u$ and $v$ are not collinear.
\end{proof}

In a similar way as in \cref{le:cns:basept}, we can now ensure that $Q$ has a base point, i.e., $Q(v) = 1$ for some $v \in V$.

\begin{lemma}\label{recover quadr:rescaling}
	We can re-choose $x$ in $y$ in \cref{recover quadr:two gradings} in such a way that $Q(e) = 1$ for some $e \in V$.
	We call $e$ a \emph{base point} for $Q$.
\end{lemma}
\begin{proof}
	Consider $e \in V$ such that $Q(e) \neq 0$.
	Set $\tilde x := Q(e)x$ and $\tilde y := Q(e)^{-1}y$.
	Then $\tilde x$, $\tilde y$, $c$ and $d$ still satisfy the conclusions from \cref{recover quadr:existence ext elements}.
	Since we have only replaced these elements by a scalar multiple, this does not affect any of the decompositions.

	Now denote the map obtained as in \cref{def:Q and T} from these new choices by~$\tilde Q$.
	Since $\tilde x + [e,d] + d = Q(e)x + [e,d] + d \in E$, it follows that indeed $\tilde Q(e) = 1$.
\end{proof}

\begin{definition}\label{def:ef}
	Assume from now on that $e \in V$ is a base point for $Q$.
	%
	%
	\begin{enumerate}
		\item 
			We set
			\[ e' := [d,e] \in V', \quad f := [y,e] \in [y,V], \quad f' := [y, e'] \in [y, V'] . \]
			Notice that $f' = [d,f]$ because $[d,y] = 0$.
		\item
			Since $T$ is non-degenerate by \cref{le:T nondeg}, we can fix an element $\delta \in V$ such that $T(e, \delta) = 1$.
			If $\Char(k) \neq 2$, we will assume that $\delta = \tfrac{1}{2} e$.
			
			We may assume, in addition, that $Q(\delta) \neq 0$.
			Indeed, this holds automatically if $\Char(k) \neq 2$, and if $\Char(k) = 2$ and $Q(\delta) = 0$, then we choose some $\lambda \in k$ with $\lambda^2 \neq \lambda$ (which exist because we assume $|k| > 2$, see \cref{ass:quadr}) and we replace $\delta$ by $\lambda e + \delta$.
	\end{enumerate}
\end{definition}
\begin{remark}
	This definition of $\delta$ corresponds to the definition of $\delta$ for quadrangular algebras in \cite[Definition 4.1]{Weiss2006} and \cite[Definition 7.1]{Muehlherr2019}.
	In fact, we will use this $\delta$ in \cref{pr:delta-std} below to ensure that our quadrangular algebra will be \emph{$\delta$-standard}; see \cref{def:theta}\cref{def:pi}.
\end{remark}

\begin{lemma}
	For any $v\in V$, we have
	\begin{align}
		Q(v) \bigl( [c,d] - [x,y] \bigr) &= [[v,d],[v,y]]; \label{eq:vd vy} \\		
		Q(v) \bigl( [c,d] + [x,y] \bigr) &= [v,[y,[d,v]]]. \label{eq:v y dv}
	\end{align}
	In particular,
	\begin{equation}\label{eq:ef}
		[e', f] = [c,d] - [x,y] \quad \text{and} \quad [e, f'] = [c,d] + [x,y] .
	\end{equation}
\end{lemma}
\begin{proof}
	Let $v \in V$ and let $\alpha_v \in \Aut(L)$ be as in \cref{def:v-exp}.
	Since $[d,y] = 0$ and $\alpha_v \in \Aut(L)$, we also get $[\alpha_v(d), \alpha_v(y)] = 0$, so by \cref{eq:alpha-v d,eq:alpha-v y}, we get
	\[ \bigl[ d + [v,d] + Q(v)x, \ y + [v,y] + Q(v)c \bigr] = 0 . \]
	Expanding this ---or in fact, only expanding the $(L_0 \cap L'_0)$-component suffices--- yields \cref{eq:vd vy}.
	
	Next, observe that $[v,[v,d]]=T(v,v)x=2Q(v)x$ by \cref{recover quadr:Q quadr sp}.
	Hence
	\[ [[v,d],[v,y]] = [v,[[v,d],y]] - 2Q(v)[x,y] , \]
	so together with \cref{eq:vd vy}, this yields \cref{eq:v y dv}.
	
	Finally, \cref{eq:ef} now follows by plugging in $v=e$ and using $Q(e)=1$.
\end{proof}

\begin{corollary}\label{co:XX'}
	The restriction of $\ad_f$ to $X$ defines a linear isomorphism from $X$ to $X'$ with inverse $\ad_{e'}$.
\end{corollary}
\begin{proof}
	Let $a \in X$ and $a' \in X'$.
	Then using \cref{eq:ef} and \cref{pr:5gr}\cref{5gr:grading der}, we get
	\begin{align*}
		[e', [f,a]] &= [a, [f, e']] = [[c,d] - [x,y], a] = a, \\
		[f, [e',a']] &= [a', [e', f]] = [-[c,d] + [x,y], a'] = a'. \qedhere
	\end{align*}
\end{proof}

We can now also determine the Lie bracket between $V$ and $[y,V]$ explicitly.
\begin{proposition}\label{pr:yvw}
	Let $v,w \in V$. Then $[[y,v], w] = T(v,w)c$.
\end{proposition}
\begin{proof}
	Since $V = L_{-1} \cap L'_{-1}$, it follows from \cref{pr:5gr}\cref{5gr:grading der} that
	\[ [[c,d] - [x,y], w] = 0 . \]
	By \cref{eq:vd vy}, therefore, we have
	$\bigl[ [[v,d], [v,y]], \, w \bigr] = 0$,
	and hence
	\[ \bigl[ [v,d], \, [[v,y], w] \bigr] = \bigl[ [v,y], \, [[v,d], w] \bigr] . \]
	By the grading, we know that $[[y,v], w] = \mu c$ for some $\mu \in k$. Hence, by \cref{def:Q and T}, the previous equality reduces to
	\[ [[v,d], -\mu c] = [[v,y], -T(v,w)x] . \]
	By \cref{pr:5gr}\cref{5gr:iso}, this can be rewritten as $\mu v = T(v,w)v$, so $\mu = T(v,w)$.
\end{proof}

\begin{definition}
\label{def:h and dot}
	We define maps $h \colon X\times X\to V$ and $\cdot \colon X\times V \to X \colon (a,v) \mapsto a\cdot v$ by setting
	\begin{align}
		h(a, b) &:= [a, [b, f]], \label{eq:def h} \\
		a \cdot v &:= [a, [e', [v,y]]] = [[a, [y, v]], e'] , \label{eq:def .}
	\end{align}
	for all $a,b \in X$ and all $v \in V$.
	Note that the image of these maps is indeed contained in $V$ and $X$, respectively, by \cref{recover quadr:decoomp eq 3,recover quadr:decoomp eq 4} (see again \cref{fig:qa}).
	Also notice that the second equality in \cref{eq:def .} holds because $[V',X] = 0$.
\end{definition}
\begin{lemma}\label{le:.}
	Let $a,b \in X$ and $v \in V$. Then
	\begin{enumerate}
		\item\label{le:.:e} $a \cdot e = a$,
		\item\label{le:.:f} $[a \cdot v, f] = [a, [y,v]] = [[a,y], v]$,
		\item\label{le:.:h} $h(a, b \cdot v) = [a, [b, [y,v]]]$,
		\item\label{le:.:dh} $[d, h(a, b)] = [a, [b, f']]$.
	\end{enumerate}
\end{lemma}
\begin{proof}
	\begin{enumerate}
		\item 
			Using \cref{def:ef,eq:ef}, we have
			\[ a \cdot e = [a, [e', [e, y]]] = - [a, [e', f]] = [[c,d] - [x,y], a] = a \]
			since $a \in X = L_{-1}\cap L'_0$.
		\item
			The first equality follows from \cref{eq:def .,co:XX'} and the second equality follows from the fact that $[X, V] = 0$.
		\item
			This follows from \cref{eq:def h,le:.:f}.
		\item
			Since $[d,X] = 0$, we have
			\[ [d, h(a,b)] = [d, [a, [b, f]]] = [a, [b, [d, f]]] = [a, [b, f']] . \qedhere \]
	\end{enumerate}
\end{proof}

\begin{proposition}\label{pr:delta-std}
	Let $a \in X \leq L_{-1}$.
	Then there exists a unique $a$-exponential automorphism $\alpha \in \Aut(L)$ such that
	\begin{equation}\label{eq:delta-std}
		[q_\alpha(f), [d, \delta]] = 0 .
	\end{equation}
	Moreover, $\alpha$ preserves the $L'_i$-grading \cref{recover quadr:grading 2} of $L$, so in particular, $q_\alpha([y,v]) \in V$ for all $v \in V$.
\end{proposition}
\begin{proof}
	 By \cref{recover:Galois descent}, there exists an $a$\dash exponential automorphism $\alpha \in \Aut(L)$.
	 We first show that $\alpha$ preserves the $L'_i$-grading \cref{recover quadr:grading 2} of $L$.
	 Since $c \in L_0$, we can write $\alpha(c) = c + [a,c] + q_\alpha(c)$.
	 Now $[a,c] = 0$, and if $q_\alpha(c) \neq 0$, then $\alpha(c)$ is an extremal element of the form $c + \lambda x$, which is impossible.
	 Hence $\alpha(c) = c$, and similarly $\alpha(d) = d$.
	 By \cref{recover quadr:prop aut fixing x and y}, this implies that $\alpha$ preserves the $L'_i$ grading \cref{recover quadr:grading 2}.
	 
	 In particular, $\alpha(f) \in L'_{-1}$, so we can write
	 \[ \alpha(f) = f + [a,f] + q_\alpha(f) \]
	 with $q_\alpha(f) \in L_{-1} \cap L'_{-1} = V$.
	 By \cref{th:alg}\cref{th:alg:unique}, any other such $a$-exponential automorphism $\alpha'$ can be written as $\alpha' = \exp(\lambda x) \alpha$ for some $\lambda \in k$.
	 By \cref{rem:beta}\cref{rem:beta:other}, we then have $q_{\alpha'}(f) = q_\alpha(f) + [\lambda x, f] = q_\alpha(f) - \lambda e$.
	 Notice now that $[q_\alpha(f), [d, \delta]] \in \langle x \rangle$.
	 Since $- [e, [d, \delta]] = T(e, \delta)x = x$ by \cref{def:Q and T,def:ef}, there is indeed a unique choice of $\lambda$ such that \cref{eq:delta-std} holds.
\end{proof}

\begin{definition}\label{def:theta}
	\begin{enumerate}
		\item For each $a \in X$, we denote the $a$-exponential automorphism $\alpha \in \Aut(L)$ satisfying \cref{eq:delta-std} by $\Theta_a$. The corresponding (uniquely defined) maps $q_\alpha$, $n_\alpha$ and $v_\alpha$ will be denoted by $q_a$, $n_a$ and $v_a$, respectively.
		\item For each $a \in X$ and each $v \in V$, we define $\theta(a,v) := q_a([y,v]) \in V$. In particular, we have
			\begin{equation}\label{eq:Theta}
				\Theta_a([y,v]) = [y,v] + [a, [y,v]] + \theta(a,v) = [y,v] + [a \cdot v, f] + \theta(a,v) ,
			\end{equation}
			where the second equality holds by \cref{le:.}\cref{le:.:f}.
			Notice that for each $\lambda \in k$, the automorphism $\Theta_{\lambda a}$ coincides with $(\Theta_a)_\lambda$ as defined in \cref{th:alg}\cref{th:alg:scalar}; it follows that $\theta(\lambda a, v) = \lambda^2 \theta(a,v)$.
		\item\label{def:pi} For each $a \in X$, we define $\pi(a) := \theta(a,e) = q_a(f) \in V$. In particular, we have
			\[ \Theta_a(f) = f + [a, f] + \pi(a) . \]
			Notice that \cref{eq:delta-std} tells us that
			\begin{equation}\label{eq:Tdelta}
				T(\pi(a), \delta) = 0 ,
			\end{equation}
			which corresponds precisely to condition (iii) in \cite[Definition 4.1]{Weiss2006} and \cite[Definition 7.1]{Muehlherr2019}.
	\end{enumerate}
\end{definition}

When $\Char(k) \neq 2$, we can describe $\theta$ and $\pi$ in terms of $h$ and $\cdot$; cfr.\@~\cite[Proposition 4.5(i) and Remark 4.8]{Weiss2006}.
\begin{proposition}\label{pr:theta char not 2}
	If $\Char(k) \neq 2$, then $\Theta_a = e_-(a)$ as in \cref{th:alg}\cref{th:alg:char}.
	Moreover, we have $\pi(a) = \tfrac{1}{2}h(a,a)$ and $\theta(a,v) = \tfrac{1}{2}h(a, a \cdot v)$.
\end{proposition}
\begin{proof}
	By definition, we have
	\begin{align*}
		\Theta_a(f) &= f + [a,f] + \pi(a), \\
		e_-(a)(f) &= f + [a,f] + \tfrac{1}{2}[a, [a,f]] = f + [a,f] + \tfrac{1}{2} h(a,a) .
	\end{align*}
	We claim that $T(\tfrac{1}{2}[a, [a,f]], \delta) = 0$. Recall that $\delta = \tfrac{1}{2} e$ since $\Char(k) \neq 2$.
	Using twice the fact that $[a, e'] = 0$ and using \cref{eq:ef}, we get
	\begin{align*}
		T([a, [a,f]], e)
		&= [[a, [a,f]], [e,d]] = [e', [a, [a,f]]] = [a, [e', [a,f]]] \\
		&= [a, [a, [e',f]]] = [a, [a, [c,d] - [x,y]]] = [a, a] = 0 ,
	\end{align*}
	proving our claim. This shows that $e_-(a)$ satisfies the assumption \cref{eq:delta-std}, so it coincides with $\Theta_a$.
	In particular, $\pi(a) = \tfrac{1}{2}h(a,a)$. Moreover,
	\[ \Theta_a([y,v]) = e_-(a)([y,v]) = [y,v] + [a, [y,v]] + \tfrac{1}{2}[a, [a, [y,v]]] , \]
	hence $\theta(a,v) = \tfrac{1}{2} [a, [a, [y,v]]] = \tfrac{1}{2} h(a, a \cdot v)$ by \cref{le:.}\cref{le:.:h}.
\end{proof}

Next, we return to the case where $\Char(k)$ is arbitrary and we recover identity \cref{prelim:def quadr alg ax 7} from \cref{prelim:def quadr alg}.
\begin{proposition}\label{pr:theta a+b}
	There exists a function $\gamma \colon X \times X \to k$ such that
	\[ \theta(a+b, v) = \theta(a,v) + \theta(b,v) + h(a, b \cdot v) - \gamma(a,b) v  \]
	for all $a,b \in X$ and all $v \in V$.
	Moreover, we have $\gamma(a,b) = T(h(a,b), \delta)$ for all $a,b \in X$.
\end{proposition}
\begin{proof}
	Let $a,b \in X$.
	By \cref{th:alg}\cref{th:alg:sum}, the composition $\Theta_a \Theta_b$ is an $(a+b)$-exponential automorphism of $L$, so by \cref{th:alg}\cref{th:alg:unique}, there exists some $\lambda \in k$ such that
	\begin{equation}\label{eq:Theta_a Theta_b}
		\Theta_a \Theta_b = \exp(\lambda x) \Theta_{a+b} .
	\end{equation}
	We write $\gamma(a,b) := - \lambda$.
	
	Observe now that $[b, f] \in [X,f] = X'$. Since $\Theta_a$ preserves the $L'_i$-grading by \cref{pr:delta-std}, we thus get
	\begin{equation}\label{eq:Theta2}
		\Theta_a([b,f]) = [b,f] + [a, [b, f]] = [b,f] + h(a,b) .
	\end{equation}
	Now let $v \in V$; then using \cref{eq:Theta,eq:Theta2}, we get
	\begin{align*}
		\Theta_a \Theta_b([y,v])
		&= \Theta_a\bigl( [y,v] + [b \cdot v, f] + \theta(b,v) \bigr) \\
		&= [y,v] + [a \cdot v, f] + \theta(a,v) + [b \cdot v, f] + h(a, b \cdot v) + \theta(b,v) .
	\end{align*}
	On the other hand,
	\[ \Theta_{a+b}([y,v]) = [y,v] + [(a+b) \cdot v, f] + \theta(a+b, v) , \]
	and since $[x, [y,v]] = -v$, we get
	\[ \exp(\lambda x) \Theta_{a+b}([y,v]) = [y,v] + [(a+b) \cdot v, f] + \theta(a+b, v) - \lambda v . \]
	Since $\lambda = - \gamma(a,b)$, the required expression for $\theta(a+b)$ now follows from \cref{eq:Theta_a Theta_b}.
	
	Next, if we set $v = e$ in this expression, then we get, using \cref{le:.}\cref{le:.:e},
	\begin{equation}\label{eq:pi a+b}
		\pi(a+b) = \pi(a) + \pi(b) + h(a, b) - \gamma(a,b) e .
	\end{equation}
	By \cref{eq:Tdelta}, this implies that $T(h(a,b) - \gamma(a,b)e, \delta) = 0$, and since $T(e,\delta) = 1$, this implies that
	$\gamma(a,b) = T(h(a,b), \delta)$, as claimed.
\end{proof}

Recall from \cref{prelim: def quadr space} that the quadratic form $Q$ with base point $e$ has an associated involution
\[ \sigma \colon V \to V \colon v \mapsto v^\sigma := T(v,e)e - v . \]

\begin{lemma}\label{le:a.vsigma}
	Let $a \in X$ and $v \in V$.
	Then $a \cdot v^\sigma = [a, [v, f']] = [v, [a, f']]$.
\end{lemma}
\begin{proof}
	Notice that
	\begin{align*}
		[v, e'] &= [v, [d, e]] = - T(v,e)x \quad \text{and} \\
		[v, f'] &= [v, [y, e']] = - [[v, e'], y] - [e', [v, y]] = T(v,e)[x,y] - [e', [v,y]] .
	\end{align*}
	Since $X \leq L_{-1}$, it follows that
	\[ [a, [v, f']] = T(v,e)a - [a, [e', [v,y]]] = T(v,e)a - a \cdot v = a \cdot v^\sigma . \]
	The second equality follows from the fact that $[a,v] = 0$.
\end{proof}

\begin{proposition}\label{pr:a v vs}
	Let $a \in X$ and $v \in V$. Then
	\[ (a \cdot v)\cdot v^\sigma = Q(v)a . \]
\end{proposition}
\begin{proof}
	Since $[d, a \cdot v] \in [d, X] = 0$, we have, using \cref{le:.}\cref{le:.:f}, that
	\begin{equation}\label{eq:avf'}
		[a \cdot v, f'] = [a \cdot v, [d, f]] = [d, [a \cdot v, f]] = [d, [a, [y, v]]] .
	\end{equation}
	By \cref{le:a.vsigma}, we then get
	\begin{align*}
		(a \cdot v) \cdot v^\sigma
		&= [a \cdot v, [v, f']] \\
		&= [v, [a \cdot v, f']] && \text{since $[X, V] = 0$} \\
		&= [v, [d, [a, [y, v]]]] && \text{by \cref{eq:avf'}} \\
		&= [v, [a, [d, [y, v]]]] && \text{since $[d, X] = 0$} \\
		&= [a, [v, [d, [y, v]]]] && \text{since $[X, V] = 0$} \\
		&= [[v, [d, [v, y]]], a] .
	\end{align*}
	Now $[[v,d], [v,y]] = [v, [d, [v,y]]] - [d, [v, [v,y]]] = [v, [d, [v,y]]] + \lambda [c,d]$ for some $\lambda \in k$ because $[v, [v, y]] \in \langle c \rangle$ by the grading.
	By \cref{pr:5gr}\cref{5gr:grading der}, however, $[[c,d], a] = 0$, hence
	\[ (a \cdot v) \cdot v^\sigma = \bigl[ [[v,d], [v,y]], a \bigr] . \]
	The result now follows from \cref{eq:vd vy} together with \cref{pr:5gr}\cref{5gr:grading der}.
\end{proof}

\begin{lemma}\label{le:abX}
	Let $a,b \in X$. Then $[a,b] = - T(h(a,b), e)x$.
	In particular, the bilinear map
	\[ X \times X \to k \colon (a,b) \mapsto T(h(a,b), e) \]
	is non-degenerate.
\end{lemma}
\begin{proof}
	Using the definitions of $T$ and $h$, we get
	\[ T(h(a,b), e)x = [h(a,b), [e,d]] = [[a, [b,f]], -e'] = [e', [a, [b, f]]] . \]
	Since $[e', a] = [e', b] = 0$, it follows that
	\[ T(h(a,b), e)x = [a, [b, [e', f]]] . \]
	Once again using \cref{eq:ef} and \cref{pr:5gr}\cref{5gr:grading der}, we get $[b, [e', f]] = -b$ and the required identity follows.
	
	Next, suppose that $a \in X$ is such that $T(h(a,b), e) = 0$ for all $b \in X$. Then $[a,b] = 0$ for all $b \in X$.
	Since also $[a,V] = 0$ and $[a, V'] = 0$ by the grading, this implies that $[a, L_{-1}] = 0$.
	By \cref{recover:T nondeg}, however, this implies $a = 0$, as required.
\end{proof}
\begin{corollary}\label{co:h skew}
	Let $a,b \in X$. Then $h(a,b)^\sigma = -h(b,a)$.
\end{corollary}
\begin{proof}
	Since $[x,f] = -e$, it follows from \cref{le:abX} that
	\[ T(h(a,b), e) e = [[a,b], f] = [a, [b, f]] - [b, [a, f]] = h(a,b) - h(b,a) . \]
	Since $h(a,b)^\sigma = T(h(a,b), e)e - h(a,b)$, the result follows.
\end{proof}

\begin{proposition}\label{pr:axiom 4}
	Let $a,b \in X$ and $v\in V$. Then
	\[ T(h(a, b), v) = T(h(a \cdot v, b), e) = T(h(a, b \cdot v^\sigma), e) . \]
\end{proposition}
\begin{proof}
	Using the definition of $T$, \cref{le:.}\cref{le:.:dh} and the fact that $[V, X] = 0$, we get
	\[ T(h(a, b), v)x = [v, [h(a,b), d]] = - [v, [a, [b, f']]] = - [a, [b, [v, f']]] . \]
	By \cref{le:a.vsigma}, we have $[b, [v, f']] = b \cdot v^\sigma$, so together with \cref{le:abX}, we get
	\[ T(h(a, b), v)x = - [a, b \cdot v^\sigma] = - T(h(a, b \cdot v^\sigma), e) x . \]
	Next, since $\sigma$ is an involutory isometry of $Q$, we also have, using \cref{co:h skew},
	\begin{multline*}
		T(h(a, b), v)x = T(h(a,b)^\sigma, v^\sigma) x = - T(h(b,a), v^\sigma) x = - T(h(b, a \cdot v), e) x \\
		= - T(h(b, a \cdot v)^\sigma, e^\sigma) x = T(h(a \cdot v, b), e) x . \qedhere
	\end{multline*}
\end{proof}

\begin{proposition}\label{pr:axiom 3}
	Let $a,b \in X$ and $v\in V$. Then
	\[ h(a, b\cdot v) = h(b, a \cdot v) + T(h(a, b), e)v. \]
\end{proposition}
\begin{proof}
	By \cref{pr:theta a+b}, we have
	\[ h(a, b \cdot v) - \gamma(a,b) v = h(b, a \cdot v) - \gamma(b,a) v . \]
	Moreover, we have
	\[ \gamma(a,b) - \gamma(b,a) = T(h(a,b) - h(b,a), \delta), \]
	so by \cref{co:h skew}, we get, using $T(e, \delta) = 1$, that
	\[ \gamma(a,b) - \gamma(b,a) = T \bigl( T(h(a,b), e)e, \delta \bigr) = T(h(a,b), e). \]
	The result follows.
\end{proof}

We now introduce some auxiliary notation.
\begin{notation}\label{not:Theta_a(y)}
	Let $a \in X$ and let $\Theta_a$ be as in \cref{def:theta}. Then we can decompose
	\[ \Theta_a(y) = y + [a, y] + \ell_a + \tilde a + \lambda_a x \]
	with $\ell_a \in L_0 \cap L'_0$, $\tilde a \in X$ and $\lambda_a \in k$.
\end{notation}
It turns out that understanding the adjoint action of $\ell_a$ explicitly on the various components in the $5 \times 5$-grading gives a lot of insight, so this is what we will examine in the next few results.

\begin{lemma}\label{le:lc=0}
	Let $a \in X$. Then $[\ell_a, c] = 0$ and $[\ell_a, d] = 0$.
\end{lemma}
\begin{proof}
	Since $\Theta_a(c) = c$ and $[y,c] = 0$, we have $[\Theta_a(y), c] = 0$, so in particular, $[\ell_a, c] = 0$. The proof of $[\ell_a, d] = 0$ is similar.
\end{proof}

\begin{lemma}\label{le:lv}
	Let $a \in X$ and $v \in V$. Then
	\begin{align*}
		[\ell_a, v] &= \theta(a,v) \quad \text{and} \\
		[\ell_a, [d,v]] &= [d, \theta(a,v)] .
	\end{align*}
\end{lemma}
\begin{proof}
	Since $\Theta_a(v) = v$, we have
	\begin{align*}
		\Theta_a([y,v]) = [\Theta_a(y), \Theta_a(v)]
		&= \bigl[ y + [a, y] + \ell_a + \tilde a + \lambda_a x, \ v \bigr] \\
		&= [y, v] + [[a,y], v] + [\ell_a, v] .
	\end{align*}
	On the other hand, recall from \cref{eq:Theta} that
	\[ \Theta_a([y,v]) = [y,v] + [a \cdot v, f] + \theta(a,v) . \]
	Hence $[\ell_a, v] = \theta(a,v)$.
	Next, consider the automorphism $\varphi$ from \cref{recover:switching grading} with respect to the $L'_i$-grading \cref{recover quadr:grading 2}.
	Then by \cref{recover:switching grading,le:lc=0},
	\begin{align*}
		\varphi(\ell_a) &= \ell_a + [c, [d, \ell_a]] = \ell_a, \\
		\varphi(v) &= [d, v], \\
		\varphi(\theta(a,v)) &= [d, \theta(a,v)],
	\end{align*}
	so the second identity follows from the first by applying $\varphi$.
\end{proof}
\begin{proposition}\label{pr:lx}
	Let $a \in X$ and $v \in V$. Then
	\begin{align}
		[\ell_a, x] &= T(\pi(a), e) x, \notag \\
		[\ell_a, y] &= - T(\pi(a), e) y \notag
	\intertext{and}
		T(\theta(a,v), v) &= Q(v) T(\pi(a), e) . \label{eq:Ttheta}
	\end{align}
\end{proposition}
\begin{proof}
	Since $\Theta_a(x) = x$, we have
	\begin{equation}\label{eq:Theta xy}
		\Theta_a([x,y]) = \bigl[ x, \ y + [a, y] + \ell_a + \tilde a + \lambda_a x \bigr]
		= [x,y] + a - [\ell_a, x] .
	\end{equation}
	We also have $\Theta_a([c,d]) = [\Theta_a(c), \Theta_a(d)] = [c,d]$ and $\Theta_a([d,v]) = [d,v]$.
	We now apply $\Theta_a$ to \cref{eq:vd vy}, so we get, using \cref{eq:Theta,eq:Theta xy},
	\[ Q(v) \bigl( [c,d] - [x,y] - a + [\ell_a, x] \bigr) = \bigl[ [d,v], [y,v] + [a \cdot v, f] + \theta(a,v) \bigr] . \]
	In particular, the $L_{-2}$-components of both sides are equal, so we get
	\[ Q(v) [\ell_a, x] = \bigl[ [d,v], \theta(a,v) \bigr] = T(\theta(a,v), v) . \]
	If we set $v = e$, then we get $[\ell_a, x] = T(\pi(a), e)$. Substituting this again in the identity for general $v$ then yields the third identity.
	Finally, to get the second identity, write $[\ell_a, y] = \mu y$. Since $[\ell_a, [x,y]] = 0$ by \cref{pr:5gr}\cref{5gr:grading der}, we have $[[\ell_a, x], y] = [[\ell_a, y], x]$ and hence $\mu = - T(\pi(a), e)$.
\end{proof}
\begin{corollary}\label{co:Ttheta}
	Let $a \in X$ and $v,w \in V$. Then
	\[  T(\theta(a,v), w) + T(\theta(a,w), v) = T(v,w) T(\pi(a), e) \]
	and
	\[ \theta(a, v^\sigma)^\sigma = \theta(a,v) - T(e,v) \pi(a) + T(\pi(a), v)e . \]
\end{corollary}
\begin{proof}
	The first identity follows by linearizing \cref{eq:Ttheta}.
	If we choose $w = e$, then we get
	\[  T(\theta(a,v), e) + T(\pi(a), v) = T(v,e) T(\pi(a), e) . \]
	The second identity then follows from the definition of $\sigma$ together with this identity.
\end{proof}

\begin{lemma}
	Let $a,b \in X$. Then
	\begin{equation}\label{eq:lbf}
		[\ell_a, [b, f]] = - [ a \cdot h(a,b), f ]
	\end{equation}
	and
	\begin{equation}\label{eq:h tilde}
		\theta(a, h(a,b)) + h(\tilde a, b) = 0 .
	\end{equation}
\end{lemma}
\begin{proof}
	We have $[y, [b,f]] = 0$ by the grading, hence $[\Theta_a(y), \Theta_a([b,f])] = 0$.
	Recall from \cref{eq:Theta2} that
	\[ \Theta_a([b,f]) = [b,f] + h(a,b) \in X' \oplus V . \]
	Combining this with \cref{not:Theta_a(y)}, we get
	\[ \bigl[ y + [a, y] + \ell_a + \tilde a + \lambda_a x, \ [b,f] + h(a,b) \bigr] = 0 . \]
	In particular, the $L_0$-component and $L_{-1}$-component of the left hand side are $0$, so
	\begin{align*}
		&\bigl[ [a,y], h(a,b) \bigr] + \bigl[ \ell_a, [b,f] \bigr] = 0, \\
		&\bigl[ \ell_a, h(a,b) \bigr] + \bigl[ \tilde a, [b,f] \bigr] = 0 .
	\end{align*}
	Applying \cref{le:.}\cref{le:.:f} on the first identity now yields \eqref{eq:lbf},
	and applying \cref{le:lv,eq:def h} on the second identity yields \eqref{eq:h tilde}.
\end{proof}

\begin{lemma}\label{le:pi sigma}
	Let $a \in X$ and $v \in V$. Then
	\begin{equation}\label{eq:haav}
		h(a, a \cdot v) + T(\pi(a), e)v = 2 \theta(a,v) .
	\end{equation}
	In particular, $\pi(a)^\sigma = \pi(a) - h(a,a)$.
	Moreover, we have $T(\pi(a), e) = \gamma(a,a)$ and $2 \gamma(a,a) = 0$, i.e., $\gamma(a,a) = 0$ when $\Char(k) \neq 2$.
\end{lemma}
\begin{proof}
	By \cref{pr:5gr}\cref{5gr:grading der}, we have $[[x,y], y] = 2y$.
	Applying $\Theta_a$ and using \cref{not:Theta_a(y),pr:lx,eq:Theta xy} then gives
	\[ \bigl[ [x,y] + a - T(\pi(a), e)x, \ y + [a, y] + \ell_a + \tilde a + \lambda_a x \bigr] = 2 \bigl( y + [a, y] + \ell_a + \tilde a + \lambda_a x \bigr) . \]
	In particular, the $L_0$-components of both sides are equal, so using $[[x,y], \ell_a] = 0$ (by \cref{pr:5gr}\cref{5gr:grading der} again), we get
	\[  [a, [a, y]] - T(\pi(a), e) [x,y] = 2 \ell_a . \]
	Now notice that by \cref{le:.}\cref{le:.:h} and using $[a,v] = 0$, we have $\bigl[ [a, [a,y]], v \bigr] = h(a, a \cdot v)$. Moreover, $[[x,y], v] = -v$ by \cref{pr:5gr}\cref{5gr:grading der} again, and $[\ell_a, v] = \theta(a,v)$ by \cref{le:lv}, so we get \cref{eq:haav}.
	In particular, we have $h(a,a) + T(\pi(a), e)e = 2 \pi(a)$, so
	\[ \pi(a)^\sigma = T(\pi(a), e)e - \pi(a) = \pi(a) - h(a,a) . \]
	Moreover, by \cref{pr:theta a+b} with $a=b$ and $v=e$, we have
	\[ 2 \pi(a) = h(a,a) - \gamma(a,a)e , \]
	so by comparing with the previous identity, we get $T(\pi(a), e) = - \gamma(a,a)$.
	
	Finally, observe that when $\Char(k) \neq 2$, then $\delta = \tfrac{1}{2}e$, so by \cref{pr:theta a+b}, $\gamma(a,a) = T(h(a,a), \tfrac{1}{2} e)$, which is $0$ by \cref{le:abX}.
\end{proof}

\begin{proposition}\label{pr:tilde a}
	Let $a \in X$ and $v \in V$.
	Then $\tilde a = - a \cdot \pi(a)^\sigma$ and
	\begin{equation}\label{eq:a theta av}
		a \cdot \theta(a,v) = a \cdot \pi(a) \cdot v .
	\end{equation}
	Moreover, we have
	\begin{equation}\label{eq:lambda_a v}
		\lambda_a v = \theta(a, \theta(a,v)) - h(a \cdot \pi(a)^\sigma, a \cdot v) .
	\end{equation}
\end{proposition}
\begin{proof}
	We have $[y, [y,v]] = 0$ by the grading, hence $[\Theta_a(y), \Theta_a([y,v])] = 0$.
	By \cref{eq:Theta,not:Theta_a(y)}, we get
	\begin{equation}\label{eq:theta-proof}
		\bigl[ y + [a, y] + \ell_a + \tilde a + \lambda_a x, \ [y,v] + [a \cdot v, f] + \theta(a,v) \bigr] = 0 .
	\end{equation}
	In particular, the $L_0$-component of the left hand side is $0$, so
	\[ \bigl[ [a,y], \theta(a,v) \bigr] + \bigl[ \ell_a, [a \cdot v, f] \bigr] + \bigl[ \tilde a, [y, v] \bigr] = 0 . \]
	Now using \cref{le:.}\cref{le:.:f} for the first and third term and using \cref{eq:lbf} for the second term, we get
	\[ \bigl[ a \cdot \theta(a,v) - a \cdot h(a, a \cdot v) + \tilde a \cdot v, \ f \bigr] = 0 \]
	and hence, by \cref{co:XX'},
	\[ a \cdot \theta(a,v) - a \cdot h(a, a \cdot v) + \tilde a \cdot v = 0 . \]
	By \cref{eq:haav}, this can be rewritten as
	\[ \tilde a \cdot v = a \cdot \theta(a,v) - T(\pi(a), e) a \cdot v . \]
	If we set $v = e$, then we get $\tilde a = - a \cdot \pi(a)^\sigma$. Substituting this again in the previous identity, we get, using $\pi(a)^\sigma = T(\pi(a), e)e - \pi(a)$, that
	\[ a \cdot \theta(a,v) = - a \cdot \pi(a)^\sigma \cdot v + T(\pi(a), e) a \cdot v = a \cdot \pi(a) \cdot v . \]
	To prove the second identity, we again start from \cref{eq:theta-proof}, but this time we compute the $L_{-1}$-component, so we get
	\[ \bigl[ \ell_a, \theta(a,v) \bigr] + \bigl[ \tilde a, [a \cdot v, f] \bigr] + \bigl[ \lambda_a x, [y,v] \bigr] = 0. \]
	Using $\tilde a = - a \cdot \pi(a)^\sigma$, \cref{eq:def h} and \cref{le:lv}, we get \cref{eq:lambda_a v}.
\end{proof}

\begin{lemma}\label{le:lyv}
	Let $a \in X$ and $v \in V$. Then
	\[ [\ell_a, [y, v]] = [y, \theta(a,v) - T(\pi(a), e)v] . \]
	In particular, $[\ell_a, f] = - [y, \pi(a)^\sigma]$.
\end{lemma}
\begin{proof}
	Using \cref{le:lv,pr:lx}, we get
	\[ [\ell_a, [y, v]] = [[\ell_a, y], v] + [y, [\ell_a, v]] = - T(\pi(a), e)[y,v] + [y, \theta(a,v)] . \]
	The second identity follows by setting $v = e$ since $f = [y, e]$.
\end{proof}

\begin{lemma}\label{le:lb}
	Let $a,b \in X$. Then
	\[ [\ell_a, b] = b \cdot \pi(a)^\sigma - a \cdot h(a,b) . \]
\end{lemma}
\begin{proof}
	Using \cref{eq:lbf}, \cref{le:lyv} and \cref{le:.}\cref{le:.:h}, we get
	\begin{align*}
		[[\ell_a, b], f] = [\ell_a, [b, f]] - [b, [\ell_a, f]]
			&= - [a \cdot h(a,b), f] + [b, [y, \pi(a)^\sigma]] \\
			&= [- a \cdot h(a,b) + b \cdot \pi(a)^\sigma, f] .
	\end{align*}
	The result now follows from \cref{co:XX'}.
\end{proof}

We can now explicitly determine the value of $\lambda_a$ introduced in \cref{not:Theta_a(y)}.
\begin{proposition}\label{pr:lambda_a}
	Let $a \in X$.
	Then $\lambda_a = Q(\pi(a))$.
\end{proposition}
\begin{proof}
	Since $\tilde a = - a \cdot \pi(a)^\sigma$ by \cref{pr:tilde a}, it follows from \eqref{eq:h tilde} with $b=a$ that
	$h(a \cdot \pi(a)^\sigma, a) = \theta(a, h(a,a))$,
	so by \cref{eq:a theta av}, we get
	\begin{equation}\label{eq:a pi h}
		a \cdot h(a \cdot \pi(a)^\sigma, a) = a \cdot \pi(a) \cdot h(a,a) .
	\end{equation}
	By \eqref{eq:lambda_a v} with $v = e$, we have
	$\lambda_a e = \theta(a, \pi(a)) - h(a \cdot \pi(a)^\sigma, a)$,
	hence
	\[ \lambda_a a = a \cdot \theta(a, \pi(a)) - a \cdot h(a \cdot \pi(a)^\sigma, a) . \]
	By \cref{eq:a theta av} again and by \cref{eq:a pi h}, we can rewrite this as
	\[ \lambda_a a = a \cdot \pi(a) \cdot \pi(a) - a \cdot \pi(a) \cdot h(a,a) . \]
	Since $\pi(a) - h(a,a) = \pi(a)^\sigma$ by \cref{le:pi sigma}, we get, using \cref{pr:a v vs}, that
	\[ \lambda_a a = a \cdot \pi(a) \cdot \pi(a)^\sigma = Q(\pi(a)) a . \qedhere \]
\end{proof}
We can now summarize our information about $\Theta_a(y)$.
\begin{corollary}\label{co:Theta_a y}
	Let $a \in X$. Then
	\[ \Theta_a(y) = y + [a, y] + \ell_a - a \cdot \pi(a)^\sigma + Q(\pi(a)) x . \]
\end{corollary}
\begin{proof}
	This follows from \cref{not:Theta_a(y),pr:tilde a,pr:lambda_a}.
\end{proof}

We deduce another interesting identity, which is not required for the proof of our main result (\cref{th:quadr main}), but which is interesting in its own right, and relevant also in the reconstruction process (see \cref{rem:qa reconstruct} below).
\begin{proposition}\label{pr:Vabc}
	Let $a,b,c \in X$.
	Write $V_{a,b} := [a, [y,b]] \in L_0 \cap L'_0$.
	Then
	\[ [V_{a,b}, c] = a \cdot h(b,c) + b \cdot h(a,c) + c \cdot h(b,a) . \]
\end{proposition}
\begin{proof}
	We apply \cref{eq:Theta_a Theta_b} on $y$ and invoke \cref{co:Theta_a y} to get
	\begin{multline*}
		\Theta_a \bigl( y + [b,y] + \ell_b - b \cdot \pi(b)^\sigma + Q(\pi(b)) x \bigr) \\
			= \exp( -\gamma(a,b)x ) \bigl( y + [a+b,y] + \ell_{a+b} - (a+b) \cdot \pi(a+b)^\sigma + Q(\pi(a+b)) x \bigr) .
	\end{multline*}
	Comparing the $L_0$-component of both sides yields
	\[ \ell_a + [a, [b,y]] + \ell_b = \ell_{a+b} - \gamma(a,b) [x,y] . \]
	Taking the Lie bracket with $c \in X \leq L_{-1}$ now gives
	\[ [\ell_a, c] - [V_{a,b}, c] + [\ell_b, c] = [\ell_{a+b}, c] + \gamma(a,b) c . \]
	We now apply \cref{le:lb} and we get
	\begin{multline}\label{eq:Vabc}
		[V_{a,b}, c] = \bigl( c \cdot \pi(a)^\sigma - a \cdot h(a,c) \bigr) + \bigl( c \cdot \pi(b)^\sigma - b \cdot h(b,c) \bigr) \\
		- \bigl( c \cdot \pi(a+b)^\sigma - (a+b) \cdot h(a+b,c) \bigr) - \gamma(a,b) c.
	\end{multline}
	By \cref{eq:pi a+b,co:h skew}, we have
	\[ c \cdot \pi(a+b)^\sigma = c \cdot \pi(a)^\sigma + c \cdot \pi(b)^\sigma - c \cdot h(b,a) - \gamma(a,b) c . \]
	Substituting this expression in \cref{eq:Vabc}, the result follows.
\end{proof}
\begin{remark}
	The notation $V_{a,b}$ that we use is not a coincidence: the expression for $[V_{a,b}, c]$ is exactly the expression for the \emph{$V$-operators} in the theory of structurable algebras (in characteristic $\neq 2,3$) related to quadrangular algebras.
	See \cite[Theorem 5.4]{Boelaert2013}.
\end{remark}

We have now arrived at the most difficult identity, which we will prove in \cref{pr:hardest} below.
The crucial ingredient will be \cref{le:Theta_av}, proving the equality of certain products of automorphisms, which is, in fact, a typical ``long commutator relation'' in disguise.

Similarly to \cref{pr:delta-std,def:theta}, we will need exponential maps for elements of $X'$, with respect to the $L'_i$-grading \cref{recover quadr:grading 2}.	
\begin{definition}
	\begin{enumerate}
		\item 
			For each $[a,f] \in X' \leq L'_{-1}$, we write $\hat\Theta_a$ for the unique $[a,f]$-exponential automorphism $\alpha \in \Aut(L)$ such that $[q_\alpha(e'), [d, \delta]] = 0$. The corresponding (uniquely defined) maps $q_\alpha$, $n_\alpha$ and $v_\alpha$ will be denoted by $\hat q_a$, $\hat n_a$ and $\hat v_a$, respectively.
		\item
			For each $a \in X$ and each $v \in V$, we define $\hat\theta(a,v) := \hat q_a([d,v]) \in V$. In particular, we have
			\begin{equation}\label{eq:Theta hat}
				\hat\Theta_a([d,v]) = [d,v] + [[a,f], [d,v]] + \hat\theta(a,v) = [d,v] + a \cdot v^\sigma + \hat\theta(a,v) ,
			\end{equation}
			where the second equality holds by \cref{le:a.vsigma} since $f' = [d,f]$.
	\end{enumerate}
\end{definition}

For the next lemma, recall that $\alpha_{\theta(a,v)}$ and $\beta_v$ have been introduced in \cref{def:v-exp,def:yv-exp}, respectively.
\begin{lemma}\label{le:Theta_av}
	Let $a \in X$ and $v \in V$.
	Then there is a (unique) $\mu =: \mu(a,v) \in k$ such that
	\[ \exp(- \mu c) \, \Theta_a \, \beta_v = \beta_v \, \hat\Theta_{a \cdot v} \, \alpha_{\theta(a,v)} \, \Theta_a . \]
\end{lemma}
\begin{proof}
	Consider the automorphism $\beta := \Theta_a \, \beta_v \, \Theta_a^{-1}$ and recall that $\beta_v$ is a $[y,v]$\dash exponential automorphism with respect to the $L'_i$-grading \cref{recover quadr:grading 2}. Write
	\[ \beta_v = \id + \ad_{[y,v]} + q_{[y,v]} + n_{[y,v]} + v_{[y,v]} . \]
	Now let $l_i \in L'_i$.
	Since $\Theta_a$ preserves the $L'_i$-grading \cref{recover quadr:grading 2}, also $\Theta_a^{-1}(l_i) \in L'_i$, so we can write
	\[ \beta_v \, \Theta_a^{-1}(l_i)
		= \underbrace{\Theta_a^{-1}(l_i)}_{\in L'_i}
		+ \underbrace{[[y,v], \Theta_a^{-1}(l_i)]}_{\in L'_{i-1}}
		+ \underbrace{q_{[y,v]}(\Theta_a^{-1}(l_i))}_{\in L'_{i-2}} . \]
	Hence
	\[ \beta(l_i)
		= \underbrace{l_i}_{\in L'_i}
		+ \underbrace{\bigl[ \Theta_a([y,v]), l_i \bigr]}_{\in L'_{i-1}}
		+ \underbrace{\Theta_a(q_{[y,v]}(\Theta_a^{-1}(l_i)))}_{\in L'_{i-2}} . \]
	This shows that $\beta$ is an $l$-exponential automorphism for $l = \Theta_a([y,v])$, with respect to the $L'_i$-grading~\cref{recover quadr:grading 2}. Now
	\[ l = [y,v] + [a \cdot v, f] + \theta(a,v) , \]
	so by \cref{th:alg}\cref{th:alg:sum}, also $\beta_v \, \hat\Theta_{a \cdot v} \, \alpha_{\theta(a,v)}$ is an $l$-exponential automorphism.
	(Notice that we have implicitly used \cref{pr:v-exp both gradings} for $\alpha_{\theta(a,v)}$ here.)
	The result now follows from \cref{th:alg}\cref{th:alg:unique}.
\end{proof}

\begin{definition}
	Let $a \in X$ and $v \in V$. Then we set
	\[ \phi(a,v) := \mu(a,v) - \mu(a \cdot v, e) . \]
\end{definition}
\begin{proposition}\label{pr:hardest}
	Let $a \in X$ and $v,w \in V$. Then
	\[ \hat\theta(a,v) = \theta(a, v^\sigma)^\sigma + \mu(a,e) v \]
	and
	\[ \theta(a \cdot v, w^\sigma)^\sigma = Q(v) \theta(a,w) - T(v,w) \theta(a,v) + T(\theta(a,v), w)v + \phi(a,v)w . \]
\end{proposition}
\begin{proof}
	We apply both sides of \cref{le:Theta_av} to $[d,w] \in V'$ and we extract the $V$\dash component of the resulting equality (so there are many terms that we do not need to compute explicitly).
	We first compute
	\begin{align*}
		\beta_v([d,w])
		&= [\beta_v(d), \beta_v(w)] \\
		&= \bigl[ d + [[y,v], d] + Q(v)y , \ w + [[y,v], w] \bigr] \\
		&= \bigl[ d + [[y,v], d] + Q(v)y , \ w + T(v,w)c \bigr] \quad \text{(by \cref{pr:yvw})} \\
		&= [d,w] + \bigl[ [y,v], [d,w] \bigr] + \bigl[ [[y,v], d], \ T(v,w)c \bigr] + [Q(v)y, w] \\
		&= [d,w] + \bigl[ [y,v], [d,w] \bigr] + \bigl[ y, Q(v)w - T(v,w)v \bigr] .
	\end{align*}
	After applying $\exp(-\mu(a,v) c) \Theta_a$ on $\beta_v([d,w])$, we see from the grading in \cref{fig:qa} that only two terms contribute to the $V$\dash component, namely
	\[ \exp(-\mu(a,v) c)([d,w]) \quad \text{and} \quad \Theta_a \bigl( \bigl[ y, Q(v)w - T(v,w)v \bigr] \bigr) , \]
	and the resulting $V$-component is equal to
	\begin{equation}\label{eq:hardest1}
		\mu(a,v) w + Q(v) \theta(a,w) - T(v,w) \theta(a,v) .
	\end{equation}
	On the other hand, we compute the $V$-component of
	\begin{align*}
		\beta_v \, \hat\Theta_{a \cdot v} \, \alpha_{\theta(a,v)} \, \Theta_a([d,w])
		&= \beta_v \, \hat\Theta_{a \cdot v} \bigl( [d,w] + [\theta(a,v), [d,w]] \bigr) \\
		&= \beta_v \, \hat\Theta_{a \cdot v} \bigl( [d,w] - T(\theta(a,v), w)x \bigr) .
	\end{align*}
	Again, only two terms contribute, namely
	\[ \hat\Theta_{a \cdot v}([d,w]) \quad \text{and} \quad \beta_v\bigl( - T(\theta(a,v), w)x \bigr) , \]
	and the resulting $V$-component is equal to
	\begin{equation}\label{eq:hardest2}
		\hat\theta(a \cdot v, w) - T(\theta(a,v), w) v .
	\end{equation}
	Since the expressions in \cref{eq:hardest1,eq:hardest2} coincide, we get
	\begin{equation}\label{eq:hardest3}
		\hat\theta(a \cdot v, w) = Q(v) \theta(a,w) - T(v,w) \theta(a,v) + T(\theta(a,v), w) v + \mu(a,v) w .
	\end{equation}
	Setting $v = e$ and invoking \cref{co:Ttheta} then gives
	\begin{align*}
		\hat\theta(a, w) &= \theta(a,w) - T(e,w) \pi(a) + T(\pi(a), w) e + \mu(a,e) w \\
		&= \theta(a, w^\sigma)^\sigma + \mu(a,e) w .
	\end{align*}
	Finally, substituting $a \cdot v$ for $a$ in this identity gives
	\[ \hat\theta(a \cdot v, w) = \theta(a \cdot v, w^\sigma)^\sigma + \mu(a \cdot v, e) w . \]
	Comparing this with \cref{eq:hardest3} gives the required identity.
\end{proof}

We have now assembled all that is required to prove our main result of this section.

\begin{theorem}\label{th:quadr main}
	Let $L$ be as in \cref{ass:quadr}, let $V$ and $X$ be as in \cref{recover quadr:two gradings}, let $Q$ and $T$ be as in \cref{def:Q and T}, let $e$ and $\delta$ be as in \cref{def:ef}, let $h$ and $\cdot$ be as in \cref{def:h and dot} and let $\theta$ be as in \cref{def:theta}.
	
	Then the system $(k,V,Q,T,e,X,\cdot,h,\theta)$ is a quadrangular algebra, which is $\delta$\dash standard and for which both $T$ and the bilinear map $X \times X \to k \colon (a,b) \mapsto T(h(a,b), e)$ are non-degenerate (so in particular, $h$ is non-degenerate).
\end{theorem}
\begin{proof}
	We first observe that $Q \colon V \to k$ is a regular quadratic form with associated bilinear form $T$, by \cref{def:Q and T} and \cref{le:T nondeg}, so in fact, $T$ is non-degenerate. By \cref{def:ef}, $e \in V$ is a base point for $Q$.
	By definition, the maps $\cdot \colon X \times V \to X$ and $h \colon X \times X \to V$ defined in \cref{def:h and dot} are bilinear maps.
	Finally, the map $\theta \colon X \times V \to V$ has been introduced in \cref{def:theta}.
	
	We now verify each of the defining axioms \cref{prelim:def quadr alg ax 1}--\cref{prelim:def quadr alg ax 9} from \cref{prelim:def quadr alg}.
	\begin{enumerate}
		\item This is \cref{le:.}\cref{le:.:e}.
		\item This is \cref{pr:a v vs}.
		\item This is \cref{pr:axiom 3}.
		\item This is \cref{pr:axiom 4}.
		\item This follows from the definition of $\theta$ in \cref{def:theta} because for each $a \in X$, $\Theta_a$ is a linear map.
		\item This was observed already in \cref{def:theta}.
		\item This is \cref{pr:theta a+b}.
		\item This is \cref{pr:hardest}.
		\item This is \cref{pr:tilde a}.
	\end{enumerate}
	This shows that $(k,V,Q,T,e,X,\cdot,h,\theta)$ is a quadrangular algebra.
	By \cref{def:ef,eq:Tdelta}, it is $\delta$-standard.
	Finally, the bilinear map $X \times X \to k \colon (a,b) \mapsto T(h(a,b), e)$ is non-degenerate by \cref{le:abX}.
\end{proof}

\begin{remark}\label{rem:QA conditions}
	The conditions on the quadrangular algebra in the statement of \cref{th:quadr main} imply that if $X \neq 0$, then either the quadrangular algebra is ``generic'' (as in \cite[\S 8]{Muehlherr2019}) or it is ``of split type $F_4$'' (as in \cite[\S 10]{Muehlherr2019}). (We ignore the detail that $|K|>4$ is required in the generic case in \cite[Hypothesis 8.1]{Muehlherr2019} and $|K|>3$ is required in the split $F_4$-case in \cite[Hypothesis 10.1]{Muehlherr2019}.)
	Notice, however, that we allow the situation where $X = 0$, in which case the quadrangular algebra is nothing more than a non-degenerate quadratic space with base point.
\end{remark}

\begin{remark}\label{rem:MH}
	When the quadrangular algebra is anisotropic, it ought to be possible to show that the resulting \emph{inner ideal geometry} is a Moufang generalized quadrangle and that it is isomorphic to the Moufang quadrangle obtained from the quadrangular algebra as in \cite{Weiss2006}. (Notice that in this case, the \emph{extremal geometry} has no lines. The inner ideal geometry is the geometry with point set $\E(L)$ and as line set the collection of inner ideals containing at least two extremal points and minimal with respect to this property.)
	
	The explicit computations might be lengthy in general, but in the case that $\Char(k) \neq 2$, this has been worked out in \cite[\S 4.5.2]{MeulewaeterPhD}.
\end{remark}

\begin{remark}\label{rem:qa reconstruct}
	We are confident that, similarly to \cref{co:cns reconstruct}, the Lie algebra can be completely reconstructed from the quadrangular algebra, but we have not worked out the details.
	This time, we let
	\[ K_0 := \langle [x,y], [c,d] \rangle \oplus [X, [y,X]] \oplus [V, [y, V']] , \]
	and we follow exactly the same procedure as in the proof of \cref{co:cns reconstruct}.
	However, working out all the different cases of Lie brackets between the $13$ different pieces in the grading requires substantially more computational effort than in the case of \cref{se:lines}, and one crucial ingredient seems to be \cref{pr:Vabc}.
	We leave the details to the courageous reader.
\end{remark}

\begin{remark}\label{rem:qa 3rd grading}
	As we can see, there is a third interesting grading on the Lie algebra $L$ that we obtain from the diagonals in \cref{fig:qa}.
	(There is, of course, also a fourth grading from the other diagonals.)
	Notice that this time, the ends of the grading are not one-dimensional! More precisely, we have
	\[ L = L''_{-2} \oplus L''_{-1} \oplus L''_0 \oplus L''_1 \oplus L''_2 \]
	with
	\begin{align*}
		L''_{-2} &= \langle x \rangle \oplus V \oplus \langle c \rangle \\
		L''_{-1} &= X \oplus X' \\
		L''_{0} &= V' \oplus (L_0 \cap L'_0) \oplus [y,V] \\
		L''_{1} &= [d, X'] \oplus [y,X] \\
		L''_{2} &= \langle d \rangle \oplus [y,V'] \oplus \langle y \rangle .
	\end{align*}
	This grading has been used implicitly in \cite{Boelaert2015} in the context of $J$-ternary algebras. See in particular its Theorem~3.5, where the Jordan algebra $J$ is of the form $k \oplus V \oplus k$, and where the space called $X$ plays the role of $L''_{\pm 1}$ and admits a Peirce decomposition into two isomorphic parts $X = X_0 \oplus X_1$ (introduced in its Lemma~3.4).
\end{remark}

\appendix

\section{Tables}

In this appendix, we present an overview of the possible Lie algebras \emph{of exceptional type} that arise in this context.
For each of the two situations (cubic norm structures vs.\@~quadrangular algebras), we present two tables:
\begin{itemize}
	\item We first list the dimensions of the different pieces in the decomposition of the Lie algebra, depending only on the absolute type of the Lie algebra.
	\item We then present a more detailed table with the possible Tits indices along with the corresponding precise form of the algebraic structure.
\end{itemize}
We do not provide proofs and we rely instead on the corresponding information for Tits hexagons (taken from \cite{Muehlherr2022}) and Tits quadrangles (taken from \cite{Muehlherr2019}), respectively.
Providing a direct connection between Tits hexagons and Tits quadrangles of index type and our Lie algebras is an interesting project in its own right. (See also \cref{rem:MQ,rem:MH}.)

\subsection{Cubic norm structures}

In the case of cubic norm structures, corresponding to the $G_2$-grading as in \cref{fig:cns} on p.\@~\pageref{fig:cns}, we see that
\[ \dim L = 6 + 6 \dim J + \dim (L_0 \cap L'_0) . \]
We have listed the six possibilities in \cref{table:dimG2}.
\begin{table}[p]
\[
\begin{array}{c|c|cc}
	& \dim L & \dim J & \dim (L_0 \cap L'_0) \\
	\hline
	G_2 & 14 & 1 & 2 \\
	D_4 & 28 & 3 & 4 \\
	F_4 & 52 & 6 & 10 \\
	E_6 & 78 & 9 & 18 \\
	E_7 & 133 & 15 & 37 \\
	E_8 & 248 & 27 & 80
\end{array}
\]
\caption{Dimensions of the pieces for the $G_2$-grading}
\label{table:dimG2}
\end{table}

In \cref{ta:exc G2}, we rely on \cite[Theorem 2.5.22]{Muehlherr2022} and we adopt its notation $\mathcal{H}(C, k)$ for the Freudenthal algebra constructed from the composition algebra $C$ over $k$. (See also \cite[Notation 4.1.72]{Muehlherr2022}.)

\begin{table}[p]
\renewcommand{\arraystretch}{2.3}
\hspace*{-4ex}
\begin{center}
\begin{tabular}{|c|c|c|c|}
	\hline
	{Tits index} & {\begin{minipage}{4.5ex}rel.\\type\end{minipage}} & $\dim J$ & Cubic norm structure (up to isotopy) \\
	\hline
	\hline
	\begin{tikzpicture}[line width=.8pt, scale=.6, baseline={(0,-.2)}]
	    \draw (0,0) -- (1,0);
	    \draw (0,.1) -- (1,.1);
	    \draw (0,-.1) -- (1,-.1);
	    \diagnode{(0,0)}
	    \diagnode{(1,0)}
	    \distorbit{(0,0)}
	    \distorbit{(1,0)}
		\draw[line width=.7pt] (.7,.3) -- (.4,0) -- (.7,-.3);
	\end{tikzpicture}
	& $G_2$ & $1$ & $J=k$
	\\
	\hline
	\begin{tikzpicture}[line width=.8pt, scale=.6, baseline={(0,-.2)}]
	    \draw (0,0) arc (180:90:.4) -- (1,.4);
	    \draw (0,0) arc (180:270:.4) -- (1,-.4);
	    \draw (0,0) -- (1,0);
	    \diagnode{(0,0)}
	    \diagnode{(1,0)}
	    \diagnode{(1,.4)}
	    \diagnode{(1,-.4)}
	    \distorbit{(0,0)}
	    \distlongorbit{1}
	\end{tikzpicture}
	& $G_2$ & $3$ & $J/k$ sep.\@~cubic field ext.
	\\
	\hline
	\begin{tikzpicture}[line width=.8pt, scale=.6, baseline={(0,-.2)}]
		\draw (0,0) -- (1,0);
		\draw (1,.06) -- (2,.06);
		\draw (1,-.06) -- (2,-.06);
		\draw (2,0) -- (3,0);
		\draw[line width=.7pt] (1.7,.3) -- (1.4,0) -- (1.7,-.3);
		\diagnode{(0,0)}
		\diagnode{(1,0)}
		\diagnode{(2,0)}
		\diagnode{(3,0)}
		\distorbit{(0,0)}
		\distorbit{(1,0)}
		\distorbit{(2,0)}
		\distorbit{(3,0)}
	\end{tikzpicture}
	& $F_4$ & $6$ & $\mathcal{H}(k, k)$
	\\
	\hline
	\begin{tikzpicture}[line width=.8pt, scale=.6, baseline={(0,.05)}]
	    \draw (0,0) -- (4,0);
	    \draw (2,0) -- (2,.8);
	    \diagnode{(0,0)}
	    \diagnode{(1,0)}
	    \diagnode{(2,0)}
	    \diagnode{(3,0)}
	    \diagnode{(4,0)}
	    \diagnode{(2,.8)}
	    \distorbit{(2,0)}
	    \distorbit{(2,.8)}
	\end{tikzpicture}
	& $G_2$ & $9$ & cubic division algebra
	\\
	\hline
	\begin{tikzpicture}[line width=.8pt, scale=.6, baseline={(0,-.2)}]
	    \draw (0,0) -- (1,0);
	    \draw (1,0) arc (180:90:.4) -- (3,.4);
	    \draw (1,0) arc (180:270:.4) -- (3,-.4);
	    \diagnode{(0,0)}
	    \diagnode{(1,0)}
	    \diagnode{(2,.4)} \diagnode{(2,-.4)}
	    \diagnode{(3,.4)} \diagnode{(3,-.4)}
	    \distorbit{(0,0)}
	    \distorbit{(1,0)}
	\end{tikzpicture}
	& $G_2$ & $9$ & \begin{minipage}{5.5cm}\begin{center} cubic division algebra over $E/k$ \\ with involution of the second kind \strut \end{center}\end{minipage}
	\\
	\hline
	\begin{tikzpicture}[line width=.8pt, scale=.6, baseline={(0,-.2)}]
	    \draw (0,0) -- (1,0);
	    \draw (1,0) arc (180:90:.4) -- (3,.4);
	    \draw (1,0) arc (180:270:.4) -- (3,-.4);
	    \diagnode{(0,0)}
	    \diagnode{(1,0)}
	    \diagnode{(2,.4)} \diagnode{(2,-.4)}
	    \diagnode{(3,.4)} \diagnode{(3,-.4)}
	    \distorbit{(0,0)}
	    \distorbit{(1,0)}
	    \distlongorbit{2}
	    \distlongorbit{3}
	\end{tikzpicture}
	& $F_4$ & $9$ & $\mathcal{H}(E, k)$, $E/k$ sep.\@~quadr.\@~field ext.
	\\
	\hline
	\begin{tikzpicture}[line width=.8pt, scale=.6, baseline={(0,.05)}]
	    \draw (0,0) -- (4,0);
	    \draw (2,0) -- (2,.8);
	    \diagnode{(0,0)}
	    \diagnode{(1,0)}
	    \diagnode{(2,0)}
	    \diagnode{(3,0)}
	    \diagnode{(4,0)}
	    \diagnode{(2,.8)}
	    \distorbit{(0,0)}
	    \distorbit{(1,0)}
	    \distorbit{(2,0)}
	    \distorbit{(3,0)}
	    \distorbit{(4,0)}
	    \distorbit{(2,.8)}
	\end{tikzpicture}
	& $E_6$ & $9$ & $M_3(k)$
	\\
	\hline
	\begin{tikzpicture}[line width=.8pt, scale=.6, baseline={(0,.05)}]
	    \draw (0,0) -- (5,0);
	    \draw (2,0) -- (2,.8);
	    \diagnode{(0,0)}
	    \diagnode{(1,0)}
	    \diagnode{(2,0)}
	    \diagnode{(3,0)}
	    \diagnode{(4,0)}
	    \diagnode{(5,0)}
	    \diagnode{(2,.8)}
	    \distorbit{(0,0)}
	    \distorbit{(1,0)}
	    \distorbit{(2,0)}
	    \distorbit{(4,0)}
	\end{tikzpicture}
	& $F_4$ & $15$ & $\mathcal{H}(\mathcal{Q}, k)$, $\mathcal{Q}$ quaternion division
	\\
	\hline
	\begin{tikzpicture}[line width=.8pt, scale=.6, baseline={(0,.05)}]
	    \draw (0,0) -- (5,0);
	    \draw (2,0) -- (2,.8);
	    \diagnode{(0,0)}
	    \diagnode{(1,0)}
	    \diagnode{(2,0)}
	    \diagnode{(3,0)}
	    \diagnode{(4,0)}
	    \diagnode{(5,0)}
	    \diagnode{(2,.8)}
	    \distorbit{(0,0)}
	    \distorbit{(1,0)}
	    \distorbit{(2,0)}
	    \distorbit{(3,0)}
	    \distorbit{(4,0)}
	    \distorbit{(5,0)}
	    \distorbit{(2,.8)}
	\end{tikzpicture}
	& $E_7$ & $15$ & $\mathcal{H}(\mathcal{Q}, k)$, $\mathcal{Q}$ quaternion split
	\\
	\hline
	\begin{tikzpicture}[line width=.8pt, scale=.6, baseline={(0,.05)}]
	    \draw (0,0) -- (6,0);
	    \draw (2,0) -- (2,.8);
	    \diagnode{(0,0)}
	    \diagnode{(1,0)}
	    \diagnode{(2,0)}
	    \diagnode{(3,0)}
	    \diagnode{(4,0)}
	    \diagnode{(5,0)}
	    \diagnode{(6,0)}
	    \diagnode{(2,.8)}
	    \distorbit{(5,0)}
	    \distorbit{(6,0)}
	\end{tikzpicture}
	& $G_2$ & $27$ & Albert division algebra
	\\
	\hline
	\begin{tikzpicture}[line width=.8pt, scale=.6, baseline={(0,.05)}]
	    \draw (0,0) -- (6,0);
	    \draw (2,0) -- (2,.8);
	    \diagnode{(0,0)}
	    \diagnode{(1,0)}
	    \diagnode{(2,0)}
	    \diagnode{(3,0)}
	    \diagnode{(4,0)}
	    \diagnode{(5,0)}
	    \diagnode{(6,0)}
	    \diagnode{(2,.8)}
	    \distorbit{(0,0)}
	    \distorbit{(4,0)}
	    \distorbit{(5,0)}
	    \distorbit{(6,0)}
	\end{tikzpicture}
	& $F_4$ & $27$ & $\mathcal{H}(\mathcal{O}, k)$, $\mathcal{O}$ octonion division
	\\
	\hline
	\begin{tikzpicture}[line width=.8pt, scale=.6, baseline={(0,.05)}]
	    \draw (0,0) -- (6,0);
	    \draw (2,0) -- (2,.8);
	    \diagnode{(0,0)}
	    \diagnode{(1,0)}
	    \diagnode{(2,0)}
	    \diagnode{(3,0)}
	    \diagnode{(4,0)}
	    \diagnode{(5,0)}
	    \diagnode{(6,0)}
	    \diagnode{(2,.8)}
	    \distorbit{(0,0)}
	    \distorbit{(1,0)}
	    \distorbit{(2,0)}
	    \distorbit{(3,0)}
	    \distorbit{(4,0)}
	    \distorbit{(5,0)}
	    \distorbit{(6,0)}
	    \distorbit{(2,.8)}
	\end{tikzpicture}
	& $E_8$ & $27$ & $\mathcal{H}(\mathcal{O}, k)$, $\mathcal{O}$ octonion split
	\\
	\hline
\end{tabular}
\renewcommand{\arraystretch}{1}
\newline
\caption{Exceptional Tits indices with $G_2$-graded Lie algebra}\label{ta:exc G2}
\end{center}
\vspace*{-2ex}
\end{table}

\begin{remark}
	The example of type $D_4$ in the second row of \cref{ta:exc G2} also occurs in non-triality forms, which we have not included in the list because they are not exceptional.
	These correspond to the Tits indices
	
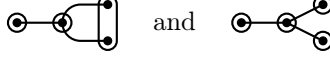
\[
	\begin{tikzpicture}[line width=.8pt, scale=.6, baseline={(0,-.1)}]
	    \draw (0,0) -- (1,0);
	    \draw (1,0) arc (180:90:.4) -- (2,.4);
	    \draw (1,0) arc (180:270:.4) -- (2,-.4);
	    \diagnode{(0,0)}
	    \diagnode{(1,0)}
	    \diagnode{(2,.4)}
	    \diagnode{(2,-.4)}
	    \distorbit{(0,0)}
	    \distorbit{(1,0)}
	    \distlongorbit{2}
	\end{tikzpicture}
	\quad \text{ and } \quad
	\begin{tikzpicture}[line width=.8pt, scale=.6, baseline={(0,-.1)}]
	    \draw (0,0) -- (1,0);
	    \draw (1,0) -- (1.8,.4);
	    \draw (1,0) -- (1.8,-.4);
	    \diagnode{(0,0)}
	    \diagnode{(1,0)}
	    \diagnode{(1.8,.4)}
	    \diagnode{(1.8,-.4)}
	    \distorbit{(0,0)}
	    \distorbit{(1,0)}
	    \distorbit{(1.8,.4)}
	    \distorbit{(1.8,-.4)}
	\end{tikzpicture}
	\]
	and have corresponding $3$-dimensional $J$ of the form $J = k \times E$, where $E/k$ is a separable quadratic field extension, for the first case,
	and of the form $J = k \times k \times k$, a split cubic étale extension, for the second case.
\end{remark}

\subsection{Quadrangular algebras}

In the case of quadrangular algebras, corresponding to the $BC_2$-grading as in \cref{fig:qa} on p.\@~\pageref{fig:qa}, we see that
\[ \dim L = 4 + 4 \dim X + 4 \dim V + \dim (L_0 \cap L'_0) . \]
We have listed the four possibilities in \cref{table:dimBC2}.
\begin{table}[p]
\[
\begin{array}{c|c|ccc}
	& \dim L & \dim V & \dim X & \dim (L_0 \cap L'_0) \\
	\hline
	F_4 & 52 & 4 & 5 & 12 \\
	E_6 & 78 & 6 & 8 & 18 \\
	E_7 & 133 & 8 & 16 & 33 \\
	E_8 & 248 & 12 & 32 & 68
\end{array}
\]
\caption{Dimensions of the pieces for the $BC_2$-grading}
\label{table:dimBC2}
\end{table}

By \cref{rem:QA conditions}, our quadrangular algebras are either generic or of split type $F_4$.
By \cite[Theorems 8.16 and 10.16]{Muehlherr2019}, these are either \emph{special} (i.e., not exceptional), or they are anisotropic, or they are isotopic to $\mathcal{Q}_2(C, k)$ for some octonion algebra $C$ or to $\mathcal{Q}_4(C, k)$ for some composition algebra~$C$.
(The definition of these quadrangular algebras can be found in \cite[Notations 4.12 and 4.14]{Muehlherr2019}.)
In other words, we get all cases from \cite[Table 1]{Muehlherr2019} except the last two rows.
We have essentially reproduced this table in \cref{ta:exc BC2}.

\begin{table}[p]
\renewcommand{\arraystretch}{2.3}
\hspace*{-4ex}
\begin{center}
\begin{tabular}{|c|c|c|c|c|}
	\hline
	{Tits index} & {\begin{minipage}{4.5ex}rel.\\type\end{minipage}} & $\dim V$ & $\dim X$ & Quadrangular algebra (up to isotopy) \\
	\hline
	\hline
	\begin{tikzpicture}[line width=.8pt, scale=.6, baseline={(0,-.2)}]
		\draw (0,0) -- (1,0);
		\draw (1,.06) -- (2,.06);
		\draw (1,-.06) -- (2,-.06);
		\draw (2,0) -- (3,0);
		\draw[line width=.7pt] (1.7,.3) -- (1.4,0) -- (1.7,-.3);
		\diagnode{(0,0)}
		\diagnode{(1,0)}
		\diagnode{(2,0)}
		\diagnode{(3,0)}
		\distorbit{(0,0)}
		\distorbit{(1,0)}
		\distorbit{(2,0)}
		\distorbit{(3,0)}
	\end{tikzpicture}
	& $F_4$ & $4$ & $5$ & $\mathcal{Q}_4(k, k)$
	\\
	\hline
	\begin{tikzpicture}[line width=.8pt, scale=.6, baseline={(0,-.2)}]
	    \draw (0,0) -- (1,0);
	    \draw (1,0) arc (180:90:.4) -- (3,.4);
	    \draw (1,0) arc (180:270:.4) -- (3,-.4);
	    \diagnode{(0,0)}
	    \diagnode{(1,0)}
	    \diagnode{(2,.4)} \diagnode{(2,-.4)}
	    \diagnode{(3,.4)} \diagnode{(3,-.4)}
	    \distorbit{(0,0)}
	    \distlongorbit{3}
	\end{tikzpicture}
	& $BC_2$ & $6$ & $8$ & anisotropic of type $E_6$
	\\
	\hline
	\begin{tikzpicture}[line width=.8pt, scale=.6, baseline={(0,-.2)}]
	    \draw (0,0) -- (1,0);
	    \draw (1,0) arc (180:90:.4) -- (3,.4);
	    \draw (1,0) arc (180:270:.4) -- (3,-.4);
	    \diagnode{(0,0)}
	    \diagnode{(1,0)}
	    \diagnode{(2,.4)} \diagnode{(2,-.4)}
	    \diagnode{(3,.4)} \diagnode{(3,-.4)}
	    \distorbit{(0,0)}
	    \distorbit{(1,0)}
	    \distlongorbit{2}
	    \distlongorbit{3}
	\end{tikzpicture}
	& $F_4$ & $6$ & $8$ & $\mathcal{Q}_4(E, k)$, $E/k$ sep.\@ quadr.\@ field ext.
	\\
	\hline
	\begin{tikzpicture}[line width=.8pt, scale=.6, baseline={(0,.05)}]
	    \draw (0,0) -- (4,0);
	    \draw (2,0) -- (2,.8);
	    \diagnode{(0,0)}
	    \diagnode{(1,0)}
	    \diagnode{(2,0)}
	    \diagnode{(3,0)}
	    \diagnode{(4,0)}
	    \diagnode{(2,.8)}
	    \distorbit{(0,0)}
	    \distorbit{(1,0)}
	    \distorbit{(2,0)}
	    \distorbit{(3,0)}
	    \distorbit{(4,0)}
	    \distorbit{(2,.8)}
	\end{tikzpicture}
	& $E_6$ & $6$ & $8$ & $\mathcal{Q}_4(k \times k, k)$
	\\
	\hline
	\begin{tikzpicture}[line width=.8pt, scale=.6, baseline={(0,.05)}]
	    \draw (0,0) -- (5,0);
	    \draw (2,0) -- (2,.8);
	    \diagnode{(0,0)}
	    \diagnode{(1,0)}
	    \diagnode{(2,0)}
	    \diagnode{(3,0)}
	    \diagnode{(4,0)}
	    \diagnode{(5,0)}
	    \diagnode{(2,.8)}
	    \distorbit{(0,0)}
	    \distorbit{(4,0)}
	\end{tikzpicture}
	& $BC_2$ & $8$ & $16$ & anisotropic of type $E_7$
	\\
	\hline
	\begin{tikzpicture}[line width=.8pt, scale=.6, baseline={(0,.05)}]
	    \draw (0,0) -- (5,0);
	    \draw (2,0) -- (2,.8);
	    \diagnode{(0,0)}
	    \diagnode{(1,0)}
	    \diagnode{(2,0)}
	    \diagnode{(3,0)}
	    \diagnode{(4,0)}
	    \diagnode{(5,0)}
	    \diagnode{(2,.8)}
	    \distorbit{(0,0)}
	    \distorbit{(4,0)}
	    \distorbit{(5,0)}
	\end{tikzpicture}
	& $C_3$ & $8$ & $16$ & $\mathcal{Q}_2(\mathcal{O}, k)$, $\mathcal{O}$ octonion division
	\\
	\hline
	\begin{tikzpicture}[line width=.8pt, scale=.6, baseline={(0,.05)}]
	    \draw (0,0) -- (5,0);
	    \draw (2,0) -- (2,.8);
	    \diagnode{(0,0)}
	    \diagnode{(1,0)}
	    \diagnode{(2,0)}
	    \diagnode{(3,0)}
	    \diagnode{(4,0)}
	    \diagnode{(5,0)}
	    \diagnode{(2,.8)}
	    \distorbit{(0,0)}
	    \distorbit{(1,0)}
	    \distorbit{(2,0)}
	    \distorbit{(4,0)}
	\end{tikzpicture}
	& $F_4$ & $8$ & $16$ & $\mathcal{Q}_4(Q, k)$, $Q$ quaternion division
	\\
	\hline
	\begin{tikzpicture}[line width=.8pt, scale=.6, baseline={(0,.05)}]
	    \draw (0,0) -- (5,0);
	    \draw (2,0) -- (2,.8);
	    \diagnode{(0,0)}
	    \diagnode{(1,0)}
	    \diagnode{(2,0)}
	    \diagnode{(3,0)}
	    \diagnode{(4,0)}
	    \diagnode{(5,0)}
	    \diagnode{(2,.8)}
	    \distorbit{(0,0)}
	    \distorbit{(1,0)}
	    \distorbit{(2,0)}
	    \distorbit{(3,0)}
	    \distorbit{(4,0)}
	    \distorbit{(5,0)}
	    \distorbit{(2,.8)}
	\end{tikzpicture}
	& $E_7$ & $8$ & $16$ & $\mathcal{Q}_4(Q, k)$, $Q$ quaternion split
	\\
	\hline
	\begin{tikzpicture}[line width=.8pt, scale=.6, baseline={(0,.05)}]
	    \draw (0,0) -- (6,0);
	    \draw (2,0) -- (2,.8);
	    \diagnode{(0,0)}
	    \diagnode{(1,0)}
	    \diagnode{(2,0)}
	    \diagnode{(3,0)}
	    \diagnode{(4,0)}
	    \diagnode{(5,0)}
	    \diagnode{(6,0)}
	    \diagnode{(2,.8)}
	    \distorbit{(0,0)}
	    \distorbit{(6,0)}
	\end{tikzpicture}
	& $BC_2$ & $12$ & $32$ & anisotropic of type $E_8$
	\\
	\hline
	\begin{tikzpicture}[line width=.8pt, scale=.6, baseline={(0,.05)}]
	    \draw (0,0) -- (6,0);
	    \draw (2,0) -- (2,.8);
	    \diagnode{(0,0)}
	    \diagnode{(1,0)}
	    \diagnode{(2,0)}
	    \diagnode{(3,0)}
	    \diagnode{(4,0)}
	    \diagnode{(5,0)}
	    \diagnode{(6,0)}
	    \diagnode{(2,.8)}
	    \distorbit{(0,0)}
	    \distorbit{(4,0)}
	    \distorbit{(5,0)}
	    \distorbit{(6,0)}
	\end{tikzpicture}
	& $F_4$ & $12$ & $32$ & $\mathcal{Q}_4(\mathcal{O}, k)$, $\mathcal{O}$ octonion division
	\\
	\hline
	\begin{tikzpicture}[line width=.8pt, scale=.6, baseline={(0,.05)}]
	    \draw (0,0) -- (6,0);
	    \draw (2,0) -- (2,.8);
	    \diagnode{(0,0)}
	    \diagnode{(1,0)}
	    \diagnode{(2,0)}
	    \diagnode{(3,0)}
	    \diagnode{(4,0)}
	    \diagnode{(5,0)}
	    \diagnode{(6,0)}
	    \diagnode{(2,.8)}
	    \distorbit{(0,0)}
	    \distorbit{(1,0)}
	    \distorbit{(2,0)}
	    \distorbit{(3,0)}
	    \distorbit{(4,0)}
	    \distorbit{(5,0)}
	    \distorbit{(6,0)}
	    \distorbit{(2,.8)}
	\end{tikzpicture}
	& $E_8$ & $12$ & $32$ & $\mathcal{Q}_4(\mathcal{O}, k)$, $\mathcal{O}$ octonion split
	\\
	\hline
\end{tabular}
\renewcommand{\arraystretch}{1}
\newline
\caption{Exceptional Tits indices with $BC_2$-graded Lie algebra}\label{ta:exc BC2}
\end{center}
\vspace*{-2ex}
\end{table}

\begin{remark}\label{rem:both}
	Comparing \cref{ta:exc G2,ta:exc BC2}, we see that the Tits indices that occur in \emph{both} tables are the following:
	\begin{enumerate}
		\item The split forms of type $F_4$, $E_6$, $E_7$ and $E_8$;
		\item The forms of \emph{relative} type $F_4$ (and of absolute type $F_4$, $E_6$, $E_7$ and $E_8$).
	\end{enumerate}
	In other words, these seven types of Lie algebras are precisely the ones admitting both a $G_2$-grading and a $BC_2$-grading, and hence they are parametrized both by a cubic norm structure and by a quadrangular algebra.
	As we have alluded to in the introduction, it is an interesting question to investigate the connection between these two different algebraic structures for each of these seven types.
\end{remark}

\afterpage{\clearpage}

\bibliographystyle{alpha}
\bibliography{doctoraat}

\begin{thebibliography}{KMRT98}

\bibitem[ABG02]{ABG02}
Bruce Allison, Georgia Benkart, and Yun Gao.
\newblock Lie algebras graded by the root systems {$BC_r,\ r\ge2$}.
\newblock {\em Mem. Amer. Math. Soc.}, 158(751):x+158, 2002.

\bibitem[BDM13]{Boelaert2013}
Lien Boelaert and Tom De~Medts.
\newblock Exceptional {M}oufang quadrangles and structurable algebras.
\newblock {\em Proceedings of the London Mathematical Society. Third Series},
  107(3):590--626, 2013.

\bibitem[BDM15]{Boelaert2015}
Lien Boelaert and Tom De~Medts.
\newblock A new construction of {M}oufang quadrangles of type {$E_6$}, {$E_7$}
  and {$E_8$}.
\newblock {\em Transactions of the American Mathematical Society},
  367(5):3447--3480, 2015.

\bibitem[BDMS19]{Boelaert2019}
Lien Boelaert, Tom De~Medts, and Anastasia Stavrova.
\newblock Moufang sets and structurable division algebras.
\newblock {\em Mem. Amer. Math. Soc.}, 259(1245):v+90, 2019.

\bibitem[Ben74]{Benkart1974}
Georgia Benkart.
\newblock {\em Inner ideals and the structure of Lie algebras}.
\newblock ProQuest LLC, Ann Arbor, MI, 1974.
\newblock Thesis (Ph.D.)--Yale University.

\bibitem[BZ96]{BZ96}
Georgia Benkart and Efim Zelmanov.
\newblock Lie algebras graded by finite root systems and intersection matrix
  algebras.
\newblock {\em Invent. Math.}, 126(1):1--45, 1996.

\bibitem[CF18]{Cuypers2018}
Hans Cuypers and Yael Fleischmann.
\newblock A geometric characterization of the classical {L}ie algebras.
\newblock {\em Journal of Algebra}, 502:1--23, 2018.

\bibitem[CF23]{Cuypers2023}
Hans Cuypers and Yael Fleischmann.
\newblock A geometric characterization of the symplectic {L}ie algebra.
\newblock {\em Innov. Incidence Geom.}, 20(2-3):223--245, 2023.

\bibitem[Che89]{Chernousov1989}
V.~I. Chernousov.
\newblock The {H}asse principle for groups of type {$E_8$}.
\newblock {\em Dokl. Akad. Nauk SSSR}, 306(5):1059--1063, 1989.

\bibitem[CI06]{Cohen2006}
Arjeh~M. Cohen and G\'{a}bor Ivanyos.
\newblock Root filtration spaces from {L}ie algebras and abstract root groups.
\newblock {\em Journal of Algebra}, 300(2):433--454, 2006.

\bibitem[CI07]{Cohen2007}
Arjeh~M. Cohen and G\'{a}bor Ivanyos.
\newblock Root shadow spaces.
\newblock {\em European Journal of Combinatorics}, 28(5):1419--1441, 2007.

\bibitem[CIR08]{Cohen2008}
Arjeh~M. Cohen, G\'{a}bor Ivanyos, and Dan Roozemond.
\newblock Simple {L}ie algebras having extremal elements.
\newblock {\em Koninklijke Nederlandse Akademie van Wetenschappen. Indagationes
  Mathematicae. New Series}, 19(2):177--188, 2008.

\bibitem[CM21]{Cuypers2021}
Hans Cuypers and Jeroen Meulewaeter.
\newblock Extremal elements in {L}ie algebras, buildings and structurable
  algebras.
\newblock {\em Journal of Algebra}, 580:1--42, 2021.

\bibitem[Coh12]{Cohen2012}
Arjeh~M. Cohen.
\newblock The geometry of extremal elements in a {L}ie algebra.
\newblock In {\em Buildings, finite geometries and groups}, volume~10 of {\em
  Springer Proc. Math.}, pages 15--35. Springer, New York, 2012.

\bibitem[CRS15]{Cuypers2015}
Hans Cuypers, Kieran Roberts, and Sergey Shpectorov.
\newblock Recovering the {L}ie algebra from its extremal geometry.
\newblock {\em Journal of Algebra}, 441:196--215, 2015.

\bibitem[CSUW01]{Cohen2001}
Arjeh~M. Cohen, Anja Steinbach, Rosane Ushirobira, and David Wales.
\newblock Lie algebras generated by extremal elements.
\newblock {\em Journal of Algebra}, 236(1):122--154, 2001.

\bibitem[Ditp08]{Draisma2008}
Jan Draisma and Jos in~'t panhuis.
\newblock Constructing simply laced {L}ie algebras from extremal elements.
\newblock {\em Algebra Number Theory}, 2(5):551--572, 2008.

\bibitem[DMM24]{DeMedts2024}
Tom De~Medts and Jeroen Meulewaeter.
\newblock Inner ideals and structurable algebras: {M}oufang sets, triangles and
  hexagons.
\newblock {\em Israel J. Math.}, 259(1):33--88, 2024.

\bibitem[Fau73]{Faulkner1973}
John~R. Faulkner.
\newblock On the geometry of inner ideals.
\newblock {\em Journal of Algebra}, 26:1--9, 1973.

\bibitem[Fau01]{Faulkner2001}
John~R. Faulkner.
\newblock Generalized quadrangles and cubic forms.
\newblock {\em Comm. Algebra}, 29(10):4641--4653, 2001.

\bibitem[FL16]{Lopez2016}
Antonio Fern\'andez~L{\'o}pez.
\newblock A {J}ordan approach to finitary {L}ie algebras.
\newblock {\em Q. J. Math.}, 67(4):565--571, 2016.

\bibitem[FL19]{FernandezLopez2019}
Antonio Fern\'{a}ndez~L\'{o}pez.
\newblock {\em Jordan structures in {L}ie algebras}, volume 240 of {\em
  Mathematical Surveys and Monographs}.
\newblock American Mathematical Society, Providence, RI, 2019.

\bibitem[Fre54]{Freudenthal1954}
Hans Freudenthal.
\newblock Beziehungen der {$E_7$} und {$E_8$} zur {O}ktavenebene. {I}.
\newblock {\em Indag. Math.}, 16:218--230, 1954.
\newblock Nederl. Akad. Wetensch. Proc. Ser. A {\bf 57}.

\bibitem[GS17]{GS17}
Philippe Gille and Tam\'{a}s Szamuely.
\newblock {\em Central simple algebras and {G}alois cohomology}, volume 165 of
  {\em Cambridge Studies in Advanced Mathematics}.
\newblock Cambridge University Press, Cambridge, 2017.
\newblock Second edition of [ MR2266528].

\bibitem[itpPR09]{intpanhuis2009}
Jos in~'t panhuis, Erik Postma, and Dan Roozemond.
\newblock Extremal presentations for classical {L}ie algebras.
\newblock {\em J. Algebra}, 322(2):295--326, 2009.

\bibitem[KMRT98]{Knus1998}
Max-Albert Knus, Alexander Merkurjev, Markus Rost, and Jean-Pierre Tignol.
\newblock {\em The book of involutions}, volume~44 of {\em American
  Mathematical Society Colloquium Publications}.
\newblock American Mathematical Society, Providence, RI, 1998.
\newblock With a preface in French by J. Tits.

\bibitem[McC69]{McCrimmon1969a}
Kevin McCrimmon.
\newblock The {F}reudenthal-{S}pringer-{T}its constructions of exceptional
  {J}ordan algebras.
\newblock {\em Transactions of the American Mathematical Society},
  139:495--510, 1969.

\bibitem[{Meu}21]{MeulewaeterPhD}
{Meulewaeter, Jeroen}.
\newblock {\em {Bridging exceptional geometries and algebras using inner ideals
  of Lie algebras}}.
\newblock PhD thesis, {Ghent University}, {2021}.

\bibitem[MW19]{Muehlherr2019}
Bernhard M\"{u}hlherr and Richard~M. Weiss.
\newblock Isotropic quadrangular algebras.
\newblock {\em Journal of the Mathematical Society of Japan}, 71(4):1321--1380,
  2019.

\bibitem[MW22]{Muehlherr2022}
Bernhard M\"uhlherr and Richard~M. Weiss.
\newblock Tits polygons.
\newblock {\em Mem. Amer. Math. Soc.}, 275(1352):xi+114, 2022.
\newblock With an appendix by Holger P. Petersson.

\bibitem[PR84]{Petersson1984}
Holger~P. Petersson and Michel~L. Racine.
\newblock Springer forms and the first {T}its construction of exceptional
  {J}ordan division algebras.
\newblock {\em Manuscripta Math.}, 45(3):249--272, 1984.

\bibitem[PR86a]{Petersson1986a}
Holger~P. Petersson and Michel~L. Racine.
\newblock Classification of algebras arising from the {T}its process.
\newblock {\em Journal of Algebra}, 98(1):244--279, 1986.

\bibitem[PR86b]{Petersson1986}
Holger~P. Petersson and Michel~L. Racine.
\newblock Jordan algebras of degree {$3$} and the {T}its process.
\newblock {\em Journal of Algebra}, 98(1):211--243, 1986.

\bibitem[PR94]{Platonov1994}
Vladimir Platonov and Andrei Rapinchuk.
\newblock {\em Algebraic groups and number theory}, volume 139 of {\em Pure and
  Applied Mathematics}.
\newblock Academic Press, Inc., Boston, MA, 1994.
\newblock Translated from the 1991 Russian original by Rachel Rowen.

\bibitem[Pre86]{Premet1986}
A.~A. Premet.
\newblock Lie algebras without strong degeneration.
\newblock {\em Mat. Sb. (N.S.)}, 129(171)(1):140--153, 1986.

\bibitem[Roo11]{Roozemond2011}
Dan Roozemond.
\newblock On {L}ie algebras generated by few extremal elements.
\newblock {\em J. Algebra}, 348:462--476, 2011.

\bibitem[Sel76]{Seligman1976}
George~B. Seligman.
\newblock {\em Rational methods in {L}ie algebras}, volume Vol. 17 of {\em
  Lecture Notes in Pure and Applied Mathematics}.
\newblock Marcel Dekker, Inc., New York-Basel, 1976.

\bibitem[Sme25]{Smet2025}
Michiel Smet.
\newblock Cubic norm pairs and hermitian cubic norm structures, 2025.

\bibitem[Spr62]{Springer1962}
T.~A. Springer.
\newblock Characterization of a class of cubic forms.
\newblock {\em Indag. Math.}, 24:259--265, 1962.
\newblock Nederl. Akad. Wetensch. Proc. Ser. A {\bf 65}.

\bibitem[Sta22]{Stavrova2017}
Anastasia Stavrova.
\newblock 5-graded simple {L}ie algebras, structurable algebras, and {K}antor
  pairs.
\newblock {\em J. Algebra}, 611:561--588, 2022.

\bibitem[TW02]{Tits2002}
Jacques Tits and Richard~M. Weiss.
\newblock {\em Moufang polygons}.
\newblock Springer Monographs in Mathematics. Springer-Verlag, Berlin, 2002.

\bibitem[Wei06]{Weiss2006}
Richard~M. Weiss.
\newblock {\em Quadrangular algebras}, volume~46 of {\em Mathematical Notes}.
\newblock Princeton University Press, Princeton, NJ, 2006.

\bibitem[ZK90]{Zelcprimemanov1990}
E.~I. Zelmanov and A.~I. Kostrikin.
\newblock A theorem on sandwich algebras.
\newblock volume 183, pages 106--111, 225. 1990.
\newblock Translated in Proc. Steklov Inst. Math. {{\bf{1}}991}, no. 4,
  121--126, Galois theory, rings, algebraic groups and their applications
  (Russian).

\end{thebibliography}

\end{document}